\documentclass[8pt]{article}
\usepackage{ragged2e}
\usepackage{csquotes}
\usepackage{varioref}
\usepackage{float}
\usepackage[thinlines]{easytable}
\usepackage{caption}
\usepackage{csquotes}
\usepackage{amsfonts,amssymb,amsmath,exscale,relsize,cite,enumitem,amsthm,multirow,physics}
\usepackage[numbers]{natbib}
\usepackage[usenames]{color}
\usepackage[colorlinks,citecolor=blue,urlcolor=blue,filecolor=blue,backref=page]{hyperref}
\usepackage{setspace}
\usepackage{appendix}
\usepackage{graphicx}%
\providecommand{\U}[1]{\protect\rule{.1in}{.1in}}
\newtheorem{theorem}{Theorem}

\newtheorem{corollary}{Corollary}

\newtheorem{lemma}{Lemma}

\newtheorem{proposition}{Proposition}
\newtheorem{remark}{Remark}

\newtheorem{solution*}{Solution}

\def\U{\mathcal{U}}

\def\1{\mathbb{1}}

\newcommand\simiid{\stackrel{iid}{\sim}}
\newcommand\simind{\stackrel{ind}{\sim}}
\title{Consistent Group Selection using Global-local Shrinkage Priors in High Dimensions}
\author{Sayantan Paul, Prasenjit Ghosh and Arijit Chakrabarti}

\begin{document}

\maketitle

\begin{center}
    \textbf{Abstract}
\end{center}
Consider a high-dimensional normal linear regression model when the candidate regressors are inherently grouped. Our interest lies in the selection of grouped variables and the estimation of model parameters in a sparse asymptotic regime. We modeled the grouped regression coefficients using a broad class of \enquote{global-local} shrinkage priors, which can also be seen as a generalization of the standard \textit{g-prior} with a shrinkage parameter. The global shrinkage parameter is treated either as a tuning parameter or in an empirical Bayesian or full Bayesian way. We consider a group selection rule, namely the Half-Thresholding rule, and propose an estimator using this rule. Our methods enjoy the oracle property asymptotically in that they achieve variable selection consistency and optimal rate of estimation under a block orthogonal design. These are the first theoretical results of their kind using such priors in this context. In our simulation study, our proposed rules perform favorably with many existing methods. 
\section{Introduction}
Selecting the relevant predictors for a regression model is a classic problem in statistics. Potential
predictors or regressors are often inherently linked, forming clusters or groups. For instance, gene expression data may
have among the potential regressors groups of genes controlling similar phenotypical traits. On the other hand, regression models for stock market data may have groups of stocks from the same sector. The grouping structure
is also seen, for example, in multifactor ANOVA and nonparametric additive models. See, for instance, the discussions in Yuan and Lin (2006) \cite{yuan2006model}, Yang and Narisetty (2010) \cite{yang2020consistent}, Wang and Leng (2008) \cite{wang2008note}, and Wei and Huang (2010) \cite{wei2010consistent}, in this connection. Depending on the problem being addressed, the individual variables within a group can vary significantly. In cases where the focus is not on examining individual variables, the primary objective shifts to identifying only the important groups. As noted by Huang et al. (2012) \cite{huang2012selective}, when a continuous factor is represented through a set of basis functions, the individual variables are artificially constructed. Thus, rather than selecting significant individual members, determining which groups are significant overall is the priority. Another case of group selection can be found in the seemingly unrelated regressions (SUR) model proposed by Zellner (1962) \cite{zellner1962efficient} or in the multitask learning model within machine learning put forth by Caruana (1997) \cite{caruana1997multitask}; and further discussed by Argyriou, Evgeniou, and Pontil (2008) \cite{argyriou2008convex}. These models operate under the assumption that specific variables have similar roles across different tasks, leading to their selection or exclusion as a group.

Given these examples, it becomes crucial to establish rules that effectively identify relevant groups of regressors while disregarding the irrelevant ones. Consequently, variable selection essentially transforms into a problem of group selection. This issue has garnered significant attention from researchers over the past couple of decades.

This work focuses on group selection and estimating group regression coefficients within a linear regression framework. We examine a linear regression model containing $G$ groups of potential regressors or predictors defined as
\begin{equation} \label{eq:1.1}
\mathbf{y}=\mathbf{X}\boldsymbol{\beta}+\boldsymbol{\epsilon}=\sum_{g=1}^{G}\mathbf{X}_g\boldsymbol{\beta}_g+\boldsymbol{\epsilon}.\tag{1.1}
\end{equation}
Here, $\mathbf{y}$ is an $n \times 1$ vector of responses, $\mathbf{X}_g$ is an $n \times m_g$ design matrix and $\boldsymbol{\beta}_g$ is an $m_g \times 1$ vector of unknown regression coefficients for the $g^{th}$ group, where $g \in \{1, \ldots, G\}$ and $\sum_{g=1}^G m_g = p$. We further assume that the vector of unobserved residuals $\boldsymbol{\epsilon}$ has a $N(\mathbf{0},\sigma^2I_n)$ distribution. Our interest is in a sparse asymptotic setting where $G\equiv G_n$ increases at the same rate as the sample size $n$ and the number of \textit{active} groups $G_{A_{n}}$ grows at a slower rate than $G_n$, an active group being one with a vector of true non-zero group regression coefficients. In this work, we propose some Bayesian and empirical Bayesian methods for selecting and estimating the active group coefficients using a broad class of hierarchical priors on the grouped regression coefficients. These priors may be considered as \textit{global-local} mixtures (in the group selection problem) of the popular g-prior of Zellner (1986) \cite{zellner1986assessing}. 
We now briefly review the most relevant existing works in this field before motivating our study. Our discussion will focus on two broad classes of methods for variable selection (grouped variable selection) and estimation, with special emphasis on sparse problems.\\
\hspace*{0.5cm} The first class consists of certain examples of penalized likelihood methods where the log-likelihood in a model (specifying a subset of predictors to be included in the regression) is penalized by some measure of the magnitude of its corresponding regression coefficients. This makes full sense when the proportions of regressors with non-null (or sufficiently large) effects are known to be small. These methods estimate $\boldsymbol{\beta}$ by minimizing (in the ungrouped regression problem) an objective function of the following form
\begin{equation*}
S(\boldsymbol{\beta})=(\mathbf{y}-\mathbf{X}\boldsymbol{\beta})^{\mathrm{T}}(\mathbf{y}-\mathbf{X}\boldsymbol{\beta})+ \sum_{j=1}^{p} p_{\lambda_j}(|\beta_j|),
\end{equation*}
where $p_{\lambda_j}(|\beta_j|)$ is an appropriately chosen penalty function, and $\lambda_j > 0$ is the penalty parameter. Notably, by choosing $\lambda_j=\lambda$ for all $j$, $p_{\lambda}(|\beta_j|)=\lambda \beta^2_j$ corresponds to the Ridge regression of Hoerl and Kennard (1970) \cite{hoerl1970ridge}, while $p_{\lambda}(|\beta_j|)= \lambda |\beta_j|$ corresponds to the LASSO estimator of Tibshirani (1996) \cite{tibshirani1996regression}, the latter being very popular due to its ability to perform both estimation and variable selection simultaneously (by estimating many components of $\boldsymbol{\beta}$ as exactly zero). We also mention the SCAD of Fan and Li (2001) \cite{fan2001variable}, the elastic net of Zou and Hastie (2006) \cite{zou2006adaptive}), the MCP of Zhang (2009) \cite{zhang2010nearly} and the adaptive LASSO. The \textit{group} LASSO of Yuan and Lin (2006) \cite{yuan2006model} mimics LASSO for use in the group problem (\ref{eq:1.1}) and is defined as the minimizer of the following objective function:
\begin{equation*}
\text{S}(\boldsymbol{\beta})=\frac{1}{2}\left(\textbf{y}-\sum_{g=1}^{G}\mathbf{X}_g \boldsymbol{\beta}_g\right)^{\mathrm{T}}\left(\textbf{y}-\sum_{g=1}^{G}\mathbf{X}_g\boldsymbol{\beta}_g\right)+\lambda \sum_{g=1}^{G} \sqrt{m_g} \sqrt{{\boldsymbol{\beta}_g}^{\mathrm{T}}{\boldsymbol{\beta}_g}} \hspace{0.2cm}.
\end{equation*}
The adaptive group LASSO of Wang and Leng (2008) \cite{wang2008note} improves the performance of the group LASSO by using separate data-based penalties for each group similar to the adaptive LASSO.

The second class consists of Bayesian methods obtained from two key approaches. In the first approach, a Bayesian starts by assigning a prior probability distribution to the model space, namely a collection of models, where, under each model, a subset of the regression coefficients ($\beta_i$'s) are considered small (or exactly zero). In contrast, the remaining $\beta_i$'s are considered large (or nonzero). Under each model, a prior quantifying the uncertainty about the parameters (regression coefficients) is then specified. The model with the highest posterior probability (the Bayes rule with respect to 0-1 loss) is typically the model of choice. Variable selection is carried out using the model thus chosen by only keeping those variables/predictors in the linear regression whose regression coefficients are supposed to be large (or nonzero) according to this model. See, for instance, George and McCullough (1993) \cite{george1993variable}, George and McCullough (1997) \cite{george1997approaches}, Scott and Berger (2010) \cite{scott2010bayes}, Bayarri et al. (2012) \cite{bayarri2012criteria}, Maruyama et al. (2011) \cite{maruyama2011fully}, Mukhopadhyay et al. (2015) \cite{mukhopadhyay2015consistency}, Liang at. al. (2008) \cite{liang2008mixtures}, to name a few, in this connection. Two popular choices of priors on the regression coefficients in the above references are (i) the g-prior of Zellner (1986) \cite{zellner1986assessing} (or variants of it) on the nonzero regression coefficients under a given model and (ii) independent spike and slab priors on the regression coefficients. In the second approach, Bayesian directly specifies a prior distribution on the parameters of the full model, namely the model with all the possible parameters/regressors. Model/variable selection is performed by appropriately using the posterior distribution of the parameter vector. See, e.g., Scott et. al. (2012) \cite{polson2012local}, Datta and Ghosh (2013) \cite{datta2013asymptotic}, Ghosh et al. (2016) \cite{ghosh2016asymptotic}, Ghosh and Chakrabarti (2017) \cite{ghosh2017asymptotic}, Tang et. al. (2018) \cite{tang2018bayesian}, Li and Lin (2010) \cite{li2010bayesian}, Bhattacharya et al. (2015) \cite{bhattacharya2015dirichlet} for examples of this kind. The priors used therein are the continuous one-group priors, which will be discussed in detail later. Both these Bayesian approaches can be adapted to the problem of grouped variable selection, and our work is aligned with this second approach. We now briefly describe two-group and one-group priors on the regression coefficients.

A natural Bayesian approach to modeling parameters in sparse regression problems involves the use of two-group spike-and-slab priors for individual regression coefficients. In this framework, each regression coefficient is either degenerate at zero (or highly concentrated around zero) with a high probability, denoted as \( \nu \in (0,1) \). This probability \( \nu \) can depend on the number of predictors \( p \equiv p(n) \) or follow an absolutely continuous heavy-tailed distribution with a small probability of \( 1 - \nu \).

The prior distribution is thus a mixture that combines a degenerate probability at zero (or a distribution that is highly concentrated around zero, such as a normal distribution with a small variance) - referred to as the \enquote{spike} part and a heavy-tailed, nondegenerate, absolutely continuous part (for example, a normal distribution with a large variance) - known as the \enquote{slab} part. 

This modeling is achieved by associating a latent random vector \( (\gamma_1, \ldots, \gamma_p) \) with \( (\beta_1, \ldots, \beta_p) \). The \( \gamma_i \) are independent and identically distributed random variables, where \( \gamma_i = 0 \) with probability \( 1 - \nu \) and \( \gamma_i = 1 \) with probability \( \nu \). Consequently, if \( \gamma_i = 0 \), then \( \beta_i \) follows the spike distribution; if \( \gamma_i = 1 \), then \( \beta_i \) follows the slab distribution. 

Several variations of spike-and-slab priors have been proposed in the literature, including those by Mitchell and Beauchamp (1988) \cite{mitchell1988bayesian}, George and McCulloch (1993) \cite{george1993variable}, Geweke (1996) \cite{geweke1996variable}, and Rockova and George (2018) \cite{rovckova2018spike}, among others.


Over the past 15 years, the literature has presented various proposals for modeling unknown parameters in sparse situations through hierarchical one-group ``shrinkage" priors, which can be formulated as scale mixtures of normals. These priors require less computational effort than the two-group model, making them particularly advantageous in high-dimensional problems and complex parametric frameworks.

One of the key features of these priors is that they assign a high probability near the origin while still allowing for non-trivial probabilities for larger coefficients, which helps induce sparsity while accommodating larger signals simultaneously. When chosen carefully, they can effectively replicate the core idea of two-group priors. Most of these priors incorporate two types of shrinkage parameters: the \enquote{global shrinkage parameter} and the \enquote{local shrinkage parameters}. The global shrinkage parameter is designed to induce overall sparsity, while the local shrinkage parameters are meant to accommodate large signals.

The application of one-group shrinkage priors to various sparse problems, as well as the exploration of their theoretical properties, has been the focus of active research for some time. For example, Table 1 of Tang et al. (2018) \cite{tang2018bayesian} summarizes different global-local shrinkage priors that have been studied in the literature. This class of priors contains horseshoe introduced by Carvalho et al. (2009 \cite{carvalho2009handling}), the Laplace prior (\cite{park2008bayesian}), the normal-exponential-gamma prior (\cite{griffin2005alternative}), three-parameter beta normal priors (\cite{armagan2011generalized}), generalized double Pareto priors (\cite{armagan2013generalized}) and the Dirichlet-laplace prior (\cite{bhattacharya2015dirichlet}), among others.

These priors were initially developed to model the mean vector \(\boldsymbol{\beta}\) in the canonical normal means problem, specifically using \(p=n\), \(X=I_n\), and \(m_g=1\) for \(g=1,\ldots,n\) in \eqref{eq:1.1} and then applying (in the canonical ungrouped regression context) for each regression coefficient vector $\boldsymbol{\beta}_{j}$, for $j=1,2,\cdots,n$, the following prior distribution given by
\begin{equation*} \label{eq:1.2}
 \boldsymbol{\beta}_j \mid \lambda^2_j,\sigma^2, \tau^2 \simind \mathcal{N}(0,\lambda^2_j\sigma^2 \tau^2), \quad 
	\lambda^2_j \simind \pi(\lambda^2_j), \quad (\tau,\sigma^2) \sim \pi(\tau,\sigma^2). \tag{1.2}
\end{equation*}
In the above prior, $\lambda^2_j$ for $j=1,\ldots,p$ are the local shrinkage parameters while $\tau^2$ is the global shrinkage parameter. It is important to note that some penalized likelihood estimators can be derived as \textit{maximum a posteriori} (MAP) estimators corresponding to appropriate one-group shrinkage priors for the unknown regression coefficients. For example, Park and Casella (2008) \cite{park2008bayesian} observed that the usual LASSO estimator can be thought of as a MAP estimator corresponding to iid double exponential prior for the regression coefficients, and the double-exponential prior, in fact, can be expressed as a scale mixture of normals which can be written in the form of a one-group prior.

Like the group LASSO, many approaches in the \enquote{ungrouped} problem can be extended to the problem of grouped variable selection and estimation. The researchers used a multivariate-Laplacian one-group prior distribution to model the unknown group coefficients for group selection. This method is called Bayesian group LASSO, and the corresponding MAP estimator is known as Bayesian group LASSO estimator (Raman et al. 2009 \cite{raman2009bayesian} and Kyung et al. 2010 \cite{casella2010penalized}).

Two-group priors can also be used in the group selection problem by using a mixture of degenerate measure at $\mathbf{0}$ and a heavy-tailed absolutely continuous distribution $F$ over $\mathbb{R}^{m_g}$ on the group coefficients $\boldsymbol{\beta}_g$'s independently for $g=1,\ldots, G$. One such example is due to Xu and Ghosh (2015) \cite{xu2015bayesian}.

They used the following hierarchical formulation
\begin{align*} \label{eq:1.3}
	&\mathbf{y}|\mathbf{X},\boldsymbol{\beta},\sigma^2  \sim \mathcal{N}_n(\mathbf{X}\boldsymbol{\beta},\sigma^2I_n),\\
	\boldsymbol{\beta}_g  &\simind \pi_0\delta_{\{\mathbf{0}\}}(\cdot)+ (1-\pi_0) \mathcal{N}_{m_g}(\mathbf{0},\sigma^2 \tau_g^2 I_{m_g}), \textrm{ for } g=1, \cdots , G, \tag{1.3}  \\
	\tau_g^2 &\simind \text{Gamma}(\frac{m_g+1}{2},\frac{\lambda^2}{2}),g=1,2,\cdots,G, \\
	\sigma^2 &\sim \text{Inverse Gamma}(\alpha,\gamma),
\end{align*}
which is known as the Bayesian group LASSO with spike and slab priors (BGL-SS). For the same problem, Yang and Narisetty (2020) \cite{yang2020consistent} modified the BGL-SS by introducing a binary latent variable for each group to indicate the activeness of the group or not and proposing
spike and slab priors on the group regression coefficients depending on the latent variables. 

We are now in a position to articulate the motivation behind our work and outline our main contributions. To achieve this, we must delve deeper into the existing literature. This exploration brings to light many natural questions and issues and highlights unresolved matters in the literature. Our work aims to address some of these questions. 



Grouped variable selection, which involves the estimation of regression coefficients, is a generalization of traditional variable selection methods. We begin with the work of Tang et al. (2018) \cite{tang2018bayesian}. Under the assumption of an orthogonal design in an ungrouped context, Theorem 1 from Tang et al. (2018) \cite{tang2018bayesian} states that if the local shrinkage parameters $\lambda^2_i$ in equation \eqref{eq:1.2} follow polynomial-tailed priors (as described in equation \eqref{eq:1.4} below), then their proposed HT variable selection rule and estimators possess the oracle property (see Fan and Li (2001) \cite{fan2001variable}, Zou (2006) \cite{zou2006adaptive}, and Zou and Zhang (2009) \cite{zou2009adaptive} in this context). This property is defined to achieve both variable selection consistency and optimal estimation rates. However, this result is based on two assumptions: (a) the number of active regressors remains fixed as the sample size increases, and (b) the sparsity level is known. Both assumptions are often challenging to satisfy in high-dimensional problems. Therefore, it is essential to investigate whether similar results hold when these assumptions are not met.


The second point can be partially addressed by showing that an empirical Bayesian version of the procedure proposed by Tang et al. (2018) \cite{tang2018bayesian} has optimality properties. This problem remains unresolved in their work. Furthermore, a careful examination of Theorem 1 from Tang et al. (2018) \cite{tang2018bayesian} reveals a significant weakness in their argument regarding achieving the optimal estimation rate. Hence, the rigorous theoretical treatment of the oracle optimality properties of thresholding procedures based on one-group shrinkage priors for variable selection problems remains unanswered.


We aim to explore these issues within the broader context of group selection and to investigate whether the half-thresholding (HT) technique can be extended to group selection problems so that the corresponding oracle property holds without relying on assumptions (a) and (b). A positive outcome would not only resolve the open question left by Tang et al. (2018) \cite{tang2018bayesian} in the ungrouped context, but would also address the weakness of their proof.

For the group selection problem, Xu \textit{et al.} (2016) \cite{xu2016bayesian} used a variant of the horseshoe prior, which they called the group horseshoe prior. They used two sets of local shrinkage parameters to control shrinkage between the groups and within the groups simultaneously. They applied their method to a prediction problem but did not provide any rule for selecting active groups. Recently, Boss et al. (2023) \cite{boss2023group} also
proposed another version of global-local prior in the context of estimating the group coefficients when the covariates under study form a block-diagonal structure. They introduced group shrinkage and local shrinkage parameters along with the global shrinkage parameter, similar to Xu et al. (2016) \cite{xu2016bayesian}. Their choices were gamma prior for the group shrinkage parameter and inverse-gamma prior for the local shrinkage parameter. The reason behind this choice was that the inverse gamma prior being heavier-tailed than the prior would prevent the overregularization of nonnull coefficients being grouped with nulls. Boss et al. (2023) \cite{boss2023group} established posterior consistency and posterior concentration results for regression coefficients in linear models and mean parameters in sparse normal means models for their proposed prior. Since their target was the estimation of the group coefficients, they also did not provide any rule regarding the selection of a group or the use of the global parameter to incorporate the underlying level of sparsity. To our knowledge, the only work in the Bayesian context in which the oracle property of the proposed estimator has been studied is due to Xu and Ghosh (2014) \cite{xu2015bayesian}, though under the assumption that the number of groups is fixed. Yang and Narisetty (2020) \cite{yang2020consistent} studied selection consistency results but not the asymptotic optimality of their proposed estimator. Nevertheless, it is important to note that both works are under a two-group framework. In a nutshell, in the context of linear regression with grouped prediction, many questions related to the optimality of decision rule based on one-group global-local priors have not been addressed until now.

Motivated by the above discussion, we study in this article a thresholding rule in the group selection problem and also an estimator of the active groups based on a very broad class of one-group shrinkage priors having polynomial tails given by
\begin{equation*} \label{eq:1.4}
 \boldsymbol{\beta}_g \mid \lambda^2_g,\sigma^2, \tau^2 \sim \mathcal{N}_{m_g}(\mathbf{0},\lambda^2_g\sigma^2 \tau^2({\mathbf{X}}^{\mathrm{T}}_g\mathbf{X}_g)^{-1}), \quad 
	\pi(\lambda^2_g) \propto (\lambda^2_g)^{-a-1}L(\lambda^2_g),\tag{1.4}
\end{equation*}
where $a$ is a positive real number and $L:(0,\infty) \to (0,\infty)$ is a measurable non-constant slowly varying function in Karamata’s sense (see Bingham et al., 1987 \cite{bingham_goldie_teugels_1987}), that is, $\frac{L(\alpha x)}{L(x)} \to 1$ as $x \to \infty$, for any $\alpha>0$. Several one-group priors can be expressed in the form \eqref{eq:1.4}. See Section \ref{sec-2} in this context. We assume the error variance $\sigma^2$ to be known for our theoretical analysis, although for our simulation study, we assume $\pi(\sigma^2) \propto \frac{1}{\sigma^2}$, the Jeffreys prior. On the other hand, depending on whether the level of sparsity is known, $\tau$ is treated as a tuning parameter(which depends on $n$) or in an empirical Bayesian or full Bayesian way. See section \ref{sec-2.2} for motivation regarding modifying the hierarchical formulation.

Our proposed group selection rule is referred to as the half-thresholding (HT) rule (given in \eqref{eq:2.5}) and declares a group to be active if the ratio of the $\ell_2$ norm of the posterior mean of the regression coefficients to that of the ordinary least square estimate of the corresponding coefficient vector exceeds half. However, in this work, the decision rule is formulated as a byproduct of two propositions stated in Section \ref{sec-2.3}. Still, for the block-orthogonal design matrix, that decision rule becomes equivalent to the rule proposed by Tang et al. (2018) \cite{tang2018bayesian}. Consequently, our proposed half-thresholding (HT) decision is a generalization to that of Tang et al. (2018) \cite{tang2018bayesian} when the group size is unity. We also propose a corresponding \enquote{Half-thresholding} estimator of the active groups.
 

Our contributions to this article are as follows. Firstly, we propose a new general class of one-group global-local shrinkage priors in the case of group selection. This is achieved by considering the grouping structure while formulating the prior. Secondly, we propose a half-thresholding rule that can be easily implemented regardless of whether the underlying sparsity level is known. We first show that when the proportion of active groups is known, the global shrinkage parameter $\tau$ can be appropriately chosen so that the resulting decision rule becomes an oracle in the sense described before. When the proportion of active groups is unknown, we propose using an empirical Bayes estimate of the global shrinkage component. This estimate generalizes the empirical Bayes estimate proposed by van der Pas et al. (2014) \cite{van2014horseshoe} for large-scale signal detection problems. We show that the resulting data-adaptive half-thresholding rule enjoys the oracle optimality properties under very mild conditions. This is the first result of this kind in the literature.
Thirdly, as an immediate consequence of our rigorous analytical treatment, it readily follows that the variable selection rule proposed by Tang \textit{et al.} (2018) \cite{tang2018bayesian} based on a broad family of shrinkage priors of a group enjoys the oracle optimality properties in the ungrouped problem and settles the optimality of the empirical Bayes rule left open in the paper. However, our theoretical results hold when the true active groups grow substantially with $n$ instead of being fixed, as many existing works assume. Last, we also studied a full Bayes procedure of our proposed decision rule by assigning a non-degenerate prior to the global parameter on our proposed interval and established the oracle property in that context.
It is important to note that we need to develop novel and rigorous analytical techniques to establish these properties theoretically, and these techniques are the first of their kind. Finally, in our simulation studies, we use both empirical Bayes and full Bayes versions of our proposed decision rule and demonstrate their superior performance compared to some well-known group selection methods in the literature.

The remainder of the paper is organized as follows. In Section \ref{sec-2}, the hierarchical form of the modified class of global-local priors with polynomial tail and the proposed half-thresholding rule are described, along with the Gibbs sampling algorithm. Section \ref{sec-3} presents the main theoretical results of the paper. Section \ref{sec-4} deals with the simulation results. The proofs of all theoretical results can be found in Section \ref{sec-5}. The paper ends with some concluding remarks in section \ref{sec-6}.


\subsection{Notation}
For any two sequences of real numbers $\{a_n\}$ and $\{b_n\}$ with $b_n \neq 0$ for all $n$, $a_n \sim b_n$ implies $\lim_{n \to \infty}a_n/b_n=1$. By $a_n=O(b_n)$, and $a_n=o(b_n)$ we denote $|a_n/b_n|<M$ for all sufficiently large $n$, and $\lim_{n \to \infty}a_n/b_n=0$, respectively, $M>0$ being a global constant independent of $n$. We write $a_n \asymp b_n$ to denote that there exist two constants $c_1$ and $ c_2$ such that $0< c_1\leq  a_n/b_n \leq c_2<\infty $ for sufficiently large $n$.
Likewise, for any two positive real-valued functions $f_1(\cdot)$ and $f_2(\cdot)$ with a common domain of definition that is unbounded to the right $f_1(x) \sim f_2(x)$ denotes $\lim_{x \to \infty}f_1(x)/f_2(x)=1.$ Throughout this article, the indicator function of any set $A$ will always be denoted $I\{A\}$.\\

Let $G_{{A}_n}$ and $G_n$ denote the number of active groups and the total number of groups, respectively, with $G_{{A}_n} \leq G_n \leq n$. Since we are interested in the sparse situation, we assume that $G_{{A}_n}=o(G_n)$. Let $\boldsymbol{\beta}_{g}^{0}$ denote the true value of the vector of unknown coefficients $\boldsymbol{\beta}_g$. A matrix $\mathbf{A}$ of order $n \times m$ is said to be block orthogonal if for
any two sub-matrices ${\mathbf{A}}_i$ (of order $n \times m_i$) and ${\mathbf{A}}_j$ (of order $n \times m_j$), we have
${\mathbf{A}}^{\mathrm{T}}_i\mathbf{A}_j = \mathbf{0}$ for all $i \neq j$. 
For any matrix $\mathbf{A}$, $e_{\text{min}}\mathbf{A}$ and $e_{\text{max}}\mathbf{A}$ denote the minimum and the maximum eigenvalues of $\mathbf{A}$, respectively. For any square matrix $\mathbf{A}$, $\mathbf{A}^{\frac{1}{2}}$ is defined as $\mathbf{A}=\mathbf{A}^{\frac{1}{2}}\mathbf{A}^{\frac{1}{2}}$.
Throughout this article, we use the notation $\mathcal{D}$ to denote the data $\mathcal{D}=\{\mathbf{y}\}$. 
\section{Prior Selection and the Half-Thresholding Rule}
\label{sec-2}
Consider the linear model \eqref{eq:1.1}. Let $m=(m_1, \cdots,m_G)$ be the number of individual variables within each group, and $p=\sum_{g=1}^{G}m_g$ be the total number of variables under consideration. Let us assume that the design matrix for the $g^{\text{th}}$ group, denoted $\mathbf{X}_g$, is of full rank, for $g=1,2,\cdots,G$. 

\subsection{Hierarchical form}
\label{sec-2.1}
In this article, we consider the following Bayesian hierarchical structure given by
\begin{align*}\label{eq:2.1}
		\mathbf{y}\mid\mathbf{X},\boldsymbol{\beta},\sigma^2  &\sim \mathcal{N}_n(\mathbf{X}\boldsymbol{\beta},\sigma^2I_n), \\
	\boldsymbol{\beta}_g \mid \lambda^2_g,\sigma^2, \tau^2 &\sim \mathcal{N}_{m_g}(\mathbf{0},\lambda^2_g\sigma^2 \tau^2({\mathbf{X}}^{\mathrm{T}}_g\mathbf{X}_g)^{-1}), \ \text{independently for}\hspace*{0.1cm} g=1,2,\cdots,G, \tag{2.1} \\
	\lambda^2_g &\sim \pi(\lambda^2_g), \ \text{independently for} \hspace*{0.1cm} g=1,2,\cdots,G, \ \text{and} \\
	(\tau,\sigma^2) &\sim \pi(\tau,\sigma^2).
\end{align*}
In (\ref{eq:2.1}), $\lambda_g$ denotes the local shrinkage parameter for the $g^{th}$ group, and $\tau$ denotes the global shrinkage parameter. Here, $\pi(\cdot)$ denotes a non-degenerate prior distribution used to model the global shrinkage component $\tau$, and the error variance $\sigma^2$. Polson and Scott (2010)\cite{polson2010shrink} suggested that in sparse problems, prior distributions on local and global shrinkage parameters should have the following properties:
\begin{enumerate}
    \item The prior on the local shrinkage parameter should have thick tails to accommodate the non-null coefficients, and
    \item  The prior on the global shrinkage parameter should have substantial mass near the origin to account for sparsity.
\end{enumerate}

  Motivated by this and the previous works of Ghosh et al. (2016) \cite{ghosh2016asymptotic}, and Ghosh and Chakrabarti (2017) \cite{ghosh2017asymptotic}, we assume throughout this article that the prior distribution of $\lambda^2_g$ will be of the form
\begin{equation} \label{eq:2.2}
	\pi(\lambda^2_g) \propto (\lambda^2_g)^{-a-1}L(\lambda^2_g).\tag{2.2}
\end{equation}
In (\ref{eq:2.2}) above, $a$ is a positive real number, and $L:(0,\infty) \to (0,\infty)$ is a measurable non-constant slowly
varying function in Karamata’s sense (see Bingham et al., 1987 \cite{bingham_goldie_teugels_1987}), that is, $\frac{L(\alpha x)}{L(x)} \to 1$  as $x \to \infty$, for any $\alpha>0$. 
Priors of the form given in \eqref{eq:2.2} are naturally heavy-tailed.
For an orthogonal design $\mathbf{X}$, the class of one-group shrinkage priors given in \eqref{eq:2.1} satisfying \eqref{eq:2.2} covers a broad array of heavy-tailed global-local shrinkage prior distributions such as the $t$-prior due to Tipping(2001) \cite{tipping2001sparse}, the negative exponential gamma prior due to Griffin and Brown (2005) \cite{griffin2005alternative}, the Horseshoe prior of Carvalho et al. (2009) \cite{carvalho2009handling}, the three-parameter beta normal mixtures of Armagan et al. (2011) \cite{armagan2011generalized}, the generalized double Pareto priors due to Armagan et al. (2013) \cite{armagan2013generalized}, the inverse gamma priors, just to name a few. See, for instance, Ghosh et al. (2016) \cite{ghosh2016asymptotic} and Ghosh and Chakrabarti (2017) \cite{ghosh2017asymptotic}, in this context.\\
\hspace*{0.5cm} The \textit{global shrinkage} parameter $\tau$ is either treated as a tuning parameter $\tau_n$ based on the proportion of non-zero means or is treated in an empirical Bayes or fully Bayesian way depending on whether the proportion of active groups is known or not. For the theoretical development of this paper, we assume that the error variance term $\sigma^2$ is fixed in \eqref{eq:2.1}. See Castillo, Schmidt-Hieber and van der Vaart \cite{castillo2015bayesian}, Rigollet and Tsybakov \cite{rigollet2012sparse} for similar treatment of $\sigma^2$.  On the other hand, for simulation, we employ Jeffry's prior to model the unknown $\sigma^2$. \\
\hspace*{0.5cm} We further assume the following conditions on the slowly varying function $L(\cdot)$ defined in \eqref{eq:2.2}:\\
{\textbf{\hypertarget{assumption1}{Assumption 1:}}}\\
  For $a \geq \frac{1}{2}$:
  \begin{itemize}
        \item There  exists some positive real constant $c_0$ such that $L(t) \geq c_0 \textrm{ for all } t \geq t_0$, for some $t_0>0$, choice of which depends on both $L$ and $c_0$.\\
    \item There exists some $M \in (0,\infty)$ such that $\sup\limits_{t \in (0,\infty)} L(t) \leq M$.
    \end{itemize}

\subsection{Motivation for the Modification of Prior }
\label{sec-2.2}
In this subsection, we motivate the readers regarding the modification in the hierarchical formulation of the general class of one-group shrinkage priors considered in \eqref{eq:2.1} satisfying \eqref{eq:2.2}. As mentioned in the introduction, Tang et al. (2018) \cite{tang2018bayesian} considered the following hierarchical formulation based on the same class of global-local priors in regression problem,
\begin{align*}\label{eq:2.2.1}
    \mathbf{y}|\mathbf{X},\boldsymbol{\beta},\sigma^2  &\sim \mathcal{N}_n(\mathbf{X}\boldsymbol{\beta},\sigma^2I_n),\\
     \boldsymbol{\beta}_j \mid \lambda^2_j,\sigma^2, \tau^2 &\simind \mathcal{N}(0,\lambda^2_j\sigma^2 \tau^2), \textrm{ for } j=1,2,\cdots,p, \tag{2.3}\\
     \lambda^2_j &\simind \pi(\lambda^2_j)=K(\lambda^2_j)^{-a-1}L(\lambda^2_j),
\end{align*}
where $K$ and $L(\cdot)$ are same as discussed before. Hence, in case of a group selection problem, a straightforward extension in the hierarchical formulation would be to model each group coefficient $\boldsymbol{\beta}_g$ as $$\boldsymbol{\beta}_g|  \lambda^2_g,\sigma^2, \tau^2 \simind \mathcal{N}_{m_g}(\mathbf{0},\lambda^2_g\sigma^2 \tau^2I_{m_g})$$ keeping the prior on the local shrinkage coefficient same as before. However, due to this formulation mentioned in \eqref{eq:2.2.1}, Tang et al. (2018) \cite{tang2018bayesian} were able to propose a decision rule for detecting a variable to be significant or not only when the design matrix was orthogonal. So, the problem of formulating the decision rule remains the same while the selection of groups is of interest.
Hence, in a group selection problem, 
it is natural to question about the 
 existence of a decision rule that can declare a group to be active or inactive even if the design matrix is not orthogonal. This was our first and main motivation regarding the modification. On the other hand, keeping in mind the suggestion of Gelman (2006) \cite{gelman2006prior}, the prior on $\boldsymbol{\beta}_g$ was supposed to be such that the variance of the prior distribution of $\boldsymbol{\beta}_g$ should be at the same scale as that of the sufficient statistic 
   $\boldsymbol{\widehat\beta}_g$. This intuitively suggests the use of g-prior due to Zellner (1986) \cite{zellner1986assessing}. The most
   relevant work in this context is due to
   Som et al. (2015) \cite{som2014block} proposed the block hyper-g prior while studying Lindley’s Paradox. However, similar to the usual g-prior, their proposed prior neither models the global and local shrinkage coefficients separately, nor provides any decision rule for detecting whether a group is important or not. This raises the question of whether the group coefficients can be modeled in such a way that the prior can be thought of as a generalization of g-prior having a flavor of global-local priors mixed in it. This
   works as another motivation for the modification of the prior distribution. In the next subsection, we justify our hierarchical form \eqref{eq:2.1} satisfying \eqref{eq:2.2} by
   providing answers to both of the questions raised in this subsection.

\subsection{The Half-Thresholding (HT) Rule}
\label{sec-2.3}
In this subsection, we propose a rule for deciding whether a group is active or not. The rule is motivated by two key observations, which are stated as two propositions below. Proofs of these two are presented in Section \ref{sec-5}. \\

\hspace*{0.5cm} Before stating them, we note that for a block orthogonal design matrix, $\mathbf{X}$ within the hierarchical framework of \eqref{eq:2.1}, the posterior mean of $\boldsymbol{\beta}_g$ conditioned on $(\lambda_g,\tau_n,\sigma^2,\mathcal{D})$ is given by
\begin{equation*}
E(\boldsymbol{\beta}_g\mid\lambda_g,\tau_n,\sigma^2,\mathcal{D})=(1-\kappa_g)\widehat{\boldsymbol{\beta}}_g \hspace{0.1cm},
\end{equation*}
where $\kappa_g=1/(1+\lambda^2_g \tau^2)$ and $\widehat{\boldsymbol{\beta}}_g$ is the least square estimate of ${\boldsymbol{\beta}}_g$.
Therefore, by Fubini's Theorem, it follows that the posterior mean of $\boldsymbol{\beta}_g$, denoted as ${\widehat{\boldsymbol{\beta}}_g}^{\text{PM}}$ is of the form
\begin{equation*} \label{eq:2.6}
	{\widehat{\boldsymbol{\beta}}_g}^{\text{PM}}=	E(\boldsymbol{\beta}_g|\mathcal{D})=(E(1-\kappa_g\mid\tau_n,\sigma^2,\mathcal{D}))\widehat{\boldsymbol{\beta}}_g \hspace{0.1cm}. \tag{2.4}
\end{equation*}
Thus $E(1-\kappa_g\mid\tau_n,\sigma^2,\mathcal{D})$ is the factor by which the usual estimator is shrunk in the Bayesian formulation. The propositions are about the behavior of the shrinkage factor under the null and the alternatives. 

\begin{proposition}
	\label{prop-1}
	Suppose that the $g^{th}$ group is inactive, that is, $\boldsymbol{\beta}_{g}^{0}=\mathbf{0}$. If $\tau_n \to 0$ as $n \to \infty$, then $E(1-\kappa_g\mid\tau_n,\sigma^2,\mathcal{D}) \xrightarrow{P} 0$ as $n \to \infty$.
\end{proposition}
\begin{proposition}
	\label{prop-2}
	Suppose that the $g^{th}$ group is active, that is, $\boldsymbol{\beta}_{g}^{0} \neq \mathbf{0}$. Consider the following assumptions:
	\begin{enumerate}[label=(A\arabic*)]
    \item\label{assmp-1} $\hspace{0.1cm} \textrm{for any}$ active group $g$, there exists some global constant $C_1>0$ such that $e_{min}\mathbf{Q}_{n,g}>C_1$.
    \item \label{assmp-2} $\hspace{0.1cm}  \textrm{for any}$ active group $g, \min_{j} |\beta^{0}_{gj}| > m_n $ with $m_n \propto n^{-b} \textrm{ and } 0 \leq b<\frac{1}{2}$,
    \item \label{assmp-3}  the group size $m_g$ satisfies $s= \sup\limits_{n \geq 1}\max\{m_g: g=1,\cdots, G_n\}<\infty$,
    \item \label{assmp-6} the global parameter $\tau_n \to 0$ as $n \to \infty$ such that $\log (\frac{1}{\tau_n})\lesssim \log G_n$,
    
\end{enumerate}
Under the assumptions \ref{assmp-1}-\ref{assmp-6}, $E(1-\kappa_g\mid\tau_n,\sigma^2,\mathcal{D}) \xrightarrow{P} 1$ as $n \to \infty$.
\end{proposition}

Propositions \ref{prop-1} and \ref{prop-2} leads to the decision rule:\vspace{2mm}

$\textrm{For each } g=1,\cdots,G$,
\begin{align*} \label{eq:2.7}
\text{ the } g^{th} \text{group is considered active if }  E(1-\kappa_g\mid\tau_n,\sigma^2,\mathcal{D}) >0.5, \textrm{ and is inactive otherwise.}
\tag{2.5}
\end{align*}
    

\begin{remark}
	Proposition \ref{prop-1} indicates that if the $g^{th}$ group is inactive, our proposed thresholding rule detects the same, and the only condition required to establish this is that $\tau_n=o(1)$ as $n\rightarrow\infty$. 
\end{remark}
\begin{remark}
	Proposition \ref{prop-2} implies that if $\tau_n \to 0$ not too fast, our thresholding rule successfully identifies an active group (asymptotically) provided the design matrix satisfies a simple condition. However, the condition 
    $\log (\frac{1}{\tau_n})\lesssim \log G_n$ allows to 
    choose $\tau_n$ from a wide range of values. The importance of other conditions is discussed in detail in Remark \ref{remark-4}.
\end{remark}
Since our proposed decision rule declares a group to be active or not depending on whether $E(1-\kappa_g\mid\tau_n,\sigma^2,\mathcal{D})$ exceeds half or not, the rule is called \textit{Half-Thresholding (HT) rule}.
Note that using \eqref{eq:2.6}, the decision rule \eqref{eq:2.7} can be alternatively stated as:
\begin{equation} \label{eq:2.5}
\text{the }	g^{th} \text{group is considered active if } \frac{||{\widehat{\boldsymbol{\beta}}_g}^{\text{PM}}||_{2}}{||\widehat{\boldsymbol{\beta}}_g||_{2}}>0.5, \textrm{ for } g=1,2,\cdots,G. \tag{2.6}
\end{equation}
For unit group size, i.e., $m_g=1, g=1,2,\cdots, G$, this is exactly the variable selection rule of Tang et al. (2018) \cite{tang2018bayesian}. This way, our proposed thresholding rule is a generalization to that of Tang et al. (2018) \cite{tang2018bayesian}, although their motivation for proposing this rule was different.
We define our half-thresholding(HT) estimator of $\boldsymbol{\beta}_g$ corresponding to the variable selection rule as
\begin{equation*} \label{eq:2.8}
	{\widehat{\boldsymbol{\beta}}_g}^{\text{HT}}={\widehat{\boldsymbol{\beta}}_g}^{\text{PM}}I\big\{E(1-\kappa_g\mid\tau_n,\sigma^2,\mathcal{D}) >0.5\big\}, \tag{2.7}
\end{equation*}
where ${\widehat{\boldsymbol{\beta}}_g}^{\text{PM}}$ denotes the posterior mean of the unknown group coefficient $\boldsymbol{\beta}_g$ corresponding to the $g^{\text{th}}$ group. 
\hspace*{0.5cm}	Note that our proposed decision rule \eqref{eq:2.7} uses the global shrinkage parameter $\tau$ as a tuning parameter chosen depending on the sample size $n$( and the $G_{{A}_n}$). This gives rise to the question of the treatment of $\tau$ in fact when $G_{{A}_n}$ is unknown, which happens very often. A natural data-adaptive solution to this problem would be the use of some empirical Bayes estimate(s) of $\tau$ by learning through the data. For the recovery of a sparse normal means vector using the horseshoe prior, van der Pas et al. (2014)\cite{van2014horseshoe} proposed an empirical Bayes estimator of $\tau$ given by
\begin{equation*} \label{eq:2.9}
	\widehat{\tau}=\text{max} \biggl\{\frac{1}{n}, \frac{1}{c_2n}\sum_{i=1}^{n}1 \left(\frac{|y_i|}{\sigma}>\sqrt{c_1\log{n}}\right)\biggr\}, \tag{2.8}
\end{equation*}
where $c_1$ and $c_2$ are two positive constants with  $c_1 \geq 2$ and $c_2 \geq 1$.
Motivated by this, we consider the following empirical Bayes estimate of $\tau$ given by
\begin{equation*} \label{eq:2.10}
	{\widehat{\tau}}^{\text{EB}}=\text{max}\biggl\{\frac{1}{G_n}, \frac{1}{c_2G_n}\sum_{g=1}^{G_n}1\left(\frac{n\widehat{\boldsymbol{\beta}}_g^{\mathrm{T}}\mathbf{Q}_{n,g}\widehat{\boldsymbol{\beta}}_g}{\sigma^2}>c_1 \log{G_n}\right)\biggr\}, \tag{2.9}
\end{equation*}
where $G_n$ denotes the total number of groups that varies with $n$ and satisfies $G_n\leq n$. From the above definition, it readily follows that ${\widehat{\tau}}^{\text{EB}}$ always lies between $\frac{1}{G_n}$ and $1$. Since a lower bound of this estimator is $\frac{1}{G_n}$, it cannot collapse to zero. Collapsing of the estimator to zero is a major concern in the context of the use of such empirical Bayes procedures as mentioned by several authors, such as Carvalho et al. (2009) \cite{carvalho2009handling}, Scott and Berger (2010) \cite{scott2010bayes}, Bogdan et al. (2008) \cite{bogdan2008comparison} and
Datta and Ghosh (2013) \cite{datta2013asymptotic}. It may be noted that when $\mathbf{X}=\mathbf{I}_p$ and $m_g=1$, for $g=1,2,\cdots,G$, \eqref{eq:2.10} boils down to \eqref{eq:2.9}.
\newline
\hspace*{0.5cm} Let $E(1-\kappa_g\mid{\widehat{\tau}}^{\text{EB}},\sigma^2,\mathcal{D})$ denote the posterior shrinkage weight corresponding to the $g^{\text{th}}$ group evaluated at $\tau={\widehat{\tau}}^{\text{EB}}$. Using this empirical Bayes estimate, our proposed data-adaptive decision rule is given by
 \begin{equation*} \label{eq:2.11}
\text{The }	g^{th} \text{group is considered active if } E(1-\kappa_g\mid {\widehat{\tau}}^{\text{EB}},\sigma^2,\mathcal{D}) >0.5, \textrm{ for } g=1,2,\cdots,G,\tag{2.10}
 \end{equation*}
 and the corresponding empirical Bayes half-thresholding(HT) estimator of $\boldsymbol{\beta}_g$, denoted ${\widehat{\boldsymbol{\beta}}_{g, EB}}^{\text{HT}}$, is given by
\begin{equation*}
	\label{eq:2.12}{\widehat{\boldsymbol{\beta}}_{g, EB}}^{\text{HT}}={\widehat{\boldsymbol{\beta}}_g}^{\text{PM}}I\big\{E(1-\kappa_g\mid {\widehat{\tau}}^{\text{EB}},\sigma^2,\mathcal{D}) >0.5\big\}.\tag{2.11}
\end{equation*}
\hspace*{0.5cm}	
An alternative approach to the above empirical Bayes procedure is to assign a non-degenerate joint prior density to $(\tau,\sigma)$. In line with the recommendation of Polson and Scott (2010) \cite{polson2010shrink}, for a fully Bayesian approach, we assume
$\pi(\sigma^2) \propto \frac{1}{\sigma^2}$ and $\pi( \tau)$ is the restriction of half Cauchy prior on a suitably chosen small interval near zero (details and motivations of which are available in Section \ref{sec-3.2.3}).
The full Bayesian half-thresholding (HT) decision rule is given by
  \begin{equation*} \label{eq:2.14}
 \text{The }	g^{th} \text{group is considered active if } E(1-\kappa_g\mid\mathcal{D}) >0.5, \textrm{ for } g=1,2,\cdots,G,\tag{2.12}
 \end{equation*}
  and the corresponding full Bayes half-thresholding(HT) estimator of $\boldsymbol{\beta}_g$, denoted ${\widehat{\boldsymbol{\beta}}_{g, FB}}^{\text{HT}}$, is given by
\begin{equation*}
	\label{eq:2.15}{\widehat{\boldsymbol{\beta}}_{g, FB}}^{\text{HT}}={\widehat{\boldsymbol{\beta}}_g}^{\text{PM}}I\big\{E(1-\kappa_g\mid\mathcal{D}) >0.5\big\}. \tag{2.13}
\end{equation*}
\hspace*{0.5cm}	
For the implementation of these decision rules, one needs to sample from the posterior distribution of relevant parameters. These are explained in the next subsection.
\subsection{Gibbs Sampling}
\label{sec-2.4}
Within the hierarchical form \eqref{eq:2.1}, and using the prior distributions on $\tau$ and $\sigma^2$ stated before, the Gibbs samples are drawn from the full conditional distributions as follows:
\subsubsection*{(1) Sampling from the Posterior Distribution of $\boldsymbol{\beta}_g$:}
Since, the full posterior distribution of $\boldsymbol{\beta}$ given $(\boldsymbol{\lambda}^2,\sigma^2, \tau^2,\mathcal{D})$ is
\begin{equation*}
	\pi(\boldsymbol{\beta}\mid\boldsymbol{\lambda}^2,\sigma^2, \tau^2,\mathcal{D}) \propto \exp\bigg[-\frac{1}{2 \sigma^2}\left(\boldsymbol{\beta}^{\mathrm{T}}{\mathbf{X}}^{\mathrm{T}}\mathbf{X}\boldsymbol{\beta}-2\boldsymbol{\beta}^{\mathrm{T}}{\mathbf{X}}^{\mathrm{T}}\mathbf{y}+\sum_{g=1}^{G} \frac{\boldsymbol{\beta}_g^{\mathrm{T}}{\mathbf{X}}^{\mathrm{T}}_g\mathbf{X}_g\boldsymbol{\beta}_g}{\lambda^2_g\tau^2}\right)\bigg]
\end{equation*}
we obtain for $g=1,2,\cdots,G$,
\begin{equation*}
\pi(\boldsymbol{\beta}_g\mid\boldsymbol{\beta}_{-g},\boldsymbol{\lambda}^2,\sigma^2, \tau^2,\mathcal{D}) \propto \exp\bigg[-\frac{1}{2 \sigma^2}\left(\boldsymbol{\beta}_g^{\mathrm{T}}{\mathbf{X}}^{\mathrm{T}}_g\mathbf{X}_g\boldsymbol{\beta}_g-2\boldsymbol{\beta}_g^{\mathrm{T}}{\mathbf{X}}^{\mathrm{T}}_g\mathbf{y}+\frac{\boldsymbol{\beta}_g^{\mathrm{T}}{\mathbf{X}}^{\mathrm{T}}_g\mathbf{X}_g\boldsymbol{\beta}_g}{\lambda^2_g\tau^2}+\sum_{g'(\neq g)=1}^{G} \boldsymbol{\beta}_g^{\mathrm{T}}{\mathbf{X}}^{\mathrm{T}}_g\mathbf{X}_{-g}\boldsymbol{\beta}_{-g} \right)\bigg],
\end{equation*}
 This is equivalent to saying for $g=1,2,\cdots,G$,
\begin{equation*}
\boldsymbol{\beta}_g \mid (\boldsymbol{\beta}_{-g},\boldsymbol{\lambda}^2,\sigma^2, \tau^2,\mathcal{D}) \sim \mathcal{N}_{m_g}(\boldsymbol{\mu}_g,\sigma^2\boldsymbol{\Sigma}_g),
\end{equation*}
with $\boldsymbol{\mu}_g=(1+\frac{1}{\lambda^2_g\tau^2})^{-1}({\mathbf{X}}^{\mathrm{T}}_g\mathbf{X}_g)^{-1}({\mathbf{X}}^{\mathrm{T}}_g\mathbf{y}-\sum_{g'(\neq g)=1}^{G} {\mathbf{X}}^{\mathrm{T}}_g\mathbf{X}_{-g}\boldsymbol{\beta}_{-g})$ and \\ $\boldsymbol{\Sigma}_g=(1+\frac{1}{\lambda^2_g\tau^2})^{-1}({\mathbf{X}}^{\mathrm{T}}_g\mathbf{X}_g)^{-1}=(1-\kappa_g)({\mathbf{X}}^{\mathrm{T}}_g\mathbf{X}_g)^{-1}$.\\
With the additional assumption on the block-orthogonality of the design matrix $\mathbf{X}$, 
we have for $g=1,2,\cdots,G$,
\begin{equation*}
\boldsymbol{\beta}_g \mid (\boldsymbol{\lambda}^2,\sigma^2, \tau^2,\mathcal{D}) \simind \mathcal{N}_{m_g}(\boldsymbol{\mu}_g,\sigma^2\boldsymbol{\Sigma}_g),
\end{equation*}
 with $\boldsymbol{\mu}_g=(1-\kappa_g)\widehat{\boldsymbol{\beta}}_g$ and $\boldsymbol{\Sigma}_g=(1-\kappa_g)({\mathbf{X}}^{\mathrm{T}}_g\mathbf{X}_g)^{-1}$.
\subsubsection*{(2) Sampling from the Posterior Distribution of $\sigma^2$:}
The full posterior distribution of $\sigma^2$ conditioned on $(\boldsymbol{\beta},\boldsymbol{\lambda}^2,\tau^2,\mathcal{D})$ is given by
\begin{equation*}
	\pi(\sigma^2\mid \boldsymbol{\beta},\boldsymbol{\lambda}^2,\tau^2,\mathcal{D}) \propto (\sigma^2)^{-(\frac{n}{2}+\sum_{g=1}^{G}\frac{m_g}{2}+1)}\times \exp \bigg[-\frac{1}{\sigma^2}\bigg\{\frac{(\mathbf{y}-\mathbf{X}\boldsymbol{\beta})^{\mathrm{T}}(\mathbf{y}-\mathbf{X}\boldsymbol{\beta})}{2}+\sum_{g=1}^{G}\frac{\boldsymbol{\beta}_g^{\mathrm{T}}{\mathbf{X}}^{\mathrm{T}}_g\mathbf{X}_g\boldsymbol{\beta}_g}{2\lambda^2_g\tau^2}\bigg\}\bigg].
\end{equation*}
Hence,
\begin{equation*} \sigma^2\mid(\boldsymbol{\beta},\boldsymbol{\lambda}^2,\tau^2,\mathcal{D}) \sim \text{Inverse Gamma}\left(\frac{n}{2}+\sum_{g=1}^{G}\frac{m_g}{2},\frac{(\mathbf{y}-\mathbf{X}\boldsymbol{\beta})^{\mathrm{T}}(\mathbf{y}-\mathbf{X}\boldsymbol{\beta})}{2}+\sum_{g=1}^{G}\frac{\boldsymbol{\beta}_g^{\mathrm{T}}{\mathbf{X}}^{\mathrm{T}}_g\mathbf{X}_g\boldsymbol{\beta}_g}{2\lambda^2_g\tau^2}\right).
\end{equation*}

\subsubsection*{(3) Sampling from the Posterior Distribution of $\lambda^2_g$}
Observe that, for each $g=1, 2,\cdots, G$,
\begin{equation*}
\pi(\lambda^2_g\mid\boldsymbol{\beta}_g,\sigma^2, \tau^2,\mathcal{D}) \propto (\lambda^2_g)^{-\frac{(m_g+1)}{2}}(1+\lambda^2_g)^{-1}\times \exp\bigg[-\frac{1}{\lambda^2_g}\cdot \frac{\boldsymbol{\beta}_g^{\mathrm{T}}{\mathbf{X}}^{\mathrm{T}}_g\mathbf{X}_g\boldsymbol{\beta}_g}{2\tau^2\sigma^2}\bigg]
\end{equation*}
Using the Slice-sampling approach of Damlen et al.(1999) \cite{damlen1999gibbs}, posterior sampling is done in two steps:

\begin{enumerate}
\item Given $\lambda^2_g$, sample $u_g$ from the Uniform distribution supported over the interval $(0,1+\lambda^2_g)$.
\item For given $(\boldsymbol{\beta}_g,\sigma^2, \tau^2,\mathcal{D})$, sample $\lambda^2_g$ from an inverse-gamma distribution with parameters $\frac{(m_g-1)}{2}$ and $\frac{\boldsymbol{\beta}_g^{\mathrm{T}}{\mathbf{X}}^{\mathrm{T}}_g\mathbf{X}_g\boldsymbol{\beta}_g}{2\tau^2\sigma^2}$, truncated over the interval $(0,\frac{1-u_g}{u_g})$.
\end{enumerate}
\subsubsection*{(4) Sampling from the Posterior Distribution of $\tau^2$}
\begin{equation*}
	\pi(\tau^2|\boldsymbol{\beta},\sigma^2,\boldsymbol{\lambda}^2,\mathcal{D}) \propto \frac{1}{1+\tau^2}\times (\tau^2)^{-\frac{p}{2}}\exp\left(-\frac{1}{\tau^2}\sum_{g=1}^{G}\frac{\boldsymbol{\beta}_g^{\mathrm{T}}{\mathbf{X}}^{\mathrm{T}}_g\mathbf{X}_g\boldsymbol{\beta}_g}{2\lambda^2_g\sigma^2}\right).
\end{equation*}
Again, using the Slice-sampling approach of Damlen et al.(1999) \cite{damlen1999gibbs}, samples are drawn from the above posterior distribution of $\tau^2$ as follows:
\begin{enumerate}
\item Given $\tau^2$, sample $u$ from the Uniform distribution supported over the interval $(0,1+\tau^2)$.
\item Given $(\boldsymbol{\beta},\sigma^2,\boldsymbol{\lambda}^2,\mathcal{D})$, sample $\tau^2$ from an inverse-gamma distribution with parameters $\frac{(p-2)}{2}$ and $\sum_{g=1}^{G}\frac{\boldsymbol{\beta}_g^{\mathrm{T}}{\mathbf{X}}^{\mathrm{T}}_g\mathbf{X}_g\boldsymbol{\beta}_g}{2\lambda^2_g\sigma^2}$, truncated to have zero probability outside the interval $(0,\frac{1-u}{u})$.
\end{enumerate}

\section{Main Theoretical results}
\label{sec-3}
In this section, we present our theoretical results concerning asymptotic properties of estimation of the group coefficients and variable selection using the proposed Half-thresholding (HT) rule.
Following the works of Fan and Li (2001) \cite{fan2001variable} and Zou (2006) \cite{zou2006adaptive}, our aim here is to establish that the proposed half-thresholding methods defined in \eqref{eq:2.7}, \eqref{eq:2.11} and \eqref{eq:2.14} attain the oracle properties(defined below) asymptotically as the number of observations $n$ grows to infinity. Let $\mathcal{A}=\{g:\boldsymbol{\beta}_{g}^{0} \neq \mathbf{0}\}$ and $\mathcal{A}_n=\{g:\widehat{\boldsymbol{\beta}}_g^{\text{HT}} \neq \mathbf{0}\}$ denote respectively the set of true active groups and the groups declared active by our half-thresholding rule. The aforesaid articles defined a procedure $\delta$ to be an oracle if the resultant procedure can identify the true model asymptotically and the estimator 
corresponding to that procedure $\widehat{\boldsymbol{\beta}}(\delta)$ can achieve the optimal rate of estimation simultaneously. The exact forms of these expressions in our context will be discussed later.
As mentioned before, for studying the oracle properties of the thresholding rules \eqref{eq:2.7}, \eqref{eq:2.11}, and \eqref{eq:2.14}, we treat the global shrinkage component $\tau$ either as a tuning parameter to be chosen appropriately when the number of active groups is assumed to be known or it is replaced either by an empirical Bayes estimator as given in \eqref{eq:2.10} or is modeled by a non-degenerate prior in case the number of active groups is unknown. In both cases, however, the error variance term $\sigma^2$ is assumed to be known and does not vary with $n$. \\ 
\hspace*{0.5cm} Since our proposed half-thresholding (HT) rules crucially hinge upon the posterior shrinkage coefficients, for the sake of completeness, we describe below the posterior distribution of $\kappa_g$ given by
\begin{equation*}\label{eq:3.1}
	\pi(\kappa_g\mid\tau_n,\sigma^2,\mathcal{D})	 \propto \kappa_g^{(a+\frac{m_g}{2}-1)}(1-\kappa_g)^{-a-1}L\left(\frac{1}{\tau^2_n}(\frac{1}{\kappa_g}-1)\right)\exp\left(-\kappa_g\cdot \frac{n\widehat{\boldsymbol{\beta}}_g^{\mathrm{T}}\mathbf{Q}_{n,g}\widehat{\boldsymbol{\beta}}_g }{2 \sigma^2}\right), 0<\kappa_g<1, \tag{3.1}
\end{equation*}
where $\mathbf{Q}_{n,g}=\frac{{\mathbf{X}}^{\mathrm{T}}_g\mathbf{X}_g}{n}$ and $g=1,2,\cdots,G$. Note that, since the error variance $\sigma^2$ is assumed to be known, the above posterior distribution of $\kappa_g$ depends only on $\tau_n$ and the data $\mathcal{D}$ We repeatedly make careful exploitation of this last observation to establish the oracle properties of the half-thresholding rules proposed in this paper. On the other hand, when the global shrinkage parameter $\tau$ is replaced with an empirical Bayes estimator $\widehat{\tau}^{EB}$ or a prior is assigned to it, the posterior distribution of $\kappa_g$ depends on the entire dataset $\mathcal{D}$ which makes the theoretical derivations significantly different, and technically more challenging.

\subsection{Oracle properties of the HT procedure when $\tau$ is known}
\label{sec-3.1}

In this sub-section, we treat the \textit{global shrinkage} parameter $\tau$ as a tuning parameter. Propositions \ref{prop-1} and \ref{prop-2} stated in subsection \ref{sec-2.3} indicate that the half-thresholding rule of the form \eqref{eq:2.7} correctly identifies the individual groups as active or inactive. Theorem \ref{Thm-1} below ensures the same for the overall group selection problem when the sample size $n$ grows to infinity. Hence the proposed half-thresholding rule defined in \eqref{eq:2.7} enjoys model selection consistency.  Proof of this result is presented in Section \ref{sec-5}.

\begin{theorem}[Variable Selection Consistency] \label{Thm-1} Consider the hierarchical framework of \eqref{eq:2.1} where $\pi(\lambda^2_g)$ is as in \eqref{eq:2.2} and the half-thresholding (HT) rule \eqref{eq:2.7} based on these. Let $\mathcal{A}_n=\{g:\boldsymbol{\beta}_{g}^{0} \neq \mathbf{0}\}$ and $\widehat{\mathcal{A}}_n=\{g:\widehat{\boldsymbol{\beta}}_g^{\text{HT}} \neq \mathbf{0}\}$ denote respectively the set of truly active groups, and the set of groups declared active by the half-thresholding rule \eqref{eq:2.7}. Let $\mathbf{Q}_{n,g}={\mathbf{X}}^{\mathrm{T}}_g\mathbf{X}_g/n$, for $g=1,\cdots, G$ and $r_n=\frac{G_{\mathcal{A}_n}}{G_n}$.
\newline

Consider the following assumptions along with \ref{assmp-1}-\ref{assmp-3}:
	\begin{enumerate}[label=(B\arabic*)]
    \item \label{assmp-4} the total number of active groups $|\mathcal{A}|=G_{\mathcal{A}_n}$ is known and satisfies $G^{\epsilon_1}_n \lesssim G_{\mathcal{A}_n} \lesssim G^{\epsilon_2}_n$ for some $0\leq \epsilon_1 <\epsilon_2<1$, and
    
    \item \label{assmp-5} the global parameter $\tau_n \to 0$ as $n \to \infty$ such that $r^{\frac{1+\delta_2}{1-\epsilon_2}}_n \lesssim \tau_n \lesssim r^{\frac{1+\delta_1}{1-\epsilon_1}}_n$ for some $\delta_2 >\delta_1>\frac{\epsilon_2-\epsilon_1}{1-\epsilon_2}$.
\end{enumerate}
Suppose further that $L(\cdot)$ in \eqref{eq:2.2} satisfies \hyperlink{assumption1}{Assumption 1}. 	 Then, under assumptions \ref{assmp-1}-\ref{assmp-3} and \ref{assmp-4}, \ref{assmp-5}, our half-thresholding rule \eqref{eq:2.7} results in variable selection consistency, i.e.,
we have, 
	\begin{equation*} 
		\lim_{n \to \infty} P(\mathcal{A}_n=\widehat{\mathcal{A}}_n)=1 \hspace*{0.2cm}\text{as} \hspace*{0.2cm} n \to \infty .
	\end{equation*}
\end{theorem}
\begin{remark}
 Observe that, asserting $\lim_{n \to \infty} P(\mathcal{A}_n=\widehat{\mathcal{A}}_n)=1 \hspace*{0.2cm}\text{as} \hspace*{0.2cm} n \to \infty $ is equivalent to saying \begin{equation*} 
		\lim_{n \to \infty} P(\mathcal{A}_n \neq \widehat{\mathcal{A}}_n)=0 \hspace*{0.2cm}\text{as} \hspace*{0.2cm} n \to \infty 
	\end{equation*}
which is what we establish to prove Theorem \ref{Thm-1}. This is proved by showing 
\begin{equation*} \label{eq:3.2}
    \sum_{g \in \mathcal{A}} P(E(1-\kappa_g\mid\tau_n,\sigma^2,\mathcal{D})<\frac{1}{2}) \to 0, \textrm{ as } n \to \infty, \tag{3.2}
\end{equation*}
and 
\begin{equation*} \label{eq:3.3}
    \sum_{g \notin \mathcal{A}} P(E(1-\kappa_g\mid\tau_n,\sigma^2,\mathcal{D})>\frac{1}{2}) \to 0, \textrm{ as } n \to \infty. \tag{3.3}
\end{equation*}
Thus, not only do both the probabilities of type-I error and type-II error tend to $0$ as $n \to \infty$ individually, but their corresponding sums also tend to $0$ as $n \to \infty$. 
\end{remark}
\begin{remark}
\label{remark-4}
Condition \ref{assmp-1} is very natural in variable selection problems. Johnson and Rossell (2012) \cite{johnson2012bayesian} and Armagan et al. (2013) \cite{armagan2013posterior} assumed the same condition on the eigenvalues of $({\mathbf{X}}^{\mathrm{T}}\mathbf{X})/n$ while studying the posterior contraction rates in high-dimensional regression problems. Under \ref{assmp-4}, a condition similar to \ref{assmp-5} of Theorem \ref{Thm-1} was considered in Tang et al. (2018) \cite{tang2018bayesian}. But the main difference lies in the assumption regarding the design matrix $\mathbf{X}$ and the cardinality of $|\mathcal{A}|$ while proving the result. In their work, Tang et al. (2018) \cite{tang2018bayesian} assumed the corresponding design matrix $\mathbf{X}$ to be orthogonal (i.e. ${\mathbf{X}}^{\mathrm{T}}\mathbf{X}=n\mathbf{I}_p$) and the number of active variables is independent of $n$, which restricts the applicability of their result. For an orthogonal design, assumption \ref{assmp-1} is trivially satisfied. Further, assuming $|\mathcal{A}|$ being fixed (that is, independent of $n$) implies that assumption \ref{assmp-2} is not required at all. To deal with a more general scenario, we have allowed the total number of active groups to vary with $n$. Therefore, in several aspects, Theorem \ref{Thm-1} of the present article is a generalization of their work.
Condition \ref{assmp-2} has been used by Zhao and Yu (2006) \cite{zhao2006model} while studying the sign consistency of LASSO. Recently, Zhang and Xiang (2016) \cite{zhang2016oracle} and Wang and Tian (2019) \cite{wang2019adaptive} assumed this condition in the context of selection consistency of adaptive group LASSO in high-dimensional linear models. The assumption on the finiteness of the group size, \ref{assmp-3} is also frequently used in group selection problems. See the works of Xu and Ghosh (2015) \cite{xu2015bayesian} and Yang and Narisetty (2020) \cite{yang2020consistent} in this context. Since, we assume $G_n \leq n$, \ref{assmp-3} indicates that the number of groups $G_n$ can be assumed of the order $n$. 
Note that \ref{assmp-4} and \ref{assmp-5} suggest a wide range of choices for $\tau_n$ to achieve selection consistency. A possible choice for $\tau_n$ is $\tau_n \asymp G_n^{-1-\delta}, \delta >0$. In the next paragraph, we mention very briefly the key steps for proving Theorem \ref{Thm-1}.\\
\end{remark}
 

\hspace*{0.5cm} To establish \eqref{eq:3.2} and \eqref{eq:3.3}, we first need to obtain some concentration inequalities based on the posterior distribution of $\kappa_g$ and the posterior mean of $1-\kappa_g, g=1,2,\cdots, G$. These are provided as lemmas \ref{lem1}-\ref{lem4} in Section \ref{sec-5}. Next, we are required to find upper bounds for the tail probabilities of non-central and central $\chi^2$ random variables. For appropriate non-central $\chi^2$, assumptions \ref{assmp-1}-\ref{assmp-4}, and Mill's ratio come in handy. On the other hand, tail bounds for central $\chi^2$ random variables (see Lemma \ref{lem5}) 
due to Gabcke (2015) \cite{gabcke2015} help us complete the proof of our result. \\
\hspace*{0.5cm} The following theorem, namely, Theorem \ref{Thm-2}, establishes the fact that the half-thresholding rule in \eqref{eq:2.7} achieves optimal estimation rate under mild conditions. Proof of this result is deferred to Section \ref{sec-5}.
\begin{theorem}
    	\label{Thm-2}
     Consider the hierarchical framework of \eqref{eq:2.1} satisfying \eqref{eq:2.2}, and the half-thresholding (HT) rule \eqref{eq:2.7} based on this.  Let $\boldsymbol{\beta}_{\mathcal{A}}^{0}=\{\boldsymbol{\beta}_{g}^{0}: g \in \mathcal{A}\}$ and $\widehat{\boldsymbol{\beta}}_{\mathcal{A}}^{\text{HT}} =\{ \widehat{\boldsymbol{\beta}}_g^{\text{HT}} : g \in \mathcal{A}\}$. Consider the following assumptions:
     \begin{enumerate}[label=(C\arabic*)]
    \item\label{assmpan-1} There exist global constants $C_1>0$, $C_2>0$ with $0<C_1\leq C_2<\infty$ such that $ 0<C_1 \leq \frac{1}{n} e_{min} ({\mathbf{X}}^{\mathrm{T}} \mathbf{X}) \leq \frac{1}{n} e_{max} ({\mathbf{X}}^{\mathrm{T}} \mathbf{X})  \leq C_2 <\infty$.
    \item \label{assmpan-2} $\hspace{0.1cm} \textrm{For all } g \in \mathcal{A}, \min_{j} |\beta^{0}_{gj}| > C_3 $, for some global constant $C_3>0$.
     \item \label{assmpan-3} The global parameter $\tau_n \to 0$ such that $G_n \sqrt{\tau_n} \log (\frac{1}{\tau_n}) \to \infty$ and $\sqrt{\tau_n} \log (\frac{1}{\tau_n})=o(\frac{1}{|\mathcal{A}|})$  as $n \to \infty$.

\end{enumerate}
Suppose further that $L(\cdot)$ in \eqref{eq:2.2} satisfies \hyperlink{assumption1}{Assumption 1}. 
	 Then, under assumptions $\ref{assmpan-1}- \ref{assmpan-3}$, the resultant estimator corresponding to \eqref{eq:2.7} enjoys optimal estimation rate, i.e.,
     for any vector $\boldsymbol{\alpha}$ with $|| \boldsymbol{\alpha}||=1$ and $\boldsymbol{\Sigma}_{\mathcal{A}}={\mathbf{X}^{\mathrm{T}}_{\mathcal{A}}}{\mathbf{X}_{\mathcal{A}}}$, we have 
     \begin{align*}
         {\boldsymbol{\alpha}}^{\mathrm{T}} {\boldsymbol{\Sigma}^{\frac{1}{2}}_{\mathcal{A}}}(\widehat{\boldsymbol{\beta}}_{\mathcal{A}}^{\text{HT}}-\boldsymbol{\beta}_{\mathcal{A}}^0) \xrightarrow{d} \mathcal{N}({0}, \sigma^2), \textrm{ as } n \to \infty.
     \end{align*}

\end{theorem}
\begin{remark}
    First, note that \ref{assmpan-1} is a slightly stronger assumption than \ref{assmp-1} since it assumes an upper bound on the largest eigenvalue of $\frac{\mathbf{X}^{\mathrm{T}}\mathbf{X}}{n}$. This is
a common assumption in the high-dimensional variable selection literature.
See, e.g., assumption (A1) of Zou and Zhang (2009) \cite{zou2009adaptive}, who assumed exactly the same condition while establishing the oracle property of their proposed adaptive elastic net estimator. Similarly, condition \ref{assmpan-2} is also a stronger version of \ref{assmp-2} as $b=0$ in \ref{assmp-2} corresponds to \ref{assmpan-2}. This assumption is also needed in our argument to establish asymptotic normality. It is interesting to note that, \ref{assmpan-2} is weaker than that of (A6) of Zou and Zhang (2009) \cite{zou2009adaptive} where an upper bound on the rate of growth of $|\beta^0_{j}|$ was also assumed for proving asymptotic normality of adaptive elastic net estimator. Also, observe that \ref{assmp-4}, \ref{assmp-5} (used in Theorem \ref{Thm-1}) and \ref{assmpan-3} provide some choices of $\tau$ for achieving both selection consistency and optimal estimation rate. One such choice of $\tau$ is $\tau \asymp G_n^{-1-\delta}, 0<\delta \leq 1$. This choice is used for an estimator in our simulation study. It also provides an idea about the range of the prior distribution of $\tau$ in the case of a full Bayes procedure when one models the situation using a non-degenerate prior on $\tau$. This is discussed in detail in section \ref{sec-3.2.3}.
\end{remark}

\hspace*{0.5cm} The proof of Theorem \ref{Thm-2} involves establishing two key facts, namely, 
\begin{equation*} \label{eq:3.4}
      {\boldsymbol{\alpha}}^{\mathrm{T}} {\boldsymbol{\Sigma}^{\frac{1}{2}}_{\mathcal{A}}}(\widehat{\boldsymbol{\beta}}_{\mathcal{A}}-\boldsymbol{\beta}_{\mathcal{A}}^0) \xrightarrow{d} \mathcal{N}({0}, \sigma^2), \textrm{ as } n \to \infty, \tag{3.4}
\end{equation*}
and
\begin{equation*} \label{eq:3.5}
    {\boldsymbol{\alpha}}^{\mathrm{T}} {\boldsymbol{\Sigma}^{\frac{1}{2}}_{\mathcal{A}}}(\widehat{\boldsymbol{\beta}}_{\mathcal{A}}^{\text{HT}}-\widehat{\boldsymbol{\beta}}_{\mathcal{A}}) \xrightarrow{P}0 \textrm{ as } n \to \infty, \tag{3.5}
\end{equation*}
where $ \widehat{\boldsymbol{\beta}}_{\mathcal{A}} =\{ 
\widehat{\boldsymbol{\beta}}_g: g \in \mathcal{A} \}$.
Finally, a simple application of Slutsky’s theorem results in the proof of Theorem \ref{Thm-2}. \eqref{eq:3.4} holds due to the results of the linear model followed by the block orthogonality of the design matrix. On the other hand, to establish \eqref{eq:3.5} one needs to use the Cauchy-Schwarz inequality along with assumptions \ref{assmpan-1}-\ref{assmpan-3}. In this context, some of the arguments used in Lemma 3 of Ghosh and Chakrabarti (2017) \cite{ghosh2017asymptotic} come in handy to obtain the asymptotic optimality of our proposed HT estimator. 
\begin{remark}
 To establish asymptotic normality, we have assumed a condition on the eigenvalues of the design matrix, as given in \ref{assmpan-1}.
However, a particular choice of the design matrix corresponding to the $g^{\text{th}}$ group trivializes the assumption and provides the following statement immediately. The proof follows similarly and is hence omitted. \\
Consider the hierarchical framework of \eqref{eq:2.1} satisfying \eqref{eq:2.2}, and the half-thresholding (HT) rule \eqref{eq:2.7} based on this along with an orthogonal design matrix, i.e. ${\mathbf{X}}^{\mathrm{T}}\mathbf{X}=n I_{p}$.  Let $\boldsymbol{\beta}_{\mathcal{A}}^{0}=\{\boldsymbol{\beta}_{g}^{0}: g \in \mathcal{A}\}$ and $\widehat{\boldsymbol{\beta}}_{\mathcal{A}}^{\text{HT}} =\{ \widehat{\boldsymbol{\beta}}_g^{\text{HT}} : g \in \mathcal{A}\}$. Assume $|\mathcal{A}|$ is fixed and \ref{assmp-3} is satisfied along with $L(\cdot) $ defined in \eqref{eq:2.2} satisfies  \hyperlink{assumption1}{Assumption 1}.
     Then for all $g \in \mathcal{A}$, we have     
     \begin{equation*} 
		\sqrt{n}\left(\widehat{\boldsymbol{\beta}}_g^{\text{HT}} - \boldsymbol{\beta}_g^0\right) \xrightarrow{d} \mathcal{N}_{m_g}( \mathbf{0}, \sigma^2 I_{m_g}) \textrm{ as } n \to \infty.
	\end{equation*}
The above argument implies that, when the group size reduces to unity, our result shows that the asymptotic distribution of the half-thresholding estimator proposed by Tang et al. (2018) \cite{tang2018bayesian} also achieves the optimal estimation rate. As stated in the introduction, a major contribution of Theorem \ref{Thm-2} lies in filling the gap present in the work of Tang et al. (2018) \cite{tang2018bayesian} regarding asymptotic normality of their proposed HT estimator. On the other hand, this remark establishes that the asymptotic optimality result holds based on the assumptions of Tang et al. (2018) \cite{tang2018bayesian} using our technique.
\end{remark}

\hspace*{0.05cm} Theorem \ref{Thm-1} along with Theorem \ref{Thm-2} show that our proposed half-thresholding rule \eqref{eq:2.7} is an oracle when the global shrinkage parameter is treated as a tuning one. 

\subsection{Oracle properties of the HT procedure for the empirical Bayes and full Bayes approaches}
\label{sec-3.2}
From Theorems \ref{Thm-1} and \ref{Thm-2} of the previous subsection, it is evident that the choice of $\tau$ plays a crucial role in the optimality of our variable selection rule and the estimate. It was shown that an appropriate choice of the global shrinkage parameter $\tau$ based on the proportion of active groups ensures such oracle properties. But as mentioned before, this proportion may not be known a priori. In such situations, we treat the global shrinkage parameter in empirical/full Bayes ways.
In the next three subsections, we discuss the properties of the empirical Bayes versions of the HT rule as in \eqref{eq:2.11} and \eqref{eq:3.6} and its full Bayes version as in \eqref{eq:2.14}.
\subsubsection{Oracle properties of the HT procedure using empirical Bayes approach}
\label{sec-3.2.1}
As described earlier, motivated by the work of van der Pas et al. (2014) \cite{van2014horseshoe}, we consider an empirical Bayes estimate of $\tau$ of the form \eqref{eq:2.10}. Theorem \ref{Thm-7} and \ref{Thm-6} together establish the significant fact that the rule \eqref{eq:2.11} and the corresponding estimate enjoy variable selection consistency and optimal estimation rate respectively. To the best of our knowledge, this is the first result of this kind using global-local shrinkage priors in sparse high-dimensional regression problems using the empirical Bayesian method. Proof of the theorem is deferred to section \ref{sec-5}.

\begin{theorem}
\label{Thm-7}
     Consider the hierarchical framework of \eqref{eq:2.1} where $\pi(\lambda^2_g)$ satisfies \eqref{eq:2.2}, and the empirical Bayes half-thresholding (HT) rule \eqref{eq:2.11} based on this.  Let $\mathcal{A}_n=\{g:\boldsymbol{\beta}_{g}^{0} \neq \mathbf{0}\}$ and $\widehat{\mathcal{A}}_n=\{g:{\widehat{\boldsymbol{\beta}}_{g, EB}}^{\text{HT}} \neq \mathbf{0}\}$ denote respectively the set of active groups, and the set of groups declared active by the half-thresholding rule \eqref{eq:2.11}. Recall, $\mathbf{Q}_{n,g}={\mathbf{X}}^{\mathrm{T}}_g\mathbf{X}_g/n$, for $g=1,\cdots, G$. Assume that $|\mathcal{A}|=G_{A_n}$ is unknown and tends to infinity as $n \to \infty$.
Suppose that $L(\cdot)$ in \eqref{eq:2.2} satisfies \hyperlink{assumption1}{Assumption 1} and assumptions \ref{assmp-1}-\ref{assmp-3} hold.
     Also, assume that, for $a \geq 1$, $G^{\epsilon_1}_n \lesssim G_{\mathcal{A}_n} \lesssim G^{\epsilon_2}_n$ for some $0< \epsilon_1 <\epsilon_2<\frac{1}{2}$ and when $\frac{1}{2}<a<1$, 
     $G^{\epsilon_1}_n \lesssim G_{\mathcal{A}_n} \lesssim G^{\epsilon_2}_n$ for some $0< \epsilon_1 <\epsilon_2<1-\frac{1}{2a}$, then we have
     \begin{equation*} 
		\lim_{n \to \infty} P(\mathcal{A}_n=\widehat{\mathcal{A}}_n)=1 \hspace*{0.2cm}\text{as} \hspace*{0.2cm} n \to \infty .
	\end{equation*}
\end{theorem}


Observe that the decision rule \eqref{eq:2.11} corresponding to an individual group $g$ depends on the whole dataset $\mathcal{D}$ and as such the rules for different $g$'s are dependent. Proofs of these results exploit certain ideas of van der Pas et al.(2014) \cite{van2014horseshoe} and Ghosh and Chakrabarti \cite{ghosh2017asymptotic}, together with some non-trivial concentration inequalities involving the central and non-central $\chi^2$ distributions to achieve the desired upper bounds to both types of error probabilities.\\

\hspace*{0.5cm} Regarding assumptions \ref{assmp-1}-\ref{assmp-3}, see Remark \ref{remark-4} above. Our other assumption is on the total (unknown) number of active groups and our result on variable selection consistency holds for different broad sparsity regimes depending on the value of $a$ in the prior on the local shrinkage coefficients.\\

\hspace*{0.05cm} Now we investigate the asymptotic estimation rate of our proposed empirical Bayesian half-thresholding estimate \eqref{eq:2.12}. We want to know whether the asymptotic distribution of the linear combination of ${\widehat{\boldsymbol{\beta}}_{\mathcal{A}, EB}}^{\text{HT}}$ is exactly the same as that of ${\widehat{\boldsymbol{\beta}}_{\mathcal{A}}}^{\text{HT}}$. Theorem \ref{Thm-6} below provides an affirmative answer. Proof of the Theorem is provided in section \ref{sec-5}. 
\begin{theorem}
    	\label{Thm-6}
     Consider the hierarchical framework of \eqref{eq:2.1} satisfying \eqref{eq:2.2}, and the empirical Bayesian half-thresholding (HT) estimator \eqref{eq:2.12} based on this.  Assume that $|\mathcal{A}| =G_{\mathcal{A}_n}$ satisfies 
     $G^{\epsilon_1}_n \lesssim G_{\mathcal{A}_n} \lesssim G^{\epsilon_2}_n$ for some $0\leq \epsilon_1 <\epsilon_2<\frac{1}{2}$. Also assume that \ref{assmpan-1} and \ref{assmpan-2} of Theorem \ref{Thm-2} and that $L(\cdot) $ in \eqref{eq:2.2} satisfies  \hyperlink{assumption1}{Assumption 1}.
     Then for any vector $\boldsymbol{\alpha}$ with $|| \boldsymbol{\alpha}||=1$, we have 
     \begin{align*}
         {\boldsymbol{\alpha}}^{\mathrm{T}} {\boldsymbol{\Sigma}^{\frac{1}{2}}_{\mathcal{A}}}(\widehat{\boldsymbol{\beta}}_{\mathcal{A},EB}^{\text{HT}}-\boldsymbol{\beta}_{\mathcal{A}}^0) \xrightarrow{d} \mathcal{N}({0}, \sigma^2), \textrm{ as } n \to \infty.
     \end{align*}
\end{theorem}

As an immediate consequence of Theorem \ref{Thm-6}, we have the following corollary. The proof of this follows using the same set of arguments as used in Theorem \ref{Thm-6} and is hence skipped.
\begin{corollary}
\label{cor-3}
Consider the situation of Theorem \ref{Thm-6} along with an orthogonal design matrix, that is, ${\mathbf{X}}^{\mathrm{T}}\mathbf{X}=n I_{p}$. Assume $|\mathcal{A}|$ is fixed and \ref{assmp-3} is satisfied and that $L(\cdot) $ satisfies  \hyperlink{assumption1}{Assumption 1}.
     Then for all $g \in \mathcal{A}$, we have     
     \begin{equation*} 
		\sqrt{n}\left({\widehat{\boldsymbol{\beta}}_{g, EB}}^{\text{HT}} - \boldsymbol{\beta}_g^0\right) \xrightarrow{d} \mathcal{N}_{m_g}( \mathbf{0}, \sigma^2 I_{m_g}) \textrm{ as } n \to \infty.
	\end{equation*}
\end{corollary}

\subsubsection{HT method based on modified Empirical Bayes Estimator of $\tau$}
\label{sec-3.2.2}
It may be noted that the variable selection consistency of the empirical Bayes version of our variable selection rule using the estimator \eqref{eq:2.11} of $\tau$ was only proved when $a >\frac{1}{2}$ in the prior \eqref{eq:2.2} for the local shrinkage parameter. However, a similar result for the case $a=\frac{1}{2}$, which, for instance, corresponds to the horseshoe prior, could not be theoretically established using the same technique. Our simulation results are very good even for $a=\frac{1}{2}$ and are indicative of the fact even in this case. So in search for an empirical Bayes estimator of $\tau$ that can be shown to have an oracle property for all $a \geq \frac{1}{2}$, we need to dig a little deeper into the basic intuition and also the technical aspects of the proof for the $a>\frac{1}{2}$ case. The empirical Bayes estimator \eqref{eq:2.11} used by us was motivated by a similar estimator of van der Pas et al. (2014) \cite{van2014horseshoe} as in \eqref{eq:2.9}. It may be noted that the term $\frac{1}{c_2G_n}\sum_{g=1}^{G_n}1\left(\frac{n\widehat{\boldsymbol{\beta}}_g^{\mathrm{T}}\mathbf{Q}_g\widehat{\boldsymbol{\beta}}_g}{\sigma^2}>c_1 \log{G_n}\right)$ may intuitively be thought of as an estimator (or an estimated lower bound) of $\frac{G_{\mathcal{A}_n}}{G_n}$, the proportion of active groups. Our proof reveals that we need to ensure $G_n {\tau_n}^{2a} [\log (\frac{1}{\tau})]^{\frac{s}{2}-1} \to 0$ as $n \to \infty$, where $s$ denotes the maximum group sizes. For the $a >\frac{1}{2}$ case, $\tau_n=\frac{G_{\mathcal{A}_n}}{G_n}$ satisfies the condition, and it is quite intuitive that the empirical Bayes version of the HT procedure using $\widehat{\tau}$ can be shown to have variable selection consistency. Intuitively, for $a=\frac{1}{2}$, by choosing $\tau_n=\frac{G_{\mathcal{A}_n}}{G_n}$ the above condition is not satisfied, but the choice $\tau_n=(\frac{G_{\mathcal{A}_n}}{G_n})^{1+\delta}$ for any $\delta>0$ works. This gives us the clue that if $\tau$ is estimated using a statistic which can be thought of as an estimator (or at least an estimated lower bound) of $(\frac{G_{\mathcal{A}_n}}{G_n})^{1+\delta}$ for some $\delta>0$, we might have the desired result. Based on this, we consider a modified version of our early estimate as follows
\begin{equation*} \label{eq:3.6}
	({\widehat{\tau}}^{\text{EB}})^{\frac{1}{2}}=\text{max}\biggl\{\frac{1}{G_n}, \frac{1}{c_2G_n}\sum_{g=1}^{G_n}1\left(\frac{n\widehat{\boldsymbol{\beta}}_g^{\mathrm{T}}\mathbf{Q}_g\widehat{\boldsymbol{\beta}}_g}{\sigma^2}>c_1 \log{G_n}\right)\biggr\}. \tag{3.6}
\end{equation*}
 Hence, our proposed modified data-adaptive decision rule is given by
 \begin{equation*} \label{eq:3.7}
\text{The }	g^{th} \text{group is considered active if } E(1-\kappa_g\mid {\widehat{\tau}}^{\text{EB}},\sigma^2,\mathcal{D}) >0.5, \textrm{ for } g=1,2,\cdots,G,\tag{3.7}
 \end{equation*}
 where ${\widehat{\tau}}^{\text{EB}}$ is defined in \eqref{eq:3.6}.
Our next theorem shows that, indeed, the empirical Bayes version of the HT procedure with the above estimate \eqref{eq:3.6} of $\tau$ enjoys variable selection consistency.
 
\begin{theorem}
\label{Thm-8}
     Consider the hierarchical framework of \eqref{eq:2.1} satisfying \eqref{eq:2.2}, and the half-thresholding (HT) rule \eqref{eq:3.7} based on an empirical Bayes estimate ${\widehat{\tau}}^{\text{EB}}$ given in \eqref{eq:3.6}. Let $\mathcal{A}_n=\{g:\boldsymbol{\beta}_{g}^{0} \neq \mathbf{0}\}$ and $\widehat{\mathcal{A}}_n=\{g:{\widehat{\boldsymbol{\beta}}_{g, EB}}^{\text{HT}} \neq \mathbf{0}\}$ denote respectively the set of truly active groups, and the set of groups declared active by the half-thresholding rule \eqref{eq:3.7}. Suppose that $L(\cdot)$ in \eqref{eq:2.2} satisfies \hyperlink{assumption1}{Assumption 1} and assumptions \ref{assmp-1}-\ref{assmp-3} hold.
     Also, assume that, for $a \geq 0.5$, $G_{A_n} =O (G^{\epsilon}_n)$ with $0<\epsilon<\frac{1}{2}$, then we have
     \begin{equation*} 
		\lim_{n \to \infty} P(\mathcal{A}_n=\widehat{\mathcal{A}}_n)=1 \hspace*{0.2cm}\text{as} \hspace*{0.2cm} n \to \infty.\end{equation*}
\end{theorem}
\begin{remark}
    The decision rule \eqref{eq:3.7} based on the modified empirical Bayes estimator of $\tau$ given in \eqref{eq:3.6} not only achieves the variable selection consistency but also the corresponding estimator attains the optimal estimation rate, that is, the statement of Theorem \ref{Thm-6} still holds just by replacing the definition of $\widehat{\tau}^{\text{EB}}$ given in \eqref{eq:2.10} with \eqref{eq:3.6} in \eqref{eq:2.12}. Hence, this also implies that our proposed decision rule \eqref{eq:3.7} with $\widehat{\tau}^{\text{EB}}$ defined in \eqref{eq:3.6} is oracle.
\end{remark}
\subsubsection{Oracle properties of the HT procedure using full Bayes approach}
\label{sec-3.2.3}
We now motivate and present the full Bayes approach as an alternative solution to the empirical Bayes approach mentioned in the previous subsection \ref{sec-3.2.1} when $G_{A_n}$ is unknown. As already proved in Theorem \ref{Thm-1}, using $\tau$ as a tuning parameter, our proposed decision rule \eqref{eq:2.7} results in consistency in variable selection. On the other hand, Theorem \ref{Thm-7} shows that the data-adaptive decision rule \eqref{eq:2.11} using an empirical Bayes estimate of $\tau$ still can figure out all the true active groups present in the model. The next obvious question is whether similar results still hold in a full Bayes treatment. To answer this question for the full Bayes approach, we assume a nondegenerate prior on $\tau$ and assume the following condition on the range of prior density of $\tau$.
{\textbf{\hypertarget{condition1}{D1:}}} $\int_{\gamma_{1n}}^{\gamma_{2n}} \pi(\tau) d \tau=1$ where $\gamma_{2n}$ satisfies $G_n \gamma_{2n} [\log (\frac{1}{\gamma_{2n}})]^{\frac{s}{2}-1}\to 0$ as $n \to \infty$, $\gamma_{1n}$ replaced by $\tau_n$ satisfies $\log (\frac{1}{\tau_n}) \asymp \log(G_n)$ and \ref{assmpan-3} such that $\frac{\gamma_{2n}}{\gamma_{1n}} \to \infty$ as $n \to \infty$.\\
\vskip 5pt
The motivation behind \hyperlink{condition1}{D1} comes from Theorems \ref{Thm-1} and \ref{Thm-2}. Note that, both \ref{assmp-4} and \ref{assmp-5} provide some asymptotic order of $\tau$ to achieve selection consistency when $\tau$ is used as a tuning parameter. This observation motivates
us to study whether the decision rule \eqref{eq:2.14} can produce selection consistency of the underlying model when a non-degenerate prior on $\tau$ satisfying such conditions is considered. Similar to Theorem \ref{Thm-6}, we are also interested in studying the asymptotic distribution of the posterior mean of the group coefficient under the full Bayes approach. Our next Theorem affirmatively answers these questions.


\begin{theorem}
\label{Thm-9}
     Consider the hierarchical framework of \eqref{eq:2.1}, and the half-thresholding (HT) rule \eqref{eq:2.11} where $\pi(\lambda^2_g)$ satisfies \eqref{eq:2.2} and  $\tau$ is assumed to have a non-degenerate prior distribution satisfying \hyperlink{condition1}{D1} above. Let $\mathcal{A}_n=\{g:\boldsymbol{\beta}_{g}^{0} \neq \mathbf{0}\}$ and $\widehat{\mathcal{A}}_n=\{g:{\widehat{\boldsymbol{\beta}}_{g, FB}}^{\text{HT}} \neq \mathbf{0}\}$ denote respectively the set of active groups, and the set of groups declared active by the half-thresholding rule \eqref{eq:2.14}. Define, $\mathbf{Q}_{n,g}={\mathbf{X}}^{\mathrm{T}}_g\mathbf{X}_g/n$, for $g=1,\cdots, G$. Then, under the assumptions \ref{assmpan-1} and \ref{assmpan-2} and that $G^{\epsilon_1}_n \lesssim G_{\mathcal{A}_n} \lesssim G^{\epsilon_2}_n$ for some $0< \epsilon_1 <\epsilon_2<\frac{1}{2}$, with $0<\epsilon<\frac{1}{2},$ 
     the decision rule \eqref{eq:2.14} is an oracle.
\end{theorem}
This result ensures that the decision rule \eqref{eq:2.14} based on any non-degenerate proper prior on $\tau$ defined in our proposed support as given in \hyperlink{condition1}{C1} can be used as an alternative solution to the empirical Bayes approaches to provide similar results both in terms of selection consistency and optimal estimation rate. It is noteworthy that if one is interested in establishing selection consistency only, one can establish that using slightly weaker assumptions \ref{assmp-1} and \ref{assmp-2} instead of \ref{assmpan-1} and \ref{assmpan-2}
and may assume that $\gamma_{1n}$ satisfies $\log (\frac{1}{\gamma_{1n}}) \asymp \log(G_n)$ only
in place of satisfying both $\log (\frac{1}{\gamma_{1n}}) \asymp \log(G_n)$ and \ref{assmpan-3}. However, we need to have stronger assumptions slightly stronger assumptions to prove selection consistency and optimal estimation rate, simultaneously. Tang et al. (2018) \cite{tang2018bayesian} also studied their proposed half-thresholding rule using a non-degenerate prior on $\tau$ supported on some suitable range based on the sample size $n$. However, the main drawback of their approach was the assumption that the number of active variables is fixed, a condition that is rarely satisfied in high-dimensional situations. On the other hand, the optimality of our rule \eqref{eq:2.14} is proved without that assumption. Hence, when the group size reduces to unity for all groups, Theorem \ref{Thm-9} confirms that our proposed rule is still an oracle without any strong assumption on the growth of the active variables.
To our knowledge, this is the first result of this kind in the full Bayes approach in the literature of global-local priors when the number of active groups grows with increasing sample size $n$. In this way, in a sparse high-dimensional group selection problem, Theorem \ref{Thm-9} establishes the fact that a carefully chosen broad class of global-local priors can provide optimal results even if the level of sparsity is unknown. 
\section{Simulations}
\label{sec-4}
In this section, we report the performance of our proposed rules in a detailed simulation study and compare that with other existing methods.
Let us simulate data from the following true model:- 
$\mathbf{y}=\sum_{g=1}^{G}\mathbf{X}_g\boldsymbol{\beta}_g++\boldsymbol{\epsilon}, \boldsymbol{\epsilon} \sim \mathcal{N}_n(\mathbf{0},\sigma^2I_n)$. The construction of the design matrix $\mathbf{X}_g$ is discussed separately for $n>p$ and $p>n$ below. The group coefficients $\boldsymbol{\beta}_g$ are either null or non-null. Several choices of null and non-null coefficients are considered in different simulation schemes. Different scenarios based on sample size (small, $n=50$, moderate, $n=200$ and large, $n=500$), number of covariates ($p$), number of groups, and sparsity levels are considered which are different based on group sizes and within group coefficients. \\

 Each group regression coefficient $\boldsymbol{\beta}_g$ is modeled by a global-local shrinkage prior as given in \eqref{eq:2.1}. A standard half Cauchy prior is used for the local shrinkage coefficient corresponding to each group, i.e. $\lambda_g \simiid C^{+}(0,1)$. Note that this choice of prior is included in the hierarchical formulation \eqref{eq:2.1} satisfying
\eqref{eq:2.2} for $a=\frac{1}{2}$. Together with \eqref{eq:2.1}, we also consider the modeling of $\boldsymbol{\beta}_g$ as $\boldsymbol{\beta}_g|  \lambda^2_g,\sigma^2, \tau^2 \simind \mathcal{N}_{m_g}(\mathbf{0},\lambda^2_g\sigma^2 \tau^2I_{m_g})$. These two formulations are
named modified group horseshoe and normal group horseshoe, respectively. When knowledge of the proportion of active groups is available, $\tau$ is used as a tuning parameter.
Theorem \ref{Thm-1} provides
one choice of $\tau$ as $\tau_n= (\frac{G_{\mathcal{A}_n}}{G_n})^{2+\delta}$ for any $\delta>0$.
Here we use $\tau_n= (\frac{G_{\mathcal{A}_n}}{G_n})^{2.1}$ i.e. $\delta=0.1$. When this knowledge is not available, we also consider empirical Bayes and full Bayes versions to estimate $\tau$.
 In case of the empirical Bayes procedure, we take $c_1=2,c_2=1$ and $\sigma$ equal to 1 in the definitions of ${\widehat{\tau}}^{\text{EB}}$, given in \eqref{eq:2.10} and \eqref{eq:3.6}. The resulting estimators are named Modified group horseshoe EB1(based on \eqref{eq:2.10}) and Modified group horseshoe EB2(based on \eqref{eq:3.6}), respectively.
 For the full Bayes procedure, we use standard half-Cauchy prior on $\tau$ ($\tau \sim C^{+}(0,1)$) which is supported on $[G^{-1.1}, G^{-1.1} \log G]$. Note that, since the simulation situations considered here are beyond the block-orthogonality assumption in the design matrix $\mathbf{X}$, we consider a group to be active or inactive using the decision rule \eqref{eq:2.5}. Also, note that the posterior mean of the group coefficient involves a choice of $\tau$, and hence different procedures mentioned above regarding the choice of $\tau$ play a crucial role. In the case of a block-diagonal design matrix, decision rule \eqref{eq:2.5} simplifies to \eqref{eq:2.7}. Similarly, we also use \eqref{eq:2.11} and \eqref{eq:2.14} for the empirical Bayes and full Bayes versions specifically for a block-diagonal design matrix.
  We are also going to use Group Spike and Diffusing prior of Yang and Narisetty (2020) \cite{yang2020consistent} (hereby named as GSD-SSS) used on the group regression coefficients and the
estimates computed from shotgun stochastic search algorithm(SSS) and Bayesian Group LASSO with Spike and Slab prior (BGL-SS) due to \cite{xu2015bayesian} for comparing the performance between one-group shrinkage prior and the two group spike and slab prior. Along with these two Bayesian approaches, we will also consider Group LASSO of Yaun and Lin (2006) \cite{yuan2006model} as a candidate for frequentist procedure.  The misclassification probability (MP), the false positive rate (FPR), and the true positive rate (TPR) corresponding to each of the aforementioned procedures will be compared.\\ 
All simulation situations considered here can be broadly classified into two cases, namely $n>p$ and $p>n$. \\
\vskip 2pt
 \textbf{Case-1:-} First we consider the cases where $n>p$. For each group, each row of the design matrix $\mathbf{X}_g$ is generated from a multivariate normal distribution such that the components have zero mean and unit variance and are correlated with
pairwise correlation $\rho$. Two values of $\rho$ are chosen $0$ and $0.5$, which indicates that the predictors within a group are uncorrelated and moderately correlated, respectively. In the following examples, different choices of $(n,p)$ along with different signal strengths and different sparsity levels are to be considered.

\begin{itemize}
	\item Example 1. We start the simulation study with a small sample size.
 Here, we consider a situation when sample size $n=50$ and $p=20$ covariates are grouped in $10$ groups containing $2$ covariates each. Regarding the group coefficients, we consider two situations based on the strength of the active coefficients. When the strength is weak, let $\boldsymbol{\beta}=(\mathbf{0},\mathbf{0},\mathbf{0},\mathbf{0.2},\mathbf{0},\mathbf{0},\mathbf{0},\mathbf{0},\mathbf{0},\mathbf{0})$.
 On the other hand, for strong signal strength, we assume $\boldsymbol{\beta}=(\mathbf{0},\mathbf{0},\mathbf{0},\mathbf{1},\mathbf{0},\mathbf{0},\mathbf{0},\mathbf{0},\mathbf{0},\mathbf{0})$. In both situations, $\mathbf{0}$ and $\mathbf{0.2}, \mathbf{1}$ are vectors of length 2, with all elements 0 or 0.2 and 1, respectively. Since the group size is $2$, each row of $\mathbf{X}_g$ is generated from a 
 $\mathcal{N}_2(0,0,1,1,\rho)$ where $\mathcal{N}_2(\cdot)$ denotes a 
 bivariate normal distribution.

 \item Example 2. Next, we consider the moderate sample size case. Towards that, we consider a framework where $n=200$ and $p=40$ covariates are grouped in 10 groups containing 4 covariates each. We assume only the first group is active. Similar to example 1, in this case too, we consider two situations based on signal strength.
 Group coefficients for these two cases are,  $\boldsymbol{\beta}=((0.1,0.2,0.3,0.4),\mathbf{0},\mathbf{0},\cdots,\mathbf{0},\mathbf{0})$ for weak signal strength, and
 $\boldsymbol{\beta}=((1.1, 1.2,1.3,1.4),\mathbf{0},\mathbf{0},\cdots,\mathbf{0},\mathbf{0})$ for strong signal strength. Here $\mathbf{0}$ is a null vector of length 4. The data generated scheme is similar to Example 1 except for necessary dimension changes.
 
 \item Example 3. Now we are interested in the case when the group sizes are different. Let us consider the scenario when $n=200$ and $p=50$ predictors are grouped in $16$ groups with group sizes 4,3,3,2,2,2,2,2,2,4,4,4,4,2,5 and 5 respectively. Let $\boldsymbol{\beta}=((0.1, 0.2, 0.3, 0.4),\mathbf{0},\mathbf{0},\mathbf{0},\mathbf{0},\mathbf{0},\mathbf{0},\mathbf{0},\mathbf{0},\mathbf{0}, (0,0.4,0,0),\mathbf{0},\mathbf{0},\mathbf{0},\mathbf{0},\mathbf{0})$ for weak signal strength and 
 $\boldsymbol{\beta}=((1, 2, 3, 4),\mathbf{0},\mathbf{0},\mathbf{0},\mathbf{0},\mathbf{0},\mathbf{0},\mathbf{0},\mathbf{0},\mathbf{0}, (1,1.1,1.2,1.3),\mathbf{0},\mathbf{0},\mathbf{0},\mathbf{0},\mathbf{0})$ for strength signal strength. In both of the cases, the first two $\mathbf{0}$ denote null vectors of length 3, the remaining are of length 2 and the last two are of length 5. In this case, also, the  predictors are generated in the same way as in
	Example 1 except for necessary dimension changes.

 \item Example 4. Now consider a situation when the sample size is large. Let sample size $n=500$ and $p=100$ covariates be grouped into 25 groups containing 4 covariates each. We assume only one group is active and the coefficients be,  $\boldsymbol{\beta}=((0.1, 0.2,0.3,0.4),\mathbf{0},\cdots,\mathbf{0},\mathbf{0})$ when the signal strength is weak and 
  $\boldsymbol{\beta}=((1.1, 1.2,1.3,1.4),\mathbf{0},\cdots,\mathbf{0},\mathbf{0})$ when the signal strength is strong. Here $\mathbf{0}$ is a null vector of length 4. The data generated scheme is similar to Example 1 except for necessary dimension changes.
\end{itemize}

For the next two examples, we consider the design matrices to be block-diagonal. Suppose $\mathbf{X}=(\mathbf{X_1},\mathbf{X_2},\cdots,\mathbf{X_G})$, which is generated as previously mentioned. From $\mathbf{X}$ a block-diagonal matrix $\mathbf{Z}=(\mathbf{Z_1},\mathbf{Z_2},\cdots,\mathbf{Z_G})$ is obtained as
\begin{align*} \label{4.1}
    \mathbf{Z_1} &= \mathbf{X_1} \\
    \mathbf{Z_2} &= (\mathbf{I_n}-P_{ \mathbf{Z_1}}) \mathbf{X_2} \\
     \mathbf{Z_3} &= (\mathbf{I_n}-P_{ \mathbf{Z_1}}-P_{ \mathbf{Z_2}}) \mathbf{X_3} \\ \tag{4.1}
     \cdots &\cdots \\
      \cdots &\cdots \\
      \mathbf{Z_G} &= (\mathbf{I_n}-P_{ \mathbf{Z_1}}-P_{ \mathbf{Z_2}}-\cdots -P_{ \mathbf{Z_{G-1}}}) \mathbf{X_G} ,
\end{align*}
where $\mathbf{I_n}$ denotes the identity matrix of order $n \times n$ and $P_{ \mathbf{Z_g}}= \mathbf{Z_g} ({\mathbf{Z}}^{\mathrm{T}}_g\mathbf{Z}_g)^{-1} {\mathbf{Z}}^{\mathrm{T}}_g$ denotes the projection matrix on the column space of $\mathbf{Z}_g$. Now, consider the situation where the data is simulated from the following true model:- 
$\mathbf{y}=\sum_{g=1}^{G}\mathbf{Z}_g\boldsymbol{\beta}_g+\boldsymbol{\epsilon}$, where $\mathbf{Z}_g$'s are generated from \eqref{4.1} and $ \boldsymbol{\epsilon} \sim \mathcal{N}_n(\mathbf{0},\sigma^2I_n)$.
\begin{itemize}
   \item  Example 5. We revisit Example 1 where the design matrix now becomes $\mathbf{Z}$ instead of $\mathbf{X}$, whose columns are generated using \eqref{4.1}.
   Here, sample size $n=50$ and $p=20$ covariates are grouped in $10$ groups containing $2$ covariates each. Let $\boldsymbol{\beta}=(\mathbf{0},\mathbf{0},\mathbf{0},\mathbf{0.2},\mathbf{0},\mathbf{0},\mathbf{0},\mathbf{0},\mathbf{0},\mathbf{0})$ where $\mathbf{0}$ and $\mathbf{0.2}$ are vectors of length 2, with all elements 0 or 0.2, respectively.
   
    \item Example 6. Now we revisit example 2 where the design matrix now becomes $\mathbf{Z}$ instead of $\mathbf{X}$, whose columns are generated using \eqref{4.1}. Here, $n=200$ and $p=40$ covariates are grouped into 10 groups containing 4 covariates each. We assume that only the first group is active and the coefficients be $\boldsymbol{\beta}=((1.1, 1.2,1.3,1.4),\mathbf{0},\mathbf{0},\cdots,\mathbf{0},\mathbf{0})$ where $\mathbf{0}$ is a null vector of length 4. 
\end{itemize}

\textbf{Case-2:-} Next we are interested in situations when $p>n$.  We define the jth predictor in group g as $X_{gj}=Z_{gj}+Z_g$, where $Z_g$ and $Z_{gj}$ are independent standard normal variates. Thus, predictors within a group are correlated with
a pairwise correlation of 0.5 while the predictors in different groups are independent. In the following examples, different choices of $(n,p)$ along with different signal strengths and different sparsity levels are to be considered.

\begin{itemize}
\item Example 7. First, we consider the case where $n = 50$ and $p =
100$. 100 predictors are grouped
into 25 groups of 4 covariates each. Let $\boldsymbol{\beta}=(\mathbf{0},\cdots,\mathbf{0},\mathbf{0.4},\mathbf{0},\cdots,\mathbf{0})$ when signal strength is weak, and
$\boldsymbol{\beta}=(\mathbf{0},\cdots,\mathbf{0},\mathbf{1.2},\mathbf{0},\cdots,\mathbf{0})$ for strong signal strength. Here the  $\mathbf{0}$ denotes a null vector of length 4.
 
    \item Example 8. Consider another example of large $p$ small $n$ problem with $n = 50$ and $p =
100$. 100 predictors are grouped
into 20 groups of 5 covariates each. For weak signal strength, the group coefficients are given by,
$\boldsymbol{\beta}=(\mathbf{0},\cdots,\mathbf{0},\mathbf{0.4},\mathbf{0.5}, (0.65,0.60,0.55,0.50,0.45),\mathbf{0},\cdots,\mathbf{0})$. On the other hand, when the signal strength is strong, we assume 
$\boldsymbol{\beta}=(\mathbf{1},\mathbf{0}, \cdots,\mathbf{0},\mathbf{1.2},\mathbf{1.4},\mathbf{0},\cdots,\mathbf{0})$. In each case, $\mathbf{0}$ denotes a null vector of length 5.

  \item Example 9. This is another example for large $p$ small $n$ problem with $n = 50$ and $p =
100$. Unlike previous situations, for the same combination of $(n,p)$, we consider two situations where the level of sparsity in the second situation is twice the former one.
Here 100 predictors are grouped
into 25 groups of 4 covariates each. First, we assume 4 groups to be active out of 25 groups. Let $\boldsymbol{\beta}=(\mathbf{0},\cdots,\mathbf{0},\mathbf{1.8},\mathbf{0.5}, (0.65,0.60,0.55,0.50),\mathbf{0},\cdots,\mathbf{0}, \mathbf{2.5})$ where the $\mathbf{0}$ denote null vector of length 4. Next, we assume 8 groups to be active out of 25 groups. Let $\boldsymbol{\beta}=(\mathbf{1.8},\mathbf{1.5},\mathbf{0},\cdots,\mathbf{0},\mathbf{1.8},\mathbf{1.5}, (0.65,0.60,0.55,0.50),\mathbf{0},\cdots,\mathbf{0}, \mathbf{2.5}, (0.4,0.45,0.50,0.55),\mathbf{2.5})$.
\end{itemize}

\subsection{Simulation Output}
\begin{table}[h!]
	\renewcommand\thetable{1}
	\caption{Mean True/False Positive Rate based on 100  replications(Example 1) } 
	
	\centering 
	\begin{tabular}{|c| c c c |c c c|} 
		\hline
  \multicolumn{7}{|c|}{Small group coefficients} \\
  \hline
		$$ &  $$ & $\rho=0$ & $$ & $$ & $\rho=0.5$ & \\ [0.5ex] 
		\hline 
	
	Prior & MP & FPR & TPR & MP & FPR & TPR \\
		\hline 
		Modified GH & 0.0162 & 0.0125 & 0.95 & 0.01226 & 0.0114 & 0.98 \\
		Usual GH  &  0.0181 & 0.0145 & 0.95 & 0.01226 & 0.0114 & 0.98 \\
		GSD-SSS & 0.0554 & 0.0115 & 0.55 &0.0499 & 0.0111 & 0.60 \\
		Modified GH(EB1) & 0.02148 & 0.0172 & 0.94 & 0.01832 & 0.0148 & 0.95 \\
  Modified GH(EB2) & 0.0172 & 0.0128 & 0.94 & 0.01498 & 0.0122 & 0.96\\
		Usual GH(EB)  & 0.0222 & 0.018 & 0.94 & 0.0206 & 0.0162 & 0.94 \\
            Group LASSO  & 0.232 & 0.250 & 0.93 & 0.231 & 0.250 & 0.94 \\
            Modified GH(FB) & 0.02013 & 0.0157 & 0.94 & 0.01859 & 0.0151 & 0.95 \\
		Usual GH(FB)  & 0.0276 & 0.0164 & 0.94 & 0.01788 & 0.0132 & 0.94 \\
  BGL-SS & 0.0799 & 0.0222 & 0.40 & 0.0994 & 0.0666 & 0.60 \\
  \hline
  \multicolumn{7}{|c|}{Large group coefficients} \\
  \hline
  $$ &  $$ & $\rho=0$ & $$ & $$ & $\rho=0.5$ & \\ [0.5ex] 
		\hline 
	Prior & MP & FPR & TPR & MP & FPR & TPR \\
		\hline 
		Modified GH & 0.01098 & 0.0122 & 1.00 & 0.0094 & 0.0105 & 1.00 \\
		Usual GH  &  0.01296 & 0.0144 & 1.00 & 0.0115 & 0.0128 & 1.00 \\
		GSD-SSS & 0.01026 & 0.0114 & 1.00 &0.0094 & 0.0104 & 1.00 \\
		Modified GH(EB1) & 0.01116 & 0.0124 & 1.00 & 0.0107 & 0.0119 & 1.00 \\
  Modified GH(EB2) &  0.01116 & 0.0124 & 1.00 & 0.0107 & 0.0119 & 1.00 \\
		Usual GH(EB)  & 0.01305 & 0.0145 & 1.00 & 0.0120 & 0.01334 & 1.00 \\
            Group LASSO  & 0.1998 & 0.22212 & 1.00 & 0.1858 & 0.20646 & 1.00 \\
            Modified GH(FB) & 0.01161 & 0.0129 & 1.00 & 0.0098 & 0.0109 & 1.00 \\
		Usual GH(FB)  & 0.01269 & 0.0141 & 1.00 & 0.01089 & 0.0121 & 1.00 \\
  BGL-SS & 0.01799 & 0.01999 & 1.00 & 0.01499 & 0.01665 & 1.00 \\
  \hline
	\end{tabular}
	\label{table:T-F-1} 
\end{table}
\begin{table}[h!]
	\renewcommand\thetable{2}
	\caption{Mean True/False Positive Rate based on 100  replications(Example 2) } 
	
	\centering 
	\begin{tabular}{|c| c c c |c c c|} 
		\hline
  \multicolumn{7}{|c|}{Small group coefficients} \\
  \hline
		$$ &  $$ & $\rho=0$ & $$ & $$ & $\rho=0.5$ & \\ [0.5ex] 
		\hline 
		Prior & MP & FPR & TPR & MP & FPR & TPR \\
		\hline 
			Modified GH& 0.0065
 & 0.005 & 0.98 & 0.00306 & 0.0034 & 1.000 \\
		Usual GH& 0.0121
 & 0.009 & 0.96 & 0.00531  & 0.0059 & 1.000 \\
		GSD-SSS & 0.0235
 & 0.005 & 0.81 &0.00712 & 0.0038 & 0.963 \\
		Modified GH(EB1)& 0.01085
 & 0.0065 & 0.95 & 0.00306 & 0.0034 & 1.000\\
 Modified GH(EB2)& 0.0084 & 0.006 & 0.97 & 0.00306 & 0.0034 & 1.000 \\
		Usual GH(EB)& 0.0113

 & 0.008 & 0.95 & 0.00531  & 0.0059 & 1.000 \\
            Group LASSO  & 0.3037
 & 0.333 & 0.96 & 0.1998 & 0.222 & 1.000 \\
 Modified GH(FB)& 0.0065
 & 0.005 & 0.98 & 0.0033 & 0.0037 & 1.000\\
		Usual GH(FB)& 0.0121

 & 0.009 & 0.96 & 0.00549  & 0.0061 & 1.000 \\
 BGL-SS & 0.0081 & 0.0062 & 0.98 & 0.00499 & 0.0055 & 1.000 \\
 \hline
  \multicolumn{7}{|c|}{Large group coefficients} \\
  \hline
  $$ &  $$ & $\rho=0$ & $$ & $$ & $\rho=0.5$ & \\ [0.5ex] 
		\hline 
		Prior & MP & FPR & TPR & MP & FPR & TPR \\
		\hline 
			Modified GH& 0.00405
 & 0.0045 & 1.00 & 0.00252 & 0.0028 & 1.000 \\
		Usual GH& 0.00441
 & 0.0049 & 1.00 & 0.00261  & 0.0029 & 1.000 \\
		GSD-SSS & 0.00432
 & 0.0048 & 1.00 &0.00279 & 0.0031 & 1.000 \\
		Modified GH(EB1)& 0.00522
 & 0.0058 & 1.00 & 0.00261 & 0.0029 & 1.000\\
 Modified GH(EB2) & 0.00459 & 0.0051 & 1.00 & 0.00252 & 0.0028 & 1.000\\
		Usual GH(EB)& 0.00648

 & 0.0072 & 1.00 & 0.00297  & 0.0033 & 1.000 \\
            Group LASSO  & 0.1998
 & 0.222 & 1.00 & 0.1998 & 0.222 & 1.000 \\
 Modified GH(FB)& 0.00459
 & 0.0051 & 1.00 & 0.00306 & 0.0034 & 1.000\\
		Usual GH(FB)& 0.00477

 & 0.0053 & 1.00 & 0.00279  & 0.0031 & 1.000 \\
 BGL-SS & 0.00499 & 0.0055 & 1.00 & 0.00297 & 0.0033 & 1.000\\
 \hline
	\end{tabular}
	\label{table:T-F-2} 
\end{table}

\begin{table}[h!]
	\renewcommand\thetable{3}
	\caption{Mean True/False Positive Rate based on 100  replications(Example 3) } 
	
	\centering 
	\begin{tabular}{|c| c c c |c c c|} 
		\hline
  \multicolumn{7}{|c|}{Small group coefficients} \\
  \hline
		$$ &  $$ & $\rho=0$ & $$ & $$ & $\rho=0.5$ & \\ [0.5ex] 
		\hline 
		Prior & MP & FPR & TPR & MP & FPR & TPR \\
		\hline 
		Modified GH & 0.0092 & 0.0034 & 0.95 & 0.0051 & 0.0029 & 0.98 \\
		Usual GH& 0.0118 & 0.0041 & 0.95 & 0.0082  & 0.0036 & 0.96 \\
		GSD-SSS & 0.026525 & 0.0046 & 0.82 & 0.01878 & 0.0029 & 0.87 \\
		Modified GH(EB1) & 0.0120 & 0.0038 & 0.93 & 0.0117 & 0.0034 & 0.93 \\
  Modified GH(EB2) & 0.0118 & 0.0035 & 0.92 & 0.0115 & 0.0031 & 0.93 \\
		Usual GH(EB) & 0.0127 & 0.0045 & 0.93 & 0.0122  & 0.0039 & 0.93 \\
             Group LASSO  & 0.1675 & 0.1812 & 0.94 & 0.1574 & 0.1691 & 0.94 \\
             Modified GH(FB) & 0.0119 & 0.0036 & 0.93 & 0.0129 & 0.0034 & 0.92 \\
		Usual GH(FB) & 0.0125 & 0.0043 & 0.93 & 0.0144  & 0.0036 & 0.91 \\
  BGL-SS & 0.0133 & 0.0038 & 0.92 & 0.0104 & 0.0034 & 0.94 \\
 \hline
  \multicolumn{7}{|c|}{Large group coefficients} \\
  \hline
  $$ &  $$ & $\rho=0$ & $$ & $$ & $\rho=0.5$ & \\ [0.5ex] 
		\hline 
		Prior & MP & FPR & TPR & MP & FPR & TPR \\
		\hline 
		Modified GH & 0.00113 & 0.0013 & 1.00 & 0.00105 & 0.0012 & 1.00 \\
		Usual GH & 0.0013125 & 0.0015 & 1.00 & 0.0013125 & 0.0015 & 1.00 \\
		GSD-SSS & 0.00113 & 0.0013 & 1.00 & 0.00113 & 0.0013 & 1.00 \\
		Modified GH(EB1) & 0.001225 & 0.0014 & 1.00 & 0.0013125 & 0.0015 & 1.00 \\
  Modified GH(EB2) & 0.00113 & 0.0013 & 1.00 & 0.00105 & 0.0012 & 1.00 \\
		Usual GH(EB) & 0.001225 & 0.0014 & 1.00 & 0.0014  & 0.0016 & 1.00 \\
             Group LASSO  & 0.147175 & 0.1682 & 1.00 & 0.13475 & 0.154 & 1.00 \\
             Modified GH(FB) & 0.0016625 & 0.0019 & 1.00 & 0.001225 & 0.0014 & 1.00 \\
		Usual GH(FB) & 0.001575 & 0.0018 & 1.00 & 0.00114  & 0.0013 & 1.00 \\
  BGL-SS & 0.00625 & 0.0072 & 1.000 & 0.003125 & 0.00357 & 1.000\\
  \hline
	\end{tabular}
	\label{table:T-F-3} 
\end{table}

\begin{table}[h!]
	\renewcommand\thetable{4}
	\caption{Mean True/False Positive Rate based on 100  replications(Example 4) } 
	
	\centering 
	\begin{tabular}{|c| c c c |c c c|} 
		\hline
  \multicolumn{7}{|c|}{Small group coefficients} \\
  \hline
		$$ &  $$ & $\rho=0$ & $$ & $$ & $\rho=0.5$ & \\ [0.5ex] 
		\hline 
		Prior & MP & FPR & TPR & MP & FPR & TPR \\
		\hline 
			Modified GH& 0.0029 & 0.0031 & 1.000 & 0.0023 & 0.0024 & 1.000 \\
		Usual GH& 0.0029 & 0.0031 & 1.000 & 0.0023  & 0.0024 & 1.000 \\
		GSD-SSS & 0.0029 & 0.0031 & 1.000 & 0.0023 & 0.0024 & 1.000 \\
		Modified GH(EB1)& 0.0035 & 0.0036 & 1.000 & 0.0029 & 0.0031 & 1.000\\
  Modified GH(EB2) & 0.0033 & 0.0034 & 1.000 & 0.0027 & 0.0028 & 1.000\\
		Usual GH(EB)& 0.0035 & 0.0036 & 1.000 & 0.0029  & 0.0031 & 1.000 \\
            Group LASSO  & 0.065 & 0.0681 & 1.000 & 0.0522 & 0.0544 & 1.000 \\
            Modified GH(FB)& 0.0029 & 0.0031 & 1.000 & 0.0024 & 0.0025 & 1.000\\
		Usual GH(FB)& 0.0031 & 0.0032 & 1.000 & 0.0024  & 0.0025 & 1.000 \\
  BGL-SS & 0.0032 & 0.0033 & 1.000 & 0.0024 & 0.0025 & 1.000 \\
 \hline
  \multicolumn{7}{|c|}{Large group coefficients} \\
  \hline  
  $$ &  $$ & $\rho=0$ & $$ & $$ & $\rho=0.5$ & \\ [0.5ex] 
		\hline 
		Prior & MP & FPR & TPR & MP & FPR & TPR \\
		\hline 
			Modified GH& 0.0027 & 0.0028 & 1.000 & 0.0023 & 0.0024 & 1.000 \\
		Usual GH& 0.0027 & 0.0028 & 1.000 & 0.0023  & 0.0024 & 1.000 \\
		GSD-SSS & 0.0027 & 0.0028 & 1.000 & 0.0023 & 0.0024 & 1.000 \\
		Modified GH(EB1)& 0.0025 & 0.0026 & 1.000 & 0.0029 & 0.0031 & 1.000\\
  Modified GH(EB2)& 0.0025 & 0.0026 & 1.000 & 0.0029 & 0.0031 & 1.000\\
		Usual GH(EB)& 0.0025 & 0.0026 & 1.000 & 0.0029  & 0.0031 & 1.000 \\
            Group LASSO  & 0.0570 & 0.0594 & 1.000 & 0.0522 & 0.0544 & 1.000 \\
            Modified GH(FB)& 0.0029 & 0.0031 & 1.000 & 0.0028 & 0.0029 & 1.000\\
		Usual GH(FB)& 0.0031 & 0.0032 & 1.000 & 0.0028  & 0.0029 & 1.000 \\
  BGL-SS & 0.0024 & 0.0025 & 1.000 & 0.0020 & 0.0021 & 1.000 \\
  \hline
	\end{tabular}
	\label{table:T-F-5} 
\end{table}

\begin{table}[h!]
	\renewcommand\thetable{5}
	\caption{Mean True/False Positive Rate based on 100  replications(Example 5) } 
	
	\centering 
	\begin{tabular}{|c| c c c |c c c|} 
		\hline
		$$ &  $$ & $\rho=0$ & $$ & $$ & $\rho=0.5$ & \\ [0.5ex] 
		\hline 
	
	Prior & MP & FPR & TPR & MP & FPR & TPR \\
		\hline 
		Modified GH & 0.0157 & 0.0119 & 0.95 & 0.01226 & 0.0114 & 0.98 \\
		Usual GH  &  0.0174 & 0.01138 & 0.95 & 0.01226 & 0.0114 & 0.98 \\
		GSD-SSS & 0.05226 & 0.0114 & 0.58 &0.0499 & 0.0111 & 0.60 \\
		Modified GH(EB1) & 0.02176 & 0.0164 & 0.93 & 0.01832 & 0.0148 & 0.95 \\
  Modified GH(EB2) &0.01825 & 0.0125 & 0.93 & 0.01771 & 0.0119& 0.93\\
		Usual GH(EB)  & 0.02302 & 0.0178 & 0.93 & 0.0206 & 0.0162 & 0.94 \\
            Group LASSO  & 0.3064 & 0.333 & 0.94 & 0.231 & 0.250 & 0.94 \\
            Modified GH(FB) & 0.01816 & 0.0124 & 0.93 & 0.01859 & 0.0151 & 0.95 \\
		Usual GH(FB)  & 0.01969 & 0.0141 & 0.93 & 0.01788 & 0.0132 & 0.94 \\
         BGL-SS & 0.075 & 0.01887 & 0.42 & 0.0601 & 0.0222 & 0.6\\
         \hline
	\end{tabular}
	\label{table:T-F-15} 
\end{table}

\begin{table}[h!]
	\renewcommand\thetable{6}
	\caption{Mean True/False Positive Rate based on 100  replications(Example 6) } 
	
	\centering 
	\begin{tabular}{|c| c c c |c c c|} 
		\hline
		$$ &  $$ & $\rho=0$ & $$ & $$ & $\rho=0.5$ & \\ [0.5ex] 
		\hline 
		Prior & MP & FPR & TPR & MP & FPR & TPR \\
		\hline 
			Modified GH& 0.00351
 & 0.0039 & 1.00 & 0.00252 & 0.0028 & 1.000 \\
		Usual GH& 0.00369
 & 0.0041 & 1.00 & 0.00261  & 0.0029 & 1.000 \\
		GSD-SSS & 0.00369
 & 0.0041 & 1.00 &0.00279 & 0.0031 & 1.000 \\
		Modified GH(EB1)& 0.00396
 & 0.0044 & 1.00 & 0.00261 & 0.0029 & 1.000\\
 Modified GH(EB2) & 0.00369 & 0.0041 & 1.00 & 0.00261 & 0.0029 & 1.000\\
		Usual GH(EB)& 0.00405

 & 0.0045 & 1.00 & 0.00297  & 0.0033 & 1.000 \\
            Group LASSO  & 0.1998
 & 0.222 & 1.00 & 0.1998 & 0.222 & 1.000 \\
 Modified GH(FB)& 0.00441
 & 0.0049 & 1.00 & 0.00306 & 0.0034 & 1.000\\
		Usual GH(FB)& 0.00468

 & 0.0052 & 1.00 & 0.00279  & 0.0031 & 1.000 \\
 BGL-SS & 0.00401 & 0.00444 & 1.00 & 0.00203 & 0.0022 & 1.000\\
 \hline
	\end{tabular}
	\label{table:T-F-16} 
\end{table}

\begin{table}[h!]
	\renewcommand\thetable{7}
	\caption{Mean True/False Positive Rate based on 100  replications(Example 7) } 
	
	\centering 
	\begin{tabular}{|c|c c c|} 
		\hline
  \multicolumn{4}{|c|}{Small group coefficients} \\
  \hline
		Prior & MP & FPR & TPR  \\
		\hline 
			Modified GH  & 0.0115 & 0.01138 & 0.986 \\
		Usual GH &  0.0118  & 0.01135 & 0.978 \\
		GSD-SSS &  0.04274 & 0.0349 & 0.769 \\
		Modified GH(EB1) &  0.0137 & 0.0133 & 0.975 \\
  Modified GH(EB2) & 0.0136 & 0.0134 & 0.981 \\
		Usual GH(EB)  &  0.0165 & 0.0157 & 0.964 \\
            Group LASSO  & 0.178 & 0.183 & 0.950 \\
            Modified GH(EB) &  0.0123 & 0.0121 & 0.984 \\
		Usual GH(EB)  &  0.0118 & 0.01135 & 0.977 \\
  BGL-SS & 0.012 & 0.01042 & 0.754\\
 \hline
  \multicolumn{4}{|c|}{Large group coefficients} \\
  \hline
  Prior & MP & FPR & TPR  \\
		\hline 
			Modified GH  & 0.0108 & 0.01131 & 1.00 \\
		Usual GH &  0.01422  & 0.01482 & 1.00 \\
		GSD-SSS &  0.0109 & 0.0114 & 1.00 \\
		Modified GH(EB1) &  0.012672 & 0.0132 & 1.00 \\
  Modified GH(EB2) & 0.0124 & 0.0129 & 1.00\\
		Usual GH(EB)  &  0.01434 & 0.0149 & 1.00 \\
            Group LASSO  & 0.120 & 0.125 & 1.00 \\
            Modified GH(FB) &  0.0114 & 0.0119 & 1.00 \\
		Usual GH(FB)  &  0.0117 & 0.0122 & 1.00 \\
  BGL-SS & 0.0112 & 0.01166 & 1.00 \\
  \hline
	\end{tabular}
	\label{table:T-F-4} 
\end{table}

\begin{table}[h!]
	\renewcommand\thetable{8}
	\caption{Mean True/False Positive Rate based on 100  replications(Example 8) } 
	
	\centering 
	\begin{tabular}{|c|c c c|} 
		\hline
 \multicolumn{4}{|c|}{Small group coefficients} \\
  \hline
	Prior & MP & FPR & TPR  \\
		\hline 
			Modified GH  & 0.0137 & 0.01404 & 0.988 \\
		Usual GH &  0.0142  & 0.01389 & 0.984 \\
		GSD-SSS &  0.04446 & 0.01384 & 0.782 \\
		Modified GH(EB1) &  0.01499 & 0.01411 & 0.980 \\
  Modified GH(EB2) & 0.01437 & 0.01408 & 0.984\\
		Usual GH(EB)  &  0.0158 & 0.01409 & 0.974 \\
            Group LASSO  & 0.18659 & 0.21511 & 0.975 \\
            Modified GH(FB) &  0.0131 & 0.01388 & 0.991 \\
		Usual GH(FB)  &  0.0137 & 0.01402 & 0.988 \\
  BGL-SS & 0.01344 & 0.01441 & 0.992\\	
  \hline
  \multicolumn{4}{|c|}{Large group coefficients} \\
  \hline
  Prior & MP & FPR & TPR  \\
		\hline 
			Modified GH  & 0.01114 & 0.01311 & 1.00 \\
		Usual GH &  0.01271  & 0.01495 & 1.00 \\
		GSD-SSS &  0.01097 & 0.01291 & 1.00 \\
		Modified GH(EB1) &  0.012061 & 0.0142 & 1.00 \\
  Modified GH(EB2) & 0.011475 & 0.0135 & 1.00\\
		Usual GH(EB)  &  0.01292 & 0.0152 & 1.00 \\
            Group LASSO  & 0.1133 & 0.133 & 1.00 \\
            Modified GH(FB) &  0.01293 & 0.01521 & 1.00 \\
		Usual GH(FB)  &  0.011798 & 0.01388 & 1.00 \\
  BGL-SS & 0.012 & 0.01417 & 1.00\\
 \hline
	\end{tabular}
	\label{table:T-F-12} 
\end{table}

\begin{table}[h!]
	\renewcommand\thetable{9}
	\caption{Mean True/False Positive Rate based on 100  replications(Example 9) } 
	
	\centering 
	\begin{tabular}{|c|c c c|} 
		\hline
  \multicolumn{4}{|c|}{Number of active groups=4} \\
  \hline
		Prior & MP & FPR & TPR  \\
  \hline	
Modified GH  & 0.0145 & 0.0204 & 0.998 \\
  Usual GH &  0.0145  & 0.0194 & 0.996 \\
		GSD-SSS &  0.0474 & 0.0183 & 0.892 \\
		Modified GH(EB1) &  0.0177 & 0.0232 & 0.994 \\
  Modified GH(EB2) & 0.01642 & 0.0218 & 0.995\\
		Usual GH(EB)  &  0.0176 & 0.0231 & 0.994 \\
            Group LASSO  & 0.1247 & 0.1764 & 0.985 \\
            Modified GH(FB) &  0.0131 & 0.0184 & 0.998 \\
		Usual GH(FB)  &  0.0139 & 0.0181 & 0.995 \\
  BGL-SS & 0.0161 & 0.0199 & 0.992 \\
		\hline 
\multicolumn{4}{|c|}{Number of active groups=8} \\
  \hline
  Prior & MP & FPR & TPR  \\
  \hline
  		Modified GH  & 0.01112 & 0.01209 & 0.994 \\
		Usual GH &  0.01146  & 0.01194 & 0.991 \\
		GSD-SSS &  0.03374 & 0.01178 & 0.851 \\
		Modified GH(EB1) &  0.01195 & 0.01213 & 0.989 \\
  Modified GH(EB2) & 0.01182 & 0.01198 & 0.989\\
		Usual GH(EB)  &  0.0126 & 0.01218 & 0.985 \\
            Group LASSO  & 0.1691 & 0.1978 & 0.981 \\
            Modified GH(FB) &  0.01166 & 0.01179 & 0.989 \\
		Usual GH(FB)  &  0.01178 & 0.01174 & 0.988 \\
  BGL-SS & 0.0125 & 0.0121 & 0.986 \\
  \hline
	\end{tabular}
	\label{table:T-F-13} 
\end{table}

\subsection{Interpretation of simulation results}
In each of the above examples, we computed the probability of misclassification (MP), the false positive rate (FPR) and the true positive rate (TPR) for each of the methods mentioned above. Few observations can be made from these tables.\\
\begin{itemize}
\item With an increase in the magnitude of correlation among the covariates within a group, MP and FPR decrease. This indicates that when covariates form a group with a nonsignificant amount of dependence among themselves, the right decision within a group for a particular regressor also influences the same for the remaining individuals forming the group. 
\item   Wang and Leng (2008) \cite{wang2008note} suspected that, like the LASSO, the group LASSO also may have the drawback of inconsistency in variable selection. Xu and Ghosh (2014) \cite{xu2015bayesian} proved this property in their paper. This is also reflected in our simulation setting, as in all cases, irrespective of the value of $\rho$, the LASSO group tends to select more variables and produces a higher FPR than the remaining methods.
    \item These tables also suggest that the Modified Group Horseshoe has slightly lower MP and FPR compared to the normal one, irrespective of the choice of $\rho$ in all cases. When the signal strength (i.e. coefficient of the active groups) is weak,
    Modified Group Horseshoe produces much better results than its counterparts in two groups due to Yang and Narisetty (2020) \cite{yang2020consistent} and Xu and Ghosh (2015) \cite{xu2015bayesian}
    in terms of MP, FPR, and TPR in examples 1-3 and 5, where the sample size is small ($n=50$) or moderate $(n=200)$.
\end{itemize}

\begin{itemize}
    \item Example 3 is different from the remaining ones, as the group size is different in this case. Although the group size is not used in the prior distribution of the group coefficients in any of these methods, our half-thresholding rule successfully captures the truth and hence produces better results than those of GSD-SSS and BGL-SS.
\end{itemize}
\begin{itemize}
    \item Example 4 shows that when the signal strength is weak, procedures due to Yang and Narisetty (2020) \cite{yang2020consistent} and  Xu and Ghosh (2015) \cite{xu2015bayesian} can produce similar results to that obtained by ours, only if the sample size is large ($n=500$).
    \item Examples 1-4 and 6 clarify that when signal strength is strong, regardless of sample size $n$, the performances of BGL-SS and GSD-SSS are comparable with those of our decision rule. 
   
    \item In examples 5 and 6, we consider two cases where the design matrix is block-diagonal. Since the data generation scheme is similar to those of Examples 1 and 2,
    from the previous results, we suspect that our method would provide better results than that of Yang and Narisetty (2020) \cite{yang2020consistent} and Xu and Ghosh (2015) \cite{xu2015bayesian} when the signal strength is weak and would produce comparable results when the signal strength is strong. Tables \ref{table:T-F-15} and \ref{table:T-F-16} confirm this.
    \item Examples 7-9 show that the rule \eqref{eq:2.5} will also work even if $p>n$. In this case, too, MP and TPR corresponding to our decision rule are much better than those of GSD-SSS and better than BGL-SS for weak signal strength and produce compared results for strong signal strength.
     \item Example 9 deals with situations when there is a mixture of weak and strong signal strengths. In these cases also, our proposed method outperforms Yang and Narisetty (2020) \cite{yang2020consistent} and yields better results than that of Xu and Ghosh (2015) \cite{xu2015bayesian}.
     \item In general, it can be unambiguously stated that our decision rule \eqref{eq:2.5} provides excellent results in terms of group selection even if the design matrix is not block-diagonal. 
\end{itemize}
    \section{Real data analysis}

    \label{sec-3.5}
    In this section, we compare the performance of our variable selection  and estimation rules with some existing methods when applied to a real dataset. 
   Here we consider two datasets, both of them are available in R. \\ 
\hspace*{0.5cm}\textbf{Diabetes dataset.} This dataset was used initially by \cite{efron2004large}, which is available in R package \textit{care}. This dataset was previously studied by \cite{tang2018bayesian}.
It contains ten baseline variables (predictors): age, sex,
body mass index (BMI), average blood pressure (BP), six blood serum
measurements (TC, LDL, HDL, TCH, LTH, GLU), and a quantitative measure
of disease progression one year after baseline (response) for 442 diabetes
patients. The baseline variables are standardized to have zero mean and
unit $l_2$ norm. The response variable is centered to have zero mean. The different
methods discussed in previous simulations were applied and the results are presented
in Table \ref{table:T-F-18}.
	\begin{table}[htbp]
		\renewcommand\thetable{10}
		\caption{Performance of different methods in Diabetes dataset} 
		
		\centering 
		\begin{tabular}{|c| c c c c c c c c 
 c c|} 
			\hline
            Method	& Age	& Gender &	Bmi	& Bp	&TC	&LDL	&HDL	&TCH	&LTH	&GLU \\
            \hline
Group SCAD	&0	&1	&1	&1	&1	&1	&0	&1	&1	&0\\
Group MCP	&0	&1	&1	&1	&1	&1	&0	&1	&1	&0\\
Group LASSO	&0	&1	&1	&1	&1	&1	&0	&1	&1	&0\\
BGL-SS	&0	&1	&1	&1	&1	&0	&1	&1	&1	&1\\
GSD-SS	&0	&1	&1	&1	&1	&0	&1	&0	&1	&0\\
Modified GH(EB)	  &0	&1	&1	&1	&1	&0	&1	&0	&1	&0\\
Modified GH(FB)	  &0	&1	&1	&1	&1	&0	&1	&0	&1	&0\\
			\hline 
			
		\end{tabular}
		\label{table:T-F-18} 
	\end{table}
It was observed that five variables (Gender, BMI, BP, HDL, LTH) are selected by all the methods.
These five variables are the first ones that enter the regression
equation in \cite{efron2004large}. Among all the methods, BGL-SS selects the
highest number of variables. The methods based on frequentist approaches also select LDL, but not HDL, On the other hand, all Bayesian approaches select HDL, in addition to five common variables, but not  choose TCH.\\
\hspace*{0.5cm} Next, we are interested in the prediction problem. For this,
the 442 observations in the dataset are divided into a training set and a test set. The
training set has 280 observations and the test set has 162 observations, as same as \cite{tang2018bayesian}. The methods used in the diabetes example are applied for the training data
and the mean squared prediction error (MSPE) is estimated based on the test
data for each method. The results are presented in Table \ref{table:T-F-19}. It ensures that in terms of estimated MSPE, our method yields results comparable to those of the existing methods in this literature.

\begin{table}[h!]
		\renewcommand\thetable{11}
		\caption{Mean squared prediction error corresponding to different methods for Diabetes dataset} 
		
		\centering 
		\begin{tabular}{|c| c|} 
			\hline
            Method & MSPE\\
			\hline 
			Group SCAD	&0.484159\\
Group MCP	&0.486948\\
Group LASSO	&0.486289\\
BGL-SS	&0.489933\\
GSD-SS	&0.484794\\
Modified GH(EB)	&0.484990\\
Modified GH(FB)& 0.485094\\
			\hline
			
		\end{tabular}
		\label{table:T-F-19} 
	\end{table}
{\textbf{Birth weight data.}} We consider the birth weight dataset from \cite{hosmer1989applied} with the group
methods, which is available in the R package \textit{grpreg}. This dataset was previously analyzed by \cite{yuan2006model}.
The birth weight dataset records the birth weights of 189 babies and 16 predictors
concerning the mother. These 16 covariates are divided into 8 groups named as mother’s age in years, 
 mother’s weight in pounds at the last menstrual period, 
mother’s
race, 
smoking status during pregnancy, 
number of previous
premature labours, 
history of hypertension 
presence of uterine
irritability,
and number of physician visits during the first trimester
. 
The data were collected at Baystate Medical Center, Springfield, Massachusetts, during
1986. For the prediction problem,
the 189 observations in the dataset are divided into a training and a test part. The
training part has 126 observations and the test part has 63 observations. The methods used in the diabetes example are applied to the training data
and the mean squared prediction error (MSPE) is calculated based on the test
data for each method. The results are presented in Table \ref{table:T-F-20}. The results indicate that in terms of estimated MSPE, our method yields significantly better result than those of the existing ones. We have observed that, for moderate birth-weights, our method has much better performance than the remaining ones. For the remaining cases, it produces comparable results.

\begin{table}[h!]
		\renewcommand\thetable{12}
		\caption{Mean squared prediction error corresponding to different methods for Birth weight dataset } 
		
		\centering 
		\begin{tabular}{|c| c|} 
			\hline
            Method & MSPE($\times 10^{-2}$) \\
			\hline 
Group SCAD	&9.668395\\
Group MCP	&9.607762\\
Group LASSO	&9.565219\\
BGL-SS	&9.564448\\
GSD-SS	&9.208861\\
Modified GH(EB)	&8.564448\\
Modified GH(FB)	&8.526835\\

			\hline
			
		\end{tabular}
		\label{table:T-F-20} 
	\end{table}

\section{Proofs}
\label{sec-5}
We first state and prove Lemmas \ref{lem1} to \ref{lem4} which are crucial to proving the main theorems of our work.
\begin{lemma}
	\label{lem1}
	Let $L$ be a nonnegative, measurable, slowly varying function defined over an interval unbounded to the right. Then the following results hold.\vspace{2mm}
	
	\begin{enumerate}[label=(\arabic*)]
	    \item \label{L-1.1}  $L^{\alpha}$ is slowly varying for all $\alpha \in \mathbb{R}$.\vspace{2mm}
	    
	\item \label{L-1.2} $\frac{\log L(x)}{\log x} \to 0$ as $x\to \infty$.\vspace{2mm}
	
	\item \label{L-1.3} For every $\alpha >0, \ x^{-\alpha} L(x) \to 0$ and $x^{\alpha}L(x) \to \infty$ as $x\to \infty$.\vspace{2mm}
	
	\item \label{L-1.4} For $ \alpha<-1,-\frac{\int_{x}^{\infty}t^{\alpha}L(t)dt}{x^{\alpha+1}L(x)} \to \frac{1}{\alpha+1}$ as $x\to \infty$.\vspace{2mm}
	
	\item \label{L-1.5} There exists a global constant $ A_0>0$ such that, for any $\alpha>-1,$ $\frac{\int_{A_0}^{x}t^{\alpha}L(t)dt}{x^{\alpha+1}L(x)} \to \frac{1}{\alpha+1}$ as $x\to \infty.$	
	\end{enumerate}

	
\end{lemma}
\begin{proof}
See Bingham et al. (1987) \cite{bingham_goldie_teugels_1987}.
\end{proof}  
\begin{lemma}
	\label{lem2}
	Let $L:(0,\infty) \to (0,\infty) $ be a measurable and integrable function such that for fixed $a>0$, $\int_{0}^{\infty} t^{-a-1}L(t)dt=K^{-1}$, with $K \in (0,\infty)$. Assume $\tau_n \to 0$ as $n \to \infty$. Then \begin{equation*}
		\int_{0}^{1}	u^{a+\frac{m_g}{2}-1}(1-u)^{-a-1}L\left(\frac{1}{\tau^2_n}(\frac{1}{u}-1)\right)du=K^{-1}(\tau^2_n)^{-a}(1+o(1)) \hspace{0.05cm},
	\end{equation*}
	where the $o(1)$ term is such that $\lim_{n \to \infty}o(1)=0.$
\end{lemma}
\begin{proof}
 The proof follows using the same set of arguments used to establish Lemma 5 of Ghosh et al. (2016) \cite{ghosh2016asymptotic}.
\end{proof}

\begin{lemma}
	\label{lem3}
Consider the hierarchical framework of \eqref{eq:2.1} where the local shrinkage parameters are modeled with the class of priors given by \eqref{eq:2.2}. Suppose $\tau_n \to 0$ as $n \to \infty$. Then for any given $a \in (0,1)$, there exists $A_0 \geq 1$ such that
\begin{equation*}
	E(1-\kappa_g\mid\tau_n,\sigma^2,\mathcal{D}) \leq  \frac{A_0K}{a(1-a)} (\tau^2_n) ^{a} L\left(\frac{1}{\tau^2_n}\right)\exp\left(\frac{n\widehat{\boldsymbol{\beta}}_g^{\mathrm{T}}Q_{n,g}\widehat{\boldsymbol{\beta}}_g}{2 \sigma^2}\right)(1+o(1)).
\end{equation*}
Assume that the slowly varying function $L(\cdot) $ satisfies  \hyperlink{assumption1}{Assumption 1} for some $a \geq 1$. Then
\begin{equation*}
	E(1-\kappa_g\mid\tau_n,\sigma^2,\mathcal{D}) \leq  \frac{KM}{a} \tau_n \exp\left(\frac{n\widehat{\boldsymbol{\beta}}_g^{\mathrm{T}}Q_{n,g}\widehat{\boldsymbol{\beta}}_g}{2 \sigma^2}\right)(1+o(1)).
\end{equation*}
The terms $o(1)$ in both inequalities above tend to zero as $n \to \infty$.
\end{lemma}
\begin{proof}
 The proof for the case $a \in (0,1)$ follows using the same set of arguments employed by Ghosh et al. (2016) \cite{ghosh2016asymptotic} to establish Theorem 4 of their paper.\newline
 
 Let us now consider the case $a \geq 1$. First note that
 
 \begin{align*}\label{eq:5.1}
     E(1-\kappa_g\mid\tau_n,\sigma^2,\mathcal{D})&=\frac{\int_{0}^{1}\kappa_g^{a+\frac{m_g}{2}-1}(1-\kappa_g)^{-a}L\left(\frac{1}{\tau^2_n}(\frac{1}{\kappa_g}-1)\right)\exp\big\{(1-\kappa_g)\cdot \frac{n\widehat{\boldsymbol{\beta}}_g^{\mathrm{T}}Q_{n,g}\widehat{\boldsymbol{\beta}}_g}{2 \sigma^2}\big\}d \kappa_g}{\int_{0}^{1}\kappa_g^{a+\frac{m_g}{2}-1}(1-\kappa_g)^{-a-1}L\left(\frac{1}{\tau^2_n}(\frac{1}{\kappa_g}-1)\right)\exp\big\{(1-\kappa_g)\cdot \frac{n\widehat{\boldsymbol{\beta}}_g^{\mathrm{T}}Q_{n,g}\widehat{\boldsymbol{\beta}}_g}{2 \sigma^2}\big\}d \kappa_g}.\tag{5.1}
 \end{align*}
 Using the transformation $s=\frac{1}{\tau^2_n}(\frac{1}{\kappa_g}-1)$ in the integrals above, we obtain
 \begin{align*}\label{eq:5.2}
     E(1-\kappa_g\mid\tau_n,\sigma^2,\mathcal{D})&=\tau^2_n \frac{\int_{0}^{\infty}(1+s\tau^2_n)^{-\frac{m_g}{2}-1}s^{-a}L(s)\exp\left(\frac{s\tau^2_n}{1+s \tau^2}\cdot \frac{n\widehat{\boldsymbol{\beta}}_g^{\mathrm{T}}Q_{n,g}\widehat{\boldsymbol{\beta}}_g}{2 \sigma^2}\right)d s}{\int_{0}^{\infty}(1+s\tau^2_n)^{-\frac{m_g}{2}}s^{-a-1}L(s)\exp\left(\frac{s\tau^2_n}{1+s \tau^2}\cdot \frac{n\widehat{\boldsymbol{\beta}}_g^{\mathrm{T}}Q_{n,g}\widehat{\boldsymbol{\beta}}_g}{2 \sigma^2}\right)d s}.\tag{5.2}
 \end{align*}
 Note that
 \begin{align*}\label{eq:5.3}
\int_{0}^{\infty}(1+s\tau^2_n)^{-\frac{m_g}{2}}s^{-a-1}L(s)\exp\left(\frac{s\tau^2_n}{1+s \tau^2_n}\cdot \frac{n\widehat{\boldsymbol{\beta}}_g^{\mathrm{T}}Q_{n,g}\widehat{\boldsymbol{\beta}}_g}{2 \sigma^2}\right)ds &\geq \int_{0}^{\infty}(1+s\tau^2_n)^{-\frac{m_g}{2}}s^{-a-1}L(s)ds\notag\\
&= K^{-1}(1+o(1)),\tag{5.3}
 \end{align*}
where the last equality above follows from the Dominated Convergence Theorem. Combining \eqref{eq:5.2} and \eqref{eq:5.3}, we obtain
 \begin{align*} \label{eq:5.4}
   &  E(1-\kappa_g\mid\tau_n,\sigma^2,\mathcal{D})   \leq K\tau^2\int_{0}^{\infty}(1+s\tau^2_n)^{-\frac{m_g}{2}-1}s^{-a}L(s)\exp\left(\frac{s\tau^2_n}{1+s \tau^2_n}\cdot \frac{n\widehat{\boldsymbol{\beta}}_g^{\mathrm{T}}Q_{n,g}\widehat{\boldsymbol{\beta}}_g}{2 \sigma^2}\right)ds (1+o(1)) \nonumber\\
   &  =K\tau^2\left(\int_{0}^{1}+  \int_{1}^{\frac{1}{\tau_n}}+\int_{\frac{1}{\tau_n}}^{\infty} \right)(1+s\tau^2_n)^{-\frac{m_g}{2}-1}s^{-a}L(s)\exp\left(\frac{s\tau^2_n}{1+s \tau^2_n}\cdot \frac{n\widehat{\boldsymbol{\beta}}_g^{\mathrm{T}}Q_{n,g}\widehat{\boldsymbol{\beta}}_g}{2 \sigma^2}\right)ds(1+o(1))\nonumber\\
     &= K(A_{1,\tau_n}+A_{2,\tau_n}+A_{3,\tau_n})(1+o(1)), \hspace{0.2cm} \text{say}.\tag{5.4}
     \end{align*}
     
Observe that for $s \in (0,1)$ and $\tau_n \in (0,1)$, $\frac{s \tau^2_n}{1+ s \tau^2_n} \leq \frac{1}{2}$. Also, $\int_{0}^{\infty} s^{-a-1} L(s) dt=K^{-1}$. Therefore it follows that 
\begin{equation*} \label{eq:5.5}
    A_{1,\tau_n} \leq K^{-1} \tau^2_n \exp \left({\frac{n\widehat{\boldsymbol{\beta}}_g^{\mathrm{T}}Q_{n,g}\widehat{\boldsymbol{\beta}}_g}{4 \sigma^2}} \right). \tag{5.5}
\end{equation*}
Likewise, for $s \in [1,\frac{1}{\tau_n})$ and $\tau_n \in (0,1)$, using the above arguments, we obtain
\begin{equation*} \label{eq:5.6}
    A_{2,\tau_n} \leq K^{-1} \tau_n \exp \left({\frac{n\widehat{\boldsymbol{\beta}}_g^{\mathrm{T}}Q_{n,g}\widehat{\boldsymbol{\beta}}_g}{4 \sigma^2}} \right). \tag{5.6}
\end{equation*}
Finally, using \ref{L-1.4} of Lemma \ref{lem1}, we have
\begin{align*} \label{eq:5.7}
    A_{3,\tau_n} &\leq \exp \left({\frac{n\widehat{\boldsymbol{\beta}}_g^{\mathrm{T}}Q_{n,g}\widehat{\boldsymbol{\beta}}_g}{2 \sigma^2}}\right) \int_{\frac{1}{\tau}}^{\infty} s^{-a-1} L(s) ds \notag\\
    &= \exp \left({\frac{n\widehat{\boldsymbol{\beta}}_g^{\mathrm{T}}Q_{n,g}\widehat{\boldsymbol{\beta}}_g}{2 \sigma^2}} \right)\frac{\tau^a_n}{a}L(\frac{1}{\tau})(1+o(1))\notag\\
    &\leq \frac{\tau_n}{a}M \exp \left({\frac{n\widehat{\boldsymbol{\beta}}_g^{\mathrm{T}}Q_{n,g}\widehat{\boldsymbol{\beta}}_g}{2 \sigma^2}}\right)(1+o(1)). \tag{5.7}
\end{align*}
Combining \eqref{eq:5.4}-\eqref{eq:5.7}, the desired result follows.
\end{proof}
\begin{lemma}
	\label{lem4}
	Consider the framework of Lemma \ref{lem3}. Then under \hyperlink{assumption1}{Assumption 1}, for any arbitrary constants $\eta \in (0,1), q \in (0,1)$ and any fixed $\tau >0$,
	\begin{equation*}
		P(\kappa_g>\eta|\tau,\sigma^2,\mathcal{D}) \leq \frac{(a+\frac{m_g}{2})(1-\eta q)^a}{{\tau}^{2a} {(\eta q)}^{a+\frac{m_g}{2}}C_0} \exp\left(-\frac{n\widehat{\boldsymbol{\beta}}_g^{\mathrm{T}}Q_{n,g}\widehat{\boldsymbol{\beta}}_g \eta (1-q)}{2 \sigma^2}\right) \hspace{0.05cm}.
	\end{equation*}
\end{lemma}
\begin{proof}
The proof follows using a similar set of arguments used by Ghosh et al. (2016) \cite{ghosh2016asymptotic} to establish Theorem 5 in their paper.
\end{proof}
	
{\textbf{Proof of Proposition \ref{prop-1}:}}
\begin{proof} 
First, we consider the case when $a \in (0,1)$. 
Using Lemma \ref{lem3}, we obtain
\begin{align*}\label{eq:5.8}
	E(1-\kappa_g\mid\tau_n,\sigma^2,\mathcal{D})&	\leq \frac{A_0K}{a(1-a)}   (\tau^2_n) ^{a} L(\frac{1}{\tau^2_n})\exp\left(\frac{n\widehat{\boldsymbol{\beta}}_g^{\mathrm{T}}Q_{n,g}\widehat{\boldsymbol{\beta}}_g}{2 \sigma^2}\right)(1+o(1)) \hspace{0.05cm}.\tag{5.8}
\end{align*}
When $\tau_n \to 0$ as $n \to \infty$, using Part \ref{L-1.3} of Lemma \ref{lem1},
\begin{align*}\label{eq:5.9}
	\lim_{n \to \infty} (\tau^2_n) ^{a} L(\frac{1}{\tau^2_n}) &=\lim_{n \to \infty}(\frac{1}{\tau^2_n})^{-a}L(\frac{1}{\tau^2_n})=0.\tag{5.9}
\end{align*}
Under a block-orthogonal design, from the standard theory of linear regression, the distribution of the ordinary least square estimator $\widehat{\boldsymbol{\beta}}_g$ is given by
$$\sqrt{n}\left(\widehat{\boldsymbol{\beta}}_g - \boldsymbol{\beta}_g^0\right) \sim \mathcal{N}_{m_g}(\mathbf{0},\sigma^2\mathbf{Q}_{n,g}^{-1}).$$
Clearly, if $\boldsymbol{\beta}_g^0=\mathbf{0}$, $\sqrt{n}\widehat{\boldsymbol{\beta}}_g \sim \mathcal{N}_{m_g}(\mathbf{0},\sigma^2\mathbf{Q}_{n,g}^{-1})$. Therefore,  $\frac{n\widehat{\boldsymbol{\beta}}_g^{\mathrm{T}}\mathbf{Q}_{n,g}\widehat{\boldsymbol{\beta}}_g}{\sigma^2} \sim \chi^2_{m_g},$ whence
\begin{align*}\label{eq:5.10}
\frac{n\widehat{\boldsymbol{\beta}}_g^{\mathrm{T}}\mathbf{Q}_{n,g}\widehat{\boldsymbol{\beta}}_g}{\sigma^2}=O_p(1), \textrm{ for all } n.\tag{5.10}   
\end{align*}
Combining \eqref{eq:5.8} - \eqref{eq:5.10}, and using Slutsky's Theorem, it readily follows that
$$E(1-\kappa_g\mid\tau_n,\sigma^2,\mathcal{D}) \xrightarrow{P} 0 \textrm { as } n \rightarrow \infty.$$
Next, we consider the case when case $a \geq 1$. Observe that the upper bound to $E(1-\kappa_g\mid\tau_n,\sigma^2,\mathcal{D})$ is similar to the upper bound when $a \in (0,1)$. Hence, the proof follows using the same set of arguments as in the case $a \in (0,1)$.
\end{proof}

\textbf{Proof of Proposition \ref{prop-2}:}
\begin{proof}
It would be enough to show that 
$E(\kappa_g\mid\tau_n,\sigma^2,\mathcal{D}) \xrightarrow{P} 0$ as $n \to \infty$ when $\boldsymbol{\beta}_{g}^{0} \neq \mathbf{0}$.\newline

Let us fix $\epsilon_0 >0.$ Then
\begin{align*}\label{eq:5.11}
	E(\kappa_g\mid\tau_n,\sigma^2,\mathcal{D}) &= \int_{0}^{\frac{\epsilon_0}{2}}\kappa_g \pi(\kappa_g|\tau_n,\sigma^2,\mathcal{D})d \kappa_g+\int_{\frac{\epsilon_0}{2}}^{1}\kappa_g \pi(\kappa_g|\tau_n,\sigma^2,\mathcal{D})d \kappa_g \\
	& \leq \frac{\epsilon_0}{2}+P(\kappa_g>\frac{\epsilon_0}{2}|\tau_n,\sigma^2,\mathcal{D}) \hspace{0.05cm}.\tag{5.11}
\end{align*}
Therefore, $ \textrm{ for a given } \epsilon_0 >0$,
\begin{align*}\label{eq:5.12}
P\left(E(\kappa_g\mid\tau_n,\sigma^2,\mathcal{D})>\epsilon_0\right) \leq P\left(P(\kappa_g>\frac{\epsilon_0}{2}|\tau_n,\sigma^2,\mathcal{D})>\frac{\epsilon_0}{2}\right).\tag{5.12}    
\end{align*}
Now, substituting $\eta=\frac{\epsilon_0}{2}$ in Lemma \ref{lem4}, some simple algebra yields
\begin{align*}\label{eq:5.13}
P(E(\kappa_g\mid\tau_n,\sigma^2,\mathcal{D})>\epsilon_0) &\leq P\left(\frac{(a+\frac{m_g}{2})(1-\eta q)^a}{{\tau_n}^{2a} {(\eta q)}^{a+\frac{m_g}{2}}C_0} \exp\left(-\frac{n\widehat{\boldsymbol{\beta}}_g^{\mathrm{T}}Q_{n,g}\widehat{\boldsymbol{\beta}}_g \eta (1-q)}{2 \sigma^2}\right)>\frac{\epsilon_0}{2}\right) \\
&=P\left(\frac{n\widehat{\boldsymbol{\beta}}_g^{\mathrm{T}}\mathbf{Q}_{n,g}\widehat{\boldsymbol{\beta}}_g}{\sigma^2}<d_n\right),\tag{5.13}
\end{align*}
where 
\begin{equation*}
	d_n=\frac{4 }{\epsilon_0(1-q)}\bigg[{d^{\prime}}+a \cdot {\log \left(\frac{1}{\tau^2_n}\right)}\bigg] \hspace{0.05cm},
\end{equation*}
$d^{\prime}$ being a constant is independent of $n$. 

Observe now that using \ref{assmp-1}, we have
\begin{align*}\label{eq:5.21} P\left(\frac{n\widehat{\boldsymbol{\beta}}_g^{\mathrm{T}}\mathbf{Q}_{n,g}\widehat{\boldsymbol{\beta}}_g}{\sigma^2}<d_n \right) 
&\leq \sum_{j=1}^{m_g} P \left(\frac{\sqrt{n}|\widehat{\beta}_{gj}|}{\sigma\sqrt{\zeta_j}} 
\leq \frac{1}{\sqrt{C_1}}\cdot\sqrt{\frac{d_n}{\zeta_j}}\right) \\
&= \sum_{j=1}^{m_g} P\left(-\frac{1}{\sqrt{C_1}}\cdot\sqrt{\frac{d_n}{\zeta_j}}-\frac{\sqrt{n}\beta^{0}_{gj}}{\sigma\sqrt{\zeta_j}}\leq \frac{\sqrt{n}(\widehat{\beta}_{gj}-\beta^{0}_{gj})}{\sigma\sqrt{\zeta_j}}\leq \frac{1}{\sqrt{C_1}}\cdot\sqrt{\frac{d_n}{\zeta_j}}-\frac{\sqrt{n}\beta^{0}_{gj}}{\sigma\sqrt{\zeta_j}}\right),\tag{5.14} 
\end{align*}
where $\zeta_j$ is the $j^{th}$ diagonal element of ${\mathbf{Q}^{-1}_{n,g}}$. Now the cases $\min_j \beta^{0}_{gj}>0$ and $\min_j \beta^{0}_{gj}<0$ are dealt separately as follows.\\
\textbf{Case-(I):} Assume $\min_{j} \beta^{0}_{gj}>0$. Using \eqref{eq:5.21}, we have
\begin{align*}
P\left(\frac{n\widehat{\boldsymbol{\beta}}_g^{\mathrm{T}}\mathbf{Q}_{n,g}\widehat{\boldsymbol{\beta}}_g}{\sigma^2}<d_n \right) & \leq \sum_{j=1}^{m_g}\left[ 1- \Phi \left(\frac{\sqrt{n}\beta^{0}_{gj}}{\sigma\sqrt{\zeta_j}}-\frac{1}{\sqrt{C_1}}\cdot\sqrt{\frac{d_n}{\zeta_j}} \right)\right] \\
 & \leq \sum_{j=1}^{m_g} \frac{\phi \left(\frac{\sqrt{n}\beta^{0}_{gj}}{\sigma\sqrt{\zeta_j}}-\frac{1}{\sqrt{C_1}}\cdot\sqrt{\frac{d_n}{\zeta_j}} \right)}{\frac{\sqrt{n}\beta^{0}_{gj}}{\sigma\sqrt{\zeta_j}}-\frac{1}{\sqrt{C_1}}\cdot\sqrt{\frac{d_n}{\zeta_j}}}, 
\end{align*}
 The inequality in the last line follows using the fact that $1-\Phi(t) \leq \frac{\phi(t)
}{t}, \textrm{ for any } t>0$.\newline
Under assumption \ref{assmp-2}, we have
\begin{equation*}\label{eq:5.22}
 \min_{j} \beta^{0}_{gj} > m_n \textrm{ for all } g \in \mathcal{A}  \mbox{ with } m_n \propto n^{-b} \mbox{ and } 0\leq b<\frac{1}{2}.\tag{5.15}
\end{equation*}
Next we use the fact $d_n \asymp \log(\frac{1}{\tau_n})$ and the assumption \ref{assmp-6}.
Hence, we obtain $\frac{\sqrt{d_n}}{\sqrt{n}\beta^{0}_{gj}} \to 0$ as $n \to \infty$.\newline
Therefore, we have
\begin{equation*} \label{eq:5.23}
    \frac{\sqrt{n}\beta^{0}_{gj}}{\sigma\sqrt{\zeta_j}}-\frac{1}{\sqrt{C_1}}\cdot\sqrt{\frac{d_n}{\zeta_j}} = \frac{\sqrt{n}\beta^{0}_{gj}}{\sigma\sqrt{\zeta_j}}(1+o(1)) \hspace{0.05cm}, \tag{5.16}
\end{equation*}
where $o(1)$ term tends to zero as $n \to \infty$.
Since, $\frac{1}{\sqrt{C_1}}\cdot\sqrt{\frac{d_n}{\zeta_j}}=o\left(\frac{\sqrt{n}\beta^{0}_{gj}}{\sigma\sqrt{\zeta_j}}
\right)$, we have, for $\textrm{ for any } \epsilon>0$, 
\begin{equation*}\label{eq:5.24}
    \frac{\sqrt{n d_n}\beta^{0}_{gj}}{\sigma \zeta_j \sqrt{C_1}} < \epsilon \frac{n {\beta^{0}_{gj}}^2}{\sigma^2 \zeta_j}\tag{5.17}
\end{equation*}
for sufficiently large $n$.
Let $e_1,e_2,\cdots,e_{m_g}$ be the eigenvalues of $\mathbf{Q}_{n,g}$. Then $\sum_{j=1}^{m_g}\frac{1}{e_j}=\sum_{j=1}^{m_g}\zeta_j$. Next note that, under \ref{assmp-1}, $e_j >C_1$ for all $j=1,2,\cdots,m_g$.
This implies, $\zeta_j$ is bounded above for all $j=1,2,\cdots,m_g$, i.e. for all $j=1,2,\cdots,m_g$, $\zeta_j <C_3$ for some $0<C_3<\infty$. 
Combining \eqref{eq:5.21} - \eqref{eq:5.24} and using \ref{assmp-3} with the above observation, it follows
\begin{equation*}\label{eq:5.25}
P\left(\frac{n\widehat{\boldsymbol{\beta}}_g^{\mathrm{T}}\mathbf{Q}_{n,g}\widehat{\boldsymbol{\beta}}_g}{\sigma^2}<d_n\right) \leq \frac{\sigma s \sqrt{C_3}}{\sqrt{n}m_n}\exp\bigg\{-(\frac{1}{2}-\epsilon)\frac{n{m^2_n}}{2\sigma^2 C_3}\bigg\}.\tag{5.18}
\end{equation*}
Using \eqref{eq:5.22} and choosing $\epsilon\in (0,\frac{1}{2})$ yields,
\begin{equation*}\label{eq:5.26a}  P\left(\frac{n\widehat{\boldsymbol{\beta}}_g^{\mathrm{T}}\mathbf{Q}_{n,g}\widehat{\boldsymbol{\beta}}_g}{\sigma^2}<d_n\right)=o(1), \textrm{ as } n\rightarrow\infty.\tag{5.19} 
\end{equation*}

\textbf{Case-(II):} Assume $\min_{j} 
 \beta^{0}_{gj}<0$. Again using \eqref{eq:5.21} we have
\begin{align*} \label{eq:5.27a}
P\left(\frac{n\widehat{\boldsymbol{\beta}}_g^{\mathrm{T}}\mathbf{Q}_{n,g}\widehat{\boldsymbol{\beta}}_g}{\sigma^2}<d_n \right) & \leq \sum_{j=1}^{m_g} \bigg[ \Phi \left(\frac{1}{\sqrt{C_1}}\cdot\sqrt{\frac{d_n}{\zeta_j}} -\frac{\sqrt{n}\beta^{0}_{gj}}{\sigma\sqrt{\zeta_j}} \right) -\Phi \left(-\frac{1}{\sqrt{C_1}}\cdot\sqrt{\frac{d_n}{\zeta_j}} -\frac{\sqrt{n}\beta^{0}_{gj}}{\sigma\sqrt{\zeta_j}} \right) \bigg] \\ 
 &= \sum_{j=1}^{m_g} \bigg[ P(Z >a_{1n,j})-P(Z>a_{2n,j})\bigg], \text{ say } \tag{5.20}
\end{align*}
where $a_{1n,j}=-\frac{1}{\sqrt{C_1}}\cdot\sqrt{\frac{d_n}{\zeta_j}} -\frac{\sqrt{n}\beta^{0}_{gj}}{\sigma\sqrt{\zeta_j}} $ and $a_{2n,j}=  \frac{1}{\sqrt{C_1}}\cdot\sqrt{\frac{d_n}{\zeta_j}} -\frac{\sqrt{n}\beta^{0}_{gj}}{\sigma\sqrt{\zeta_j}}$ and $Z \sim \mathcal{N}(0,1)$. It is well known
\begin{equation*}
    \frac{1}{\sqrt{2\pi}} e^{-\frac{z^2}{2}} (\frac{1}{z}-\frac{1}{z^3}) \leq P(Z>z) \leq  \frac{1}{\sqrt{2\pi} z} e^{-\frac{z^2}{2}},
\end{equation*}
Using this it follows from \eqref{eq:5.27a} that
\begin{align*} \label{eq:5.27b}
P\left(\frac{n\widehat{\boldsymbol{\beta}}_g^{\mathrm{T}}\mathbf{Q}_{n,g}\widehat{\boldsymbol{\beta}}_g}{\sigma^2}<d_n \right) & \leq  \sum_{j=1}^{m_g}  \frac{1}{\sqrt{2 \pi}} \bigg[\frac{1}{a_{1n,j}} \exp(-\frac{a^2_{1n,j}}{2})-(\frac{1}{a_{2n,j}}-\frac{1}{a^3_{2n}})\exp(-\frac{a^2_{2n,j}}{2}) \bigg] \\
& \leq  \sum_{j=1}^{m_g}  \frac{1}{\sqrt{2 \pi}} \bigg[\frac{1}{a_{1n,j}} \exp(-\frac{a^2_{1n,j}}{2})+(\frac{1}{a_{2n,j}}+\frac{1}{a^3_{2n,j}})\exp(-\frac{a^2_{2n,j}}{2}) \bigg] \\
& \leq  \sum_{j=1}^{m_g} \bigg[ \frac{3}{\sqrt{2 \pi}} \cdot \frac{1}{a_{1n,j}} \exp(-\frac{a^2_{1n,j}}{2})\bigg], \tag{5.21}
\end{align*}
where inequality in the last line holds due to the use of $a_{1n,j} \leq a_{2n,j}$. Note that using \ref{assmp-2} and the assumption \ref{assmp-6}, we have
$-\frac{\sqrt{d_n}}{\sqrt{n}\beta^{0}_{gj}} \to 0$ as $n \to \infty$.\newline

Therefore, we have
\begin{equation*} \label{eq:5.27c}
   a_{1n,j}= -\frac{\sqrt{n}\beta^{0}_{gj}}{\sigma\sqrt{\zeta_j}}-\frac{1}{\sqrt{C_1}}\cdot\sqrt{\frac{d_n}{\zeta_j}} =- \frac{\sqrt{n}\beta^{0}_{gj}}{\sigma\sqrt{\zeta_j}}(1+o(1)) \hspace{0.05cm}, \tag{5.22}
\end{equation*}
where $o(1)$ term tends to zero as $n \to \infty$.
Hence, using arguments similar to that of \eqref{eq:5.24} and \eqref{eq:5.25}, we obtain 
\begin{equation*}\label{eq:5.27d}  P\left(\frac{n\widehat{\boldsymbol{\beta}}_g^{\mathrm{T}}\mathbf{Q}_{n,g}\widehat{\boldsymbol{\beta}}_g}{\sigma^2}<d_n\right)=o(1), \textrm{ as } n\rightarrow\infty.\tag{5.23} 
\end{equation*}
Hence combining \eqref{eq:5.13}, \eqref{eq:5.26a} and \eqref{eq:5.27d}, the proof of Proposition \ref{prop-2} is complete. \\


\end{proof}

\textbf{Proof of Theorem \ref{Thm-1}:}\\
\begin{proof}
First, we observe that
\begin{align*}\label{eq:5.17}
	P(\mathcal{A}_n \neq \widehat{\mathcal{A}}_n) \leq \sum_{g \in \mathcal{A}} P(E(1-\kappa_g\mid\tau_n,\sigma^2,\mathcal{D}) < \frac{1}{2})+\sum_{g \notin \mathcal{A}}P(E(1-\kappa_g\mid\tau_n,\sigma^2,\mathcal{D})>\frac{1}{2}).\tag{5.24}
\end{align*}
To prove this result, it suffices to show
\begin{align*}\label{eq:5.18}
\sum_{g \in \mathcal{A}} P(E(1-\kappa_g\mid\tau_n,\sigma^2,\mathcal{D}) < \frac{1}{2})=o(1), \textrm{ as } n\rightarrow\infty,\tag{5.25}    
\end{align*}
and
\begin{align*}\label{eq:5.19}
\sum_{g \notin \mathcal{A}}P(E(1-\kappa_g\mid\tau_n,\sigma^2,\mathcal{D})>\frac{1}{2})=o(1), \textrm{ as } n\rightarrow\infty,\tag{5.26}    
\end{align*}
both when $\frac{1}{2} \leq a<1$, and $a\geq 1$.
\\
\textbf{Proof of \eqref{eq:5.18}:-}
 Fix an arbitrary $\epsilon_0>0$. Now, using the arguments employed in the proof of proposition \ref{prop-2} and applying  Lemma \ref{lem4} with $\eta=\frac{\epsilon_0}{2}$, we have, 
\begin{align*}\label{eq:5.20}
\sum_{g \in \mathcal{A}} P(E(\kappa_g\mid\tau_n,\sigma^2,\mathcal{D})>\frac{1}{2}) 
&\leq \sum_{g \in \mathcal{A}} P\left(\frac{n\widehat{\boldsymbol{\beta}}_g^{\mathrm{T}}\mathbf{Q}_{n,g}\widehat{\boldsymbol{\beta}}_g}{\sigma^2}<d_n\right).\tag{5.27}
\end{align*}
Since $G_{{A}_n} \leq {n}$, it follows from \eqref{eq:5.25} that
\begin{equation*}
   \sum_{g \in \mathcal{A}}  P\left(\frac{n\widehat{\boldsymbol{\beta}}_g^{\mathrm{T}}\mathbf{Q}_{n,g}\widehat{\boldsymbol{\beta}}_g}{\sigma^2}<d_n\right)\leq  \frac{\sigma s \sqrt{C_3}  \sqrt{n}}{m_n}\exp\bigg\{-(\frac{1}{2}-\epsilon)\frac{n{m^2_n}}{2\sigma^2 C_3}\bigg\}(1+o(1))\rightarrow 0, \textrm{ as } n\rightarrow\infty, 
\end{equation*}
whence
\begin{equation*}\label{eq:5.26}
   \sum_{g \in \mathcal{A}}  P\left(\frac{n\widehat{\boldsymbol{\beta}}_g^{\mathrm{T}}\mathbf{Q}_{n,g}\widehat{\boldsymbol{\beta}}_g}{\sigma^2}<d_n\right)=o(1), \textrm{ as } n\rightarrow\infty.\tag{5.28} 
\end{equation*}
Hence, \eqref{eq:5.20} coupled with \eqref{eq:5.26} completes the proof of \eqref{eq:5.18}.\\

 \textbf{Proof of \eqref{eq:5.19}:-}

\textbf{Case (I):} First consider the case when $a \in [\frac{1}{2},1)$. Using Lemma \ref{lem3} and our previous arguments, it follows that for all $g \notin \mathcal{A}$, 
\begin{align*}\label{eq:5.27}
	P\left(E(1-\kappa_g\mid\tau_n,\sigma^2,\mathcal{D})>\frac{1}{2}\right) 
&\leq P\left(\frac{n\widehat{\boldsymbol{\beta}}_g^{\mathrm{T}}\mathbf{Q}_{n,g}\widehat{\boldsymbol{\beta}}_g}{ \sigma^2}>M_n\right)(1+o(1)),\tag{5.29}
\end{align*}
where $M_n=2 \log \left(\frac{C_4}{(\tau^2_n) ^{a} L(\frac{1}{\tau^2_n})}\right)$ and $C_4$ is a global constant that is independent of $n$. In \eqref{eq:5.27}, the $o(1)$ is such that it is independent of any specific group $g$, and $\lim_{n\rightarrow\infty}o(1)=0$.\newline

Now observe that for all $g \notin \mathcal{A},\frac{n\widehat{\boldsymbol{\beta}}_g^{\mathrm{T}}\mathbf{Q}_{n,g}\widehat{\boldsymbol{\beta}}_g}{\sigma^2} \sim \chi^2_{m_g}.$\\
 We consider the two
cases $m_g = 1$, and $m_g > 1$ separately. \\
For $m_g =1$, we use
\begin{align*}
    P\left(\frac{n\widehat{\boldsymbol{\beta}}_g^{\mathrm{T}}\mathbf{Q}_{n,g}\widehat{\boldsymbol{\beta}}_g}{ \sigma^2}>M_n\right) 
&= P(|Z| > \sqrt{M_n}) \\
& = 
\sqrt{\frac{2}{\pi}} e^{-\frac{M_n}{2}} M^{-\frac{1}{2}}_n(1+o(1)),
\end{align*}
where $Z$ denotes a standard normal random variable and the last line follows due to Mill's ratio. \\
On the other hand, for $m_g \geq 2$, we first observe that 
 a $\chi^2_{m_g}$ distribution can equivalently be regarded as a $\textrm{Gamma}(\frac{m_g}{2},\frac{1}{2})$ distribution, with shape parameter $\frac{m_g}{2}$, and scale parameter $\frac{1}{2}$. Then we have
 \begin{align*}\label{eq:5.28}
P\left(\frac{n\widehat{\boldsymbol{\beta}}_g^{\mathrm{T}}\mathbf{Q}_{n,g}\widehat{\boldsymbol{\beta}}_g}{ \sigma^2}>M_n\right) 
&= \frac{1}{2^{\frac{m_g}{2}}\Gamma(\frac{m_g}{2})}\int_{M_n}^{\infty}e^{-\frac{u}{2}}u^{\frac{m_g}{2}-1}du\\
&= \frac{1}{\Gamma(\frac{m_g}{2})}\int_{M_n/2}^{\infty}e^{-u}u^{\frac{m_g}{2}-1}du, \tag{5.30}
 \end{align*}
where $\Gamma(r)=\int_{0}^{\infty}e^{-u}u^{r-1}du$ denotes the gamma function evaluated at $r>0$.\newline

Now, we state below a result due to Gabcke (2015) \cite{gabcke2015} that is instrumental in completing the remainder of this proof. This is presented as Lemma \ref{lem5} below.
\begin{lemma}\label{lem5}
 When $r\geq 1$ and $c> r+1$,
 \begin{equation*}
 e^{-c}c^{r-1}\leq \int_{c}^{\infty}e^{-u}u^{r-1}du\leq re^{-c}c^{r-1},\notag 
 \end{equation*}
 that is, for sufficiently large $c>0$,
 \begin{equation*}
  \int_{c}^{\infty}e^{-u}u^{r-1}du\lesssim re^{-c}c^{r-1}.\notag   
 \end{equation*}
\end{lemma}
Thus, using Lemma \ref{lem5} coupled with the \eqref{eq:5.28} and the fact that $M_n\rightarrow\infty$ as $n\rightarrow\infty$, we have, for all sufficiently large $n$, for all $g \notin \mathcal{A}$,
 \begin{align*}\label{eq:5.29}
e^{-\frac{M_n}{2}}M_n^{\frac{m_g}{2}-1}\leq P\left(\frac{n\widehat{\boldsymbol{\beta}}_g^{\mathrm{T}}\mathbf{Q}_{n,g}\widehat{\boldsymbol{\beta}}_g}{ \sigma^2}>M_n\right)\lesssim e^{-\frac{M_n}{2}}M_n^{\frac{s}{2}-1},\tag{5.31}
 \end{align*}
 where $s= \sup_{n \geq 1}\max_{g\in \{1,2,\cdots,G_n\}}m_g$, is finite. 
 Using this observation, and combining \eqref{eq:5.27}-\eqref{eq:5.29}, we have,
 \begin{align*}\label{eq:5.30}
	\sum_{g \notin \mathcal{A}}P\left(E(1-\kappa_g\mid\tau_n,\sigma^2,\mathcal{D})>\frac{1}{2}\right) & \lesssim  G_n(\tau^2_n) ^{a} L(\frac{1}{\tau^2_n}) \bigg[\log( \frac{1}{(\tau^2_n) ^{a} L(\frac{1}{\tau^2_n})})\bigg]^{\frac{s}{2}-1}.
 \tag{5.32}
\end{align*}
Hence, for $a \in [\frac{1}{2},1)$ using \hyperlink{assumption1}{Assumption 1} on $L(\cdot)$, the term of the right-hand side of \eqref{eq:5.30} converges to $0$, as $n\rightarrow \infty$ if $G_n{\tau_n}[\log (\frac{1}{\tau_n})]^{\frac{s}{2}-1} \rightarrow 0$ as $n \to \infty$. Note that \ref{assmp-4} and \ref{assmp-5} imply $G_n{\tau_n}\big[\log (\frac{1}{\tau_n})\big]^{\frac{s}{2}-1} \to 0$ as $n \to \infty$. This completes the proof of \eqref{eq:5.19} when $\frac{1}{2} \leq a<1$. \newline

\textbf{Case (II):} Now we consider the situation $a \geq 1$. Using similar arguments employed to prove \textbf{Case (I)}, one can easily verify that there exists a constant $C_5$ independent of $n$, such that $M_n= 2 \log (\frac{C_5}{\tau_n})$ and 
\begin{align*}\label{eq:5.31}
 \sum_{g \notin \mathcal{A}} P(E(1-\kappa_g\mid\tau_n,\sigma^2,\mathcal{D})>\frac{1}{2}) & \lesssim G_n{\tau_n}\bigg[\log (\frac{1}{\tau_n})\bigg]^{\frac{s}{2}-1}.\tag{5.33}
\end{align*}
Again observe \ref{assmp-4} and \ref{assmp-5} imply $G_n{\tau_n}\big[\log (\frac{1}{\tau_n})\big]^{\frac{s}{2}-1} \to 0$ as $n \to \infty$, for $a \geq 1$. 
Hence under the assumption of \ref{assmp-4} and \ref{assmp-5}, the right hand side of \eqref{eq:5.31} goes to $0$ as $n\rightarrow \infty$, for each fixed $a \geq 1$, which establishes \eqref{eq:5.19}. 
This completes the proof of Theorem \ref{Thm-1}.\newline
\end{proof}
\textbf{Proof of Theorem \ref{Thm-2}:}\\
\begin{proof}
Define $T= {\boldsymbol{\alpha}}^{\mathrm{T}} {\boldsymbol{\Sigma}^{\frac{1}{2}}_{\mathcal{A}}}(\widehat{\boldsymbol{\beta}}_{\mathcal{A}}^{\text{HT}}-\boldsymbol{\beta}_{\mathcal{A}}^0)$. Then $T=T_1+T_2$, where $T_1={\boldsymbol{\alpha}}^{\mathrm{T}} {\boldsymbol{\Sigma}^{\frac{1}{2}}_{\mathcal{A}}}(\widehat{\boldsymbol{\beta}}_{\mathcal{A}}-\boldsymbol{\beta}_{\mathcal{A}}^0)$ and $T_2={\boldsymbol{\alpha}}^{\mathrm{T}} {\boldsymbol{\Sigma}^{\frac{1}{2}}_{\mathcal{A}}}(\widehat{\boldsymbol{\beta}}_{\mathcal{A}}^{\text{HT}}-\widehat{\boldsymbol{\beta}}_{\mathcal{A}})$. Now, it boils down to show that, 
\begin{equation*} \label{eq:5.32}
    T_1  \xrightarrow{d} \mathcal{N}({0}, \sigma^2), \textrm{ as } n \to \infty, \tag{5.34} 
\end{equation*}
and
\begin{equation*} \label{eq:5.33}
    T_2 \xrightarrow{P}0 \textrm{ as } n \to \infty. \tag{5.35}
\end{equation*}
First, we prove \eqref{eq:5.32}. Note that, due to the block-diagonal property of the design matrix $\mathbf{X}$, $\widehat{\boldsymbol{\beta}}_{\mathcal{A}}={\boldsymbol{\Sigma}^{-1}_{\mathcal{A}}}{\mathbf{X}^{\mathrm{T}}_{\mathcal{A}}} \mathbf{y}$ and
using the standard theory of linear models, it readily follows that for $||\boldsymbol{\alpha}||=1$,
\begin{equation*}
    T_1 \sim \mathcal{N}(0,\sigma^2).
\end{equation*}

For the time being, let us assume \eqref{eq:5.33} to be true. Then, combining \eqref{eq:5.32} and \eqref{eq:5.33}, coupled with Slustky's Theorem, the desired asymptotic normality of  $\widehat{\boldsymbol{\beta}}_{\mathcal{A}}^{\text{HT}}$ follows.\newline

We now turn our focus on establishing \eqref{eq:5.33} above. 
Towards that, using Cauchy-Schwarz inequality, we have
\begin{align*} \label{eq:5.34}
    T^2_2 & \leq {(\widehat{\boldsymbol{\beta}}_{\mathcal{A}}^{\text{HT}}-\widehat{\boldsymbol{\beta}}_{\mathcal{A}})}^{\mathrm{T}} {\boldsymbol{\Sigma}_{\mathcal{A}}}(\widehat{\boldsymbol{\beta}}_{\mathcal{A}}^{\text{HT}}-\widehat{\boldsymbol{\beta}}_{\mathcal{A}}) \\
    & \leq C_2 n {(\widehat{\boldsymbol{\beta}}_{\mathcal{A}}^{\text{HT}}-\widehat{\boldsymbol{\beta}}_{\mathcal{A}})}^{\mathrm{T}} (\widehat{\boldsymbol{\beta}}_{\mathcal{A}}^{\text{HT}}-\widehat{\boldsymbol{\beta}}_{\mathcal{A}}) \\
    & = C_2 \sum_{g \in \mathcal{A}}\norm{\sqrt{n}(\widehat{\boldsymbol{\beta}}_g^{\text{HT}} - \widehat{\boldsymbol{\beta}}_g^{})}^2_2 . \tag{5.36}
\end{align*}
The second inequality in \eqref{eq:5.34} follows due to the assumption on the eigenvalues of $\mathbf{X}^{\mathrm{T}}\mathbf{X}$ as given in \ref{assmpan-1} and using the fact that $e_{max}(\mathbf{X}_{\mathcal{A}}^{\mathrm{T}}\mathbf{X}_{\mathcal{A}}) \leq e_{max}(\mathbf{X}^{\mathrm{T}}\mathbf{X})$.\\

Next, observe that using the form of the posterior mean $\widehat{\boldsymbol{\beta}}_g^{\text{PM}}$ as given by \eqref{eq:2.6} coupled with the definition of the half-thresholding estimator $\widehat{\boldsymbol{\beta}}_g^{\text{HT}}$ given by \eqref{eq:2.8}, one may rewrite the difference $\sqrt{n}(\widehat{\boldsymbol{\beta}}_g^{\text{HT}} - \widehat{\boldsymbol{\beta}}_g^{})$ as
 \begin{align*}\label{eq:5.35}
  \sqrt{n}(\widehat{\boldsymbol{\beta}}_g^{\text{HT}} - \widehat{\boldsymbol{\beta}}_g^{})
       &= \sqrt{n} \bigg[ E(1-\kappa_g\mid\tau_n,\sigma^2,\mathcal{D}) I\bigg\{E(1-\kappa_g\mid\tau_n,\sigma^2,\mathcal{D}) > 0.5\bigg\}-1\bigg]\widehat{\boldsymbol{\beta}}_g\\  
       &= -\sqrt{n} \widehat{\boldsymbol{\beta}}_g E(\kappa_g\mid\tau_n,\sigma^2,\mathcal{D})- \sqrt{n} \widehat{\boldsymbol{\beta}}_g^{} E(1-\kappa_g\mid\tau_n,\sigma^2,\mathcal{D})I\bigg\{E(1-\kappa_g\mid\tau_n,\sigma^2,\mathcal{D}) \leq 0.5\bigg\}. \tag{5.37} 
 \end{align*}
Note that
\begin{equation*}\label{eq:5.36}
     E(1-\kappa_g\mid\tau_n,\sigma^2,\mathcal{D}) \leq 0.5  \textrm{ if and only if } E(\kappa_g\mid\tau_n,\sigma^2,\mathcal{D}) \geq 0.5.\tag{5.38}
 \end{equation*}
Thus,
\begin{equation*}\label{eq:5.37}
 0\leq E(1-\kappa_g\mid\tau_n,\sigma^2,\mathcal{D})I\bigg\{E(1-\kappa_g\mid\tau_n,\sigma^2,\mathcal{D}) \leq 0.5\bigg\}  \leq E(\kappa_g\mid\tau_n,\sigma^2,\mathcal{D}),\tag{5.39}
\end{equation*}
whence
\begin{equation*}\label{eq:5.38}
  \lVert \sqrt{n} \widehat{\boldsymbol{\beta}}_g^{} E(1-\kappa_g\mid\tau_n,\sigma^2,\mathcal{D})I\bigg\{E(1-\kappa_g\mid\tau_n,\sigma^2,\mathcal{D}) \leq 0.5\bigg\} \rVert  \leq  \lVert \sqrt{n} \widehat{\boldsymbol{\beta}}_g^{} E(\kappa_g\mid\tau_n,\sigma^2,\mathcal{D}) \rVert.\tag{5.40}  
\end{equation*}
Now to establish \eqref{eq:5.33}, let us first define the following random variables:
$$W_{n,g}=\frac{n\widehat{\boldsymbol{\beta}}_g^{\mathrm{T}}\mathbf{Q}_{n,g}\widehat{\boldsymbol{\beta}}_g}{\sigma^2}, \textrm{ and } U_{n,g}= W_{n,g} E(\kappa^2_g\mid\tau_n,\sigma^2,\mathcal{D}).$$ 
Combining \eqref{eq:5.34} - \eqref{eq:5.38} together with the triangle inequality for the $\ell_2$ norm, we obtain 
\begin{align*} \label{eq:5.39}
    T^2_2 & \leq 4 \sum_{g \in \mathcal{A}} n\widehat{\boldsymbol{\beta}}_g^{\mathrm{T}}\widehat{\boldsymbol{\beta}}_g E(\kappa^2_g\mid\tau_n,\sigma^2,\mathcal{D}) \\
    & \leq \frac{4 \sigma^2}{ C_1} \sum_{g \in \mathcal{A}} \frac{n\widehat{\boldsymbol{\beta}}_g^{\mathrm{T}}\mathbf{Q}_{n,g}\widehat{\boldsymbol{\beta}}_g}{\sigma^2} E(\kappa^2_g\mid\tau_n,\sigma^2,\mathcal{D}) \\
    &= \frac{4 \sigma^2}{ C_1} \sum_{g \in \mathcal{A}} U_{ng}.  \tag{5.41}
\end{align*}
The second inequality above follows using
\ref{assmpan-1} and the fact that $\textrm{ for all g } \in \mathcal{A}, e_{min}(\frac{\mathbf{X}_g^{\mathrm{T}}\mathbf{X}_g}{n}) \geq e_{min} (\frac{\mathbf{X}^{\mathrm{T}}\mathbf{X}}{n})$. Thus to prove \eqref{eq:5.33}, it is enough to show that 
\begin{equation*} \label{eq:5.40}
     \sum_{g \in \mathcal{A}} U_{ng}  \xrightarrow{P} 0 \hspace*{0.2cm} \text{as} \hspace*{0.2cm}  n \to \infty.
    \tag{5.42}  
\end{equation*}

Using the above definitions coupled with \eqref{eq:3.1}, we obtain
\begin{align*} \label{eq:5.41}
    U_{n,g} &=  W_{n,g} \frac{\int_{0}^{1}\kappa^2_g \cdot \kappa_g^{(a+\frac{m_g}{2}-1)}(1-\kappa_g)^{-a-1}L\left(\frac{1}{\tau^2_n}(\frac{1}{\kappa_g}-1)\right)\exp\left(-\kappa_g\cdot \frac{ W_{n,g} }{2}\right) d \kappa_g }{\int_{0}^{1}\kappa_g^{(a+\frac{m_g}{2}-1)}(1-\kappa_g)^{-a-1}L\left(\frac{1}{\tau^2_n}(\frac{1}{\kappa_g}-1)\right)\exp\left(-\kappa_g\cdot \frac{ W_{n,g} }{2}\right)d \kappa_g} \\
    &= J(W_{n,g},\tau) \hspace*{0.05cm},\text{ say}\hspace*{0.05cm}.\tag{5.43}
\end{align*}
Next, we follow the arguments used in Lemma 3 of Ghosh and Chakrabarti (2017) \cite{ghosh2017asymptotic} to find an upper bound to $J(W_{n,g},\tau)$.
First note that, given any $c>2$, we can find $\eta,q \in (0,1)$ such that $c=\frac{2}{\eta(1-q)}$. Following Ghosh and Chakrabarti (2017) \cite{ghosh2017asymptotic}
there exists a non-negative measurable function $h(W_{n,g},\tau)=h_1(W_{n,g},\tau)+h_2(W_{n,g},\tau)$, where $h_1(W_{n,g},\tau)=C_{*} [W_{n,g} \int_{0}^{\frac{W_{n,g}}{1+t_0}} e^{-\frac{u}{2}} u^{m_g+a-1} du]^{-1}$ and
$h_2(W_{n,g},\tau)=C^{*} W_{n,g} {\tau}^{-2a} e^{-\frac{\eta (1-q)}{2} W_{n,g}}$ where $t_0$ is as in \hyperlink{assumption1}{Assumption 1} and $C_{*}$ and $C^{*}$ are two constants which depends on $a,\eta,q,L(\cdot)$ and satisfies: \newline
for any $W_{n,g}$, 
\begin{equation*}\label{eq:5.42}
    J(W_{n,g},\tau) \leq h(W_{n,g},\tau), \tag{5.44}
\end{equation*}
and we also have for any $\rho >c$,
\begin{equation*} \label{eq:5.43}
    \lim_{\tau \to 0} \sup_{W_{n,g}> 2a \rho \frac{1}{\sqrt{\tau} \log(\frac{1}{\tau})} } h(W_{n,g},\tau)=0.\tag{5.45}
\end{equation*}

Let $\epsilon>0$ be given. Then
\begin{align*}
    P(\sum_{g \in \mathcal{A}} U_{ng} >\epsilon) & \leq \sum_{g \in \mathcal{A}} P( U_{ng} > \frac{\epsilon}{|A|}) .
\end{align*}

Let us fix some $c>2$ and any $\rho>c$. Let $B_n$ and $C_n$ denote the events $\{ U_{n,g}> \frac{\epsilon}{|A|} \}$ and $\{W_{n,g}> 2a \rho \frac{1}{\sqrt{\tau} \log(\frac{1}{\tau})} \}$, respectively. Then,
\begin{align*}\label{eq:5.44}
  \sum_{g \in \mathcal{A}}  P(U_{n,g}>\frac{\epsilon}{|A|} )&= \sum_{g \in \mathcal{A}} P(B_n) \\
    &= \sum_{g \in \mathcal{A}} P(B_n \cap C_n)+ \sum_{g \in \mathcal{A}} P(B_n \cap C^{c}_n) \\
    & \leq \sum_{g \in \mathcal{A}} P(B_n \cap C_n)+ \sum_{g \in \mathcal{A}} P( C^{c}_n).\tag{5.46}
\end{align*}
Using \eqref{eq:5.41} and \eqref{eq:5.42} along with the form of $h_k(W_{n,g},\tau), k=1,2$, it follows that for sufficiently large $n$, under the assumption \ref{assmpan-3},
\begin{align*} \label{eq:5.45}
    \sup_{W_{n,g}> 2a \rho \frac{1}{\sqrt{\tau} \log(\frac{1}{\tau})} } h_1(W_{n,g},\tau) & \leq \sqrt{\tau} \log(\frac{1}{\tau}) =o(\frac{1}{|\mathcal{A}|}). \tag{5.47}
\end{align*}
Also note that
\begin{align*} \label{eq:5.46}
    \sup_{W_{n,g}> 2a \rho \frac{1}{\sqrt{\tau} \log(\frac{1}{\tau})} } h_2(W_{n,g},\tau) & \leq \frac{\tau^{-2a-\frac{1}{2}}}{\log (\frac{1}{\tau})} e^{- \frac{a\rho \eta (1-q)}{\sqrt{\tau} \log (\frac{1}{\tau})}} =o(\frac{1}{|\mathcal{A}|}). \tag{5.48}
\end{align*}
Observe that, with the definition of $B_n$, and using \eqref{eq:5.45} and \eqref{eq:5.46} we have, for sufficiently large $n$, 
\begin{equation*} \label{eq:5.47}
\textrm{ for all } g \in \mathcal{A}, P(B_n \cap C_n)=0.\tag{5.49} 
\end{equation*}
This implies the first term in the right-hand side of \eqref{eq:5.44} goes to zero and we are left with the second term only. \\
 For the second term, under the assumption of
 \ref{assmpan-2} and \ref{assmpan-3}, using similar set of arguments used in \eqref{eq:5.21}-\eqref{eq:5.26}, we can show that
\begin{align*} \label{eq:5.48}
\lim_{n\rightarrow\infty} \sum_{g \in \mathcal{A}}  P( C^{c}_n)&=0.   \tag{5.50}
\end{align*}
Since $\epsilon >0$ is arbitrary, combining \eqref{eq:5.41}-\eqref{eq:5.48} ensures \eqref{eq:5.40} holds. This completes the proof of Theorem \ref{Thm-2}.
\end{proof}

\textbf{Proof of Theorem \ref{Thm-7}:}\\
\begin{proof}
    As noted before,
\begin{align*}\label{eq:5.75}
	P(\mathcal{A}_n \neq \widehat{\mathcal{A}}_n) \leq \sum_{g \in \mathcal{A}} P(E(1-\kappa_g\mid
{\widehat{\tau}}^{\text{EB}},\sigma^2,\mathcal{D}) < \frac{1}{2})+\sum_{g \notin \mathcal{A}}P(E(1-\kappa_g\mid{\widehat{\tau}}^{\text{EB}},\sigma^2,\mathcal{D})>\frac{1}{2}).\tag{5.51}
\end{align*}
To prove \eqref{eq:5.75}, it suffices to show 
\begin{align*}\label{eq:5.76}
\sum_{g \in \mathcal{A}} P(E(1-\kappa_g\mid{\widehat{\tau}}^{\text{EB}},\sigma^2,\mathcal{D}) < \frac{1}{2})=o(1), \textrm{ as } n\rightarrow\infty,\tag{5.52}    
\end{align*}
and
\begin{align*}\label{eq:5.77}
\sum_{g \notin \mathcal{A}}P(E(1-\kappa_g\mid{\widehat{\tau}}^{\text{EB}},\sigma^2,\mathcal{D})>\frac{1}{2})=o(1), \textrm{ as } n\rightarrow\infty,\tag{5.53}    
\end{align*}
both when $0.5< a<1$, and $a\geq 1$. \\
Note that, for fixed $\mathcal{D}=\{\mathbf{y}\}$ and $\sigma^2$, $E(\kappa_g\mid\tau,\sigma^2,\mathcal{D})$ is a non-increasing function of $\tau$. Moreover, $\widehat{\tau}^{\text{EB}}\geq \gamma_n$, where $\gamma_n=\frac{1}{G_n}$, for $n\geq 1$. Combining these two facts, we obtain
 \begin{align*} \label{eq:5.78}
\sum_{g \in \mathcal{A}}  P(E(\kappa_g\mid\widehat{\tau}^{\text{EB}},\sigma^2,\mathcal{D})> \frac{1}{2})
	& \leq \sum_{g \in \mathcal{A}} P(E(\kappa_g\mid\gamma_n,\sigma^2,\mathcal{D})> \frac{1}{2}). \tag{5.54}
\end{align*}
Observe that, the sequence $\gamma_n=G^{-1}_n$ for $n\geq 1$ satisfies $\log (\frac{1}{\tau_n}) \asymp \log(G_n)$. Therefore, under assumptions \ref{assmp-1}-\ref{assmp-3} of Theorem \ref{Thm-1}, using exactly the same set of arguments employed in proving \textbf{Part-I} of Theorem  \ref{Thm-1} (when $\tau$ was a tuning parameter), one has
\begin{equation} \label{eq:5.79}
     \lim_{n\rightarrow \infty}\sum_{g \in \mathcal{A}} P(E(\kappa_g\mid\gamma_n,\sigma^2,\mathcal{D})> \frac{1}{2})=0. \tag{5.55}
\end{equation}
Therefore, \eqref{eq:5.78} and \eqref{eq:5.79} together yield
\begin{equation}
     \lim_{n\rightarrow \infty}\sum_{g \in \mathcal{A}} P(E(\kappa_g\mid\widehat{\tau}^{\text{EB}},\sigma^2,\mathcal{D})> \frac{1}{2})=0,\notag
\end{equation}
which completes the proof of \eqref{eq:5.76}.
Now, we are left to prove \eqref{eq:5.77}. Note that, for any $\alpha_n>0$,
 \begin{align*}\label{eq:5.50}
  P(E(1-\kappa_g\mid\widehat{\tau}^{\text{EB}},\sigma^2,\mathcal{D})> \frac{1}{2})
  &=  P(E(1-\kappa_g\mid\widehat{\tau}^{\text{EB}},\sigma^2,\mathcal{D})> \frac{1}{2}, \ {\widehat{\tau}}^{\text{EB}} \leq 2\alpha_{n}) \ +\notag\\
  &  \qquad P(E(1-\kappa_g\mid\widehat{\tau}^{\text{EB}},\sigma^2,\mathcal{D})> \frac{1}{2}, \ {\widehat{\tau}}^{\text{EB}}> 2\alpha_{n}).\tag{5.56}
 \end{align*} 
To complete the proof, we now appropriately choose $\{\alpha_{n}\}_{n\geq 1}>0$ such that $\alpha_n \to 0$ as $n \to \infty$ so that both 
\begin{align*}\label{eq:5.51}
   \sum_{g \notin \mathcal{A}}P(E(1-\kappa_g\mid{\widehat{\tau}}^{\text{EB}},\sigma^2,\mathcal{D})>\frac{1}{2}, {\widehat{\tau}^{\text{EB}}}> 2\alpha_{n})=o(1), \textrm{ as } n\rightarrow\infty \tag{5.57}
\end{align*}
and
\begin{align*}\label{eq:5.49}
\sum_{g \notin \mathcal{A}}P(E(1-\kappa_g\mid{\widehat{\tau}}^{\text{EB}},\sigma^2,\mathcal{D})>\frac{1}{2}, {\widehat{\tau}^{\text{EB}}} \leq 2\alpha_{n})=o(1), \textrm{ as } n\rightarrow\infty.\tag{5.58}
\end{align*}
For studying \eqref{eq:5.51}, we define

$$\widehat{\tau}_1=\frac{1}{G_n}, \textrm{ and } {\widehat{\tau}_2}=\frac{1}{c_2G_n}\sum_{g=1}^{G_n}1\bigg\{\frac{n\widehat{\boldsymbol{\beta}}_g^{\mathrm{T}}\mathbf{Q}_{n,g}\widehat{\boldsymbol{\beta}}_g}{\sigma^2}>c_1 \log{G_n}\bigg\},$$
where $c_1\geq 2$, and $c_2\geq 1$.\newline

Clearly,
$${\widehat{\tau}^{\text{EB}}}=\max\big\{\widehat{\tau}_1,{\widehat{\tau}_2}\big\}.$$
Therefore we have,
\begin{align*}\label{eq:5.57}
  P(E(1-\kappa_g\mid\widehat{\tau}^{\text{EB}},\sigma^2,\mathcal{D})> \frac{1}{2}, \ {\widehat{\tau}^{\text{EB}}}> 2\alpha_{n})
  &\leq P({\widehat{\tau}^{\text{EB}}}> 2\alpha_{n})\\
 &\leq P(\widehat{\tau}_1> 2\alpha_{n})+P({\widehat{\tau}_2}> 2\alpha_{n}).\tag{5.59}
\end{align*}
We note that if $\alpha_n>0$ is such that
\begin{align*}\label{eq:5.58}
    \frac{1}{G_n} \leq 2\alpha_n, \textrm{ for all sufficiently large } n,\tag{5.60}
\end{align*}
whence
\begin{align*}\label{eq:5.59}
 P(\widehat{\tau}_1> 2\alpha_{n})=0, \textrm{ for all sufficiently large } n.\tag{5.61}   
\end{align*}
Thus, \eqref{eq:5.57} coupled with \eqref{eq:5.59} yields
\begin{align*}\label{eq:5.60}
  P(E(1-\kappa_g\mid\widehat{\tau}^{\text{EB}},\sigma^2,\mathcal{D})> \epsilon_0, \ {\widehat{\tau}^{\text{EB}}}> 2\alpha_{n}) &\leq P({\widehat{\tau}_2}> 2\alpha_{n}),  \tag{5.62} 
\end{align*}
for all sufficiently large $n$ for $\alpha_n$ satisfying \eqref{eq:5.58}.\newline
Let us define
$${\widehat{\tau}_3}=\frac{1}{c_2G_n}\sum_{g\in\mathcal{A}}1\bigg\{\frac{n\widehat{\boldsymbol{\beta}}_g^{\mathrm{T}}\mathbf{Q}_{n,g}\widehat{\boldsymbol{\beta}}_g}{\sigma^2}>c_1 \log{G_n}\bigg\},$$
and 
$${\widehat{\tau}_4}=\frac{1}{c_2G_n}\sum_{g\notin\mathcal{A}}1\bigg\{\frac{n\widehat{\boldsymbol{\beta}}_g^{\mathrm{T}}\mathbf{Q}_{n,g}\widehat{\boldsymbol{\beta}}_g}{\sigma^2}>c_1 \log{G_n}\bigg\},$$
so that $${\widehat{\tau}_2}={\widehat{\tau}_3}+{\widehat{\tau}_4}.$$
Clearly,
\begin{align*}\label{eq:5.61}
P({\widehat{\tau}_2}> 2\alpha_{n})\leq P({\widehat{\tau}_3}> \alpha_{n})+P({\widehat{\tau}_4}> \alpha_{n}). \tag{5.63}   
\end{align*}

Observe that
\begin{align*}\label{eq:5.62}
{\widehat{\tau}_3}=\frac{1}{c_2G_n}\sum_{g\in\mathcal{A}}1\bigg\{\frac{n\widehat{\boldsymbol{\beta}}_g^{\mathrm{T}}\mathbf{Q}_{n,g}\widehat{\boldsymbol{\beta}}_g}{\sigma^2}>c_1 \log{G_n}\bigg\} \leq \frac{1}{c_2G_n}\sum_{g\in\mathcal{A}}1=\frac{G_{A_n}}{c_2G_n}.\tag{5.64}   
\end{align*}
We now observe that if $\alpha_n>0$ is chosen such that
\begin{align*}\label{eq:5.63}
 \frac{G_{A_n}}{c_2G_n} \leq \alpha_n,  \textrm{ for all sufficiently large } n, \tag{5.65} 
\end{align*}
then
\begin{align*}\label{eq:5.64}
 P({\widehat{\tau}_3}> \alpha_{n})=0,  \textrm{ for all sufficiently large } n. \tag{5.66}   
\end{align*}
Note condition \eqref{eq:5.63} implies condition \eqref{eq:5.58} and therefore \eqref{eq:5.59} is automatically satisfied for this choice of $\alpha_n$.
For bounding the second term in the right-hand side of \eqref{eq:5.61}, we consider the generalized version of Chernoff-Hoeffding bound for independent but non i.i.d. sequence of random variables, which is stated in the next lemma.
\begin{lemma} \label{lem6}
    Let $Z_1, Z_2, \cdots, Z_m$ be $m$ independent $0-1$ random variables with $\mathbb{E}(Z_i)=p_i$, $i=1,2,\cdots,m$. Let $Z=\sum_{i=1}^{m}Z_i$, $\mu=\mathbb{E}(Z)=\sum_{i=1}^{m}p_i$ and $p=\frac{\mu}{m}$. Then
    \begin{equation*}
        \mathbb{P}(Z \geq \mu+\lambda) \leq \exp{ -\{mH_p(p+\frac{\lambda}{m})\}}, \hspace*{0.3cm} \text{for} \hspace*{0.3cm} 0<\lambda<m-\mu,
    \end{equation*}
    and 
    \begin{equation*}
        \mathbb{P}(Z \leq \mu-\lambda) \leq \exp{ -\{mH_{1-p}(1-p+\frac{\lambda}{m})\}}, \hspace*{0.3cm} \text{for} \hspace*{0.3cm} 0<\lambda<\mu,
    \end{equation*}
    where $H_p(x)=x \log (\frac{x}{p})+(1-x) \log (\frac{1-x}{1-p})$ is the relative entropy of $x$ w.r.t. $p$.
\end{lemma}
Note that
\begin{align*}\label{eq:5.65}
 P({\widehat{\tau}_2}> 2\alpha_{n})
 &\leq P({\widehat{\tau}_4}> \alpha_{n})\\
 &= P\bigg[\sum_{g\notin\mathcal{A}}1\bigg\{\frac{n\widehat{\boldsymbol{\beta}}_g^{\mathrm{T}}\mathbf{Q}_{n,g}\widehat{\boldsymbol{\beta}}_g}{\sigma^2}>c_1 \log{G_n}\bigg\}>c_2\alpha_n G_n\bigg]\\
 &= P\bigg[\frac{1}{(G_n-G_{A_n})}\sum_{g\notin\mathcal{A}}1\bigg\{\frac{n\widehat{\boldsymbol{\beta}}_g^{\mathrm{T}}\mathbf{Q}_{n,g}\widehat{\boldsymbol{\beta}}_g}{\sigma^2}>c_1 \log{G_n}\bigg\} > c_2 \frac{\alpha_n G_n}{(G_n-G_{A_n})}\bigg] \\ 
& \leq P\bigg[\frac{1}{(G_n-G_{A_n})}\sum_{g\notin\mathcal{A}}1\bigg\{\frac{n\widehat{\boldsymbol{\beta}}_g^{\mathrm{T}}\mathbf{Q}_{n,g}\widehat{\boldsymbol{\beta}}_g}{\sigma^2}>c_1 \log{G_n}\bigg\} \   \geq \alpha_n\bigg] \\
&= P \bigg[\sum_{g\notin\mathcal{A}}1\bigg\{\frac{n\widehat{\boldsymbol{\beta}}_g^{\mathrm{T}}\mathbf{Q}_{n,g}\widehat{\boldsymbol{\beta}}_g}{\sigma^2}>c_1 \log{G_n}\bigg\} \   \geq \alpha_n(G_n-G_{A_n}) \bigg] \\
&= P\bigg[S_n - \mathbb{E}(S_n) \geq \alpha_n(G_n-G_{A_n}) -  \mathbb{E}(S_n) \bigg ]
\ \textrm{ say,}\tag{5.67}
\end{align*}
where the second inequality holds due to $c_2 G_n>G_n-G_{A_n}$ and
\begin{align*}
S_n=\sum_{g\notin\mathcal{A}}1\bigg\{\frac{n\widehat{\boldsymbol{\beta}}_g^{\mathrm{T}}\mathbf{Q}_{n,g}\widehat{\boldsymbol{\beta}}_g}{\sigma^2}>c_1 \log{G_n}\bigg\}.    
\end{align*}
To find an upper bound for \eqref{eq:5.65}, we use Lemma \ref{lem6} by taking $m=G_n-G_{A_n}$, $Z_i= 1\bigg\{\frac{n\widehat{\boldsymbol{\beta}}_g^{\mathrm{T}}\mathbf{Q}_{n,g}\widehat{\boldsymbol{\beta}}_g}{\sigma^2}>c_1 \log{G_n}\bigg\}$, 
$\mathbb{E}(Z_i)= p_i= P\left({\frac{n\widehat{\boldsymbol{\beta}}_g^{\mathrm{T}}\mathbf{Q}_{n,g}\widehat{\boldsymbol{\beta}}_g}{\sigma^2}>c_1 \log{G_n}}\right) \leq m_g G^{-\frac{c_1}{2}}_n (\log G_n)^{\frac{m_g}{2}-1}$, $\mu \leq s G^{-\frac{c_1}{2}+1}_n (\log G_n)^{\frac{s}{2}-1}$, and $p=\frac{\mu}{G_n-G_{A_n}} \leq sG^{-\frac{c_1}{2}}_n (\log G_n)^{\frac{s}{2}-1}$ , where $s$ denotes the maximum of the group size. For $\lambda=\alpha_n(G_n-G_{A_n}) -\mu$, we have $0<\lambda<m-\mu$. Also, we have $p+\frac{\lambda}{G_n-G_{\mathcal{A}_n}}=\alpha_n$.
Hence,
\begin{align*}\label{eq:5.66}
    H_p(p+\frac{\lambda}{G_n-G_{A_n}}) &= \alpha_n \log (\frac{\alpha_n}{p}) +(1-\alpha_n) \log (\frac{1-\alpha_n}{1-p}).\tag{5.68}
\end{align*}
Also, note that, since $c_1 \geq 2$, $\frac{p}{\alpha_n} \to 0$ as $n \to \infty$ under the assumption that $G^{\epsilon_1}_n \lesssim G_{\mathcal{A}_n} \lesssim G^{\epsilon_2}_n$ for some $0< \epsilon_1 <\epsilon_2<\frac{1}{2}$. \\
Recall that, $\frac{\log (\frac{1}{1-y})}{y} \to 1$ as $y \to 0$. Hence, with $y=\frac{p-\alpha_n}{1-\alpha_n}$, the second term in the right hand side of \eqref{eq:5.66} is of the form
\begin{align*}\label{eq:5.67}
     \log (\frac{1-\alpha_n}{1-p}) &= \frac{p-\alpha_n}{1-\alpha_n} (1+o(1)), \tag{5.69}
\end{align*}
where $o(1)$ depends only on $n$ such that $\lim_{n \to \infty} o(1)=0$. Hence, using \eqref{eq:5.66} and \eqref{eq:5.67}, an lower bound of $H_p(\alpha_n)$ is given by
\begin{align*}\label{eq:5.68}
    H_p(\alpha_n) &= \alpha_n \log (\frac{\alpha_n}{p}) +(1-\alpha_n)\cdot\frac{(p-\alpha_n)}{(1-\alpha_n)} (1+o(1)) \\
    &= \alpha_n \log (\frac{\alpha_n}{p}) (1+o(1)) \\
    & \gtrsim \alpha_n (1+o(1)).\tag{5.70}
\end{align*}
With the use of Lemma \ref{lem6} and \eqref{eq:5.68} with the assumption $G_{A_n}=o(G_n)$, the upper bound of \eqref{eq:5.65} is obtained as
\begin{align*} \label{eq:5.69}
   P({\widehat{\tau}_2}> 2\alpha_{n})
 &\leq   P\bigg[\frac{1}{(G_n-G_{A_n})}\sum_{g\notin\mathcal{A}}1\bigg\{\frac{n\widehat{\boldsymbol{\beta}}_g^{\mathrm{T}}\mathbf{Q}_{n,g}\widehat{\boldsymbol{\beta}}_g}{\sigma^2}>c_1 \log{G_n}\bigg\} \   \geq \alpha_n\bigg] \\
 & \leq e^{-{\alpha_n G_n}(1+o(1))} \\
    & \leq e^{-\frac{ G_{A_n}}{c_2}(1+o(1))}, \tag{5.71}
\end{align*}
where inequality in the last line holds due to \eqref{eq:5.63}. By choosing  
$\alpha_n=c_2\frac{G_{A_n}}{G_n}$, we immediately see that $\alpha_n \to 0$ and satisfies \eqref{eq:5.58} and \eqref{eq:5.63}. Thus, with the above choice of $\alpha_n$, and combining \eqref{eq:5.60} and \eqref{eq:5.69}, we obtain
\begin{align*}
   \sum_{g \notin \mathcal{A}} P(E(1-\kappa_g\mid\widehat{\tau}^{\text{EB}},\sigma^2,\mathcal{D})> \frac{1}{2}, \ {\widehat{\tau}^{\text{EB}}}> 2\alpha_{n}) 
  &\leq G_n e^{-\frac{ G_{A_n}}{c_2}(1+o(1))} .
\end{align*}
Since, $G^{\epsilon_1}_n \lesssim G_{\mathcal{A}_n} \lesssim G^{\epsilon_2}_n$ for some $0< \epsilon_1 <\epsilon_2<\frac{1}{2}$, we can conclude that
\begin{align*}\label{eq:5.80}
\sum_{g \notin \mathcal{A}}P(E(1-\kappa_g\mid{\widehat{\tau}}^{\text{EB}},\sigma^2,\mathcal{D})>\frac{1}{2}, {\widehat{\tau}^{\text{EB}}}> 2\alpha_{n})=o(1), \textrm{ as } n\rightarrow\infty.\tag{5.72}    
\end{align*}
We now proceed to prove that
\begin{align*}\label{eq:5.81}
\sum_{g \notin \mathcal{A}}P(E(1-\kappa_g\mid{\widehat{\tau}}^{\text{EB}},\sigma^2,\mathcal{D})>\frac{1}{2}, {\widehat{\tau}^{\text{EB}}} \leq 2\alpha_{n})=o(1), \textrm{ as } n\rightarrow\infty.\tag{5.73}    
\end{align*}

\textbf{Case-1} $a\geq 1$.\\
Now using the monotonicity of the shrinkage coefficient and then by Markov's inequality, the term in the left hand side of \eqref{eq:5.81} can be bounded as
\begin{align*} \label{eq:5.82}
  &  \sum_{g \notin \mathcal{A}}P(E(1-\kappa_g\mid{\widehat{\tau}}^{\text{EB}},\sigma^2,\mathcal{D})>\frac{1}{2}, {\widehat{\tau}}^{\text{EB}} \leq 2 \alpha_n) \\
   & \leq  \sum_{g \notin \mathcal{A}}P(E(1-\kappa_g\mid 2 \alpha_n,\sigma^2,\mathcal{D})>\frac{1}{2}) \\
   & \leq 2 \sum_{g \notin \mathcal{A}} E \left(E(1-\kappa_g\mid 2 \alpha_n,\sigma^2,\mathcal{D})\right),\tag{5.74}
\end{align*}
where the outer expectation is w.r.t. to $W_{n,g}=\frac{n\widehat{\boldsymbol{\beta}}_g^{\mathrm{T}}\mathbf{Q}_{n,g}\widehat{\boldsymbol{\beta}}_g}{\sigma^2}$. Since, by definition, $E(1-\kappa_g\mid 2 \alpha_n,\sigma^2,\mathcal{D}) \leq 1$, so, we have
\begin{align*} \label{eq:5.83}
    E(1-\kappa_g\mid 2 \alpha_n,\sigma^2,\mathcal{D}) &= E(1-\kappa_g\mid 2 \alpha_n,\sigma^2,\mathcal{D}) 1_{\{W_{n,g} >2c_1 \log G_n\}}+E(1-\kappa_g\mid 2 \alpha_n,\sigma^2,\mathcal{D}) 1_{\{W_{n,g} \leq 2c_1 \log G_n\}}\\
    & \leq 1_{\{W_{n,g} >2c_1 \log G_n\}}+E(1-\kappa_g\mid 2 \alpha_n,\sigma^2,\mathcal{D})  1_{\{W_{n,g} \leq 2c_1 \log G_n\}}.\tag{5.75}
\end{align*}
Next, with the use of arguments similar to Lemma 1 of Paul and Chakrabarti (2023) \cite{paul2022posterior}, 
we provide an upper bound on $E(1-\kappa_g\mid \tau,\sigma^2,\mathcal{D})$ for any $\tau \in (0,1)$.\\
\begin{align*} \label{eq:5.84}
    E(1-\kappa_g\mid \tau,\sigma^2,\mathcal{D}) & \leq \left(\tau^2 e^{\frac{W_{n,g}}{4}}+ K \int_{1}^{\infty} \frac{t \tau^2}{1+t \tau^2} \cdot \frac{1}{{(1+t \tau^2)}^{\frac{m_g}{2}}} t^{-a-1} L(t) e^{\frac{t \tau^2}{1+t \tau^2} \cdot \frac{W_{n,g}}{2}} dt \right) (1+o(1)),\tag{5.76}
\end{align*}
where the term $o(1)$ depends only on $\tau$ and $\lim_{\tau \to 0}o(1)=0$ and $\int_{0}^{\infty} t^{-a-1} L(t) dt=K^{-1}$.

Now, using \eqref{eq:5.83} and \eqref{eq:5.84}, we obtain, for any $\tau \in (0,1)$
\begin{align*}\label{eq:5.85}
  &   E \left(E(1-\kappa_g\mid \tau,\sigma^2,\mathcal{D})\right) \\
   & \leq P( {W_{n,g}} >2c_1 \log G_n) + E \bigg[ \left(\tau^2 e^{\frac{W_{n,g}}{4}}+ K \int_{1}^{\infty} \frac{t ^{-a}\tau^2}{(1+t \tau^2)^{\frac{m_g}{2}+1}}  L(t) e^{\frac{t \tau^2}{1+t \tau^2} \cdot \frac{W_{n,g}}{2}} dt \right) (1+o(1))  1_{\{W_{n,g} \leq 2c_1 \log G_n\}} \bigg] \tag{5.77}
\end{align*}
Since, $\textrm{ for all } \notin \mathcal{A}$, $W_{n,g}\sim \chi^2_{m_g}$, so, by Lemma \ref{lem5},
\begin{equation*}\label{eq:5.86}
     P( {W_{n,g}} >2c_1 \log G_n) \lesssim G^{-{c_1}}_n (\log G_n)^{\frac{s}{2}-1} .\tag{5.78}
\end{equation*}
Next, for the second term in the right-hand side of \eqref{eq:5.85}, noting that  the term $(1+o(1))$ is independent of any particular $g$,
\begin{align*}\label{eq:5.87}
    E \bigg[ \tau^2 e^{\frac{W_{n,g}}{4} }(1+o(1))  1_{\{W_{n,g} \leq 2c_1 \log G_n\}} \bigg] & = \tau^2  \frac{1}{2^{\frac{m_g}{2}}\Gamma(\frac{m_g}{2})}\int_{0}^{2c_1 \log G_n} e^{\frac{u}{4}}e^{-\frac{u}{2}}u^{\frac{m_g}{2}-1}du (1+o(1))\\
    &= \tau^2 \frac{1}{\Gamma(\frac{m_g}{2})} \int_{0} 
 ^{{c_1 \log G_n}}e^{-\frac{u}{2}}u^{\frac{m_g}{2}-1}du (1+o(1))\\
    & \lesssim \tau^2 (1+o(1)).\tag{5.79}
\end{align*}
Finally, for the third term in the right-hand side of \eqref{eq:5.85}, note that
\begin{align*}\label{eq:5.88}
  &  \int_{0}^{2c_1 \log G_n} \int_{1}^{\infty} \frac{t \tau^2}{1+t \tau^2} \cdot \frac{1}{{(1+t \tau^2)}^{\frac{m_g}{2}}} t^{-a-1} L(t) e^{\frac{t \tau^2}{1+t \tau^2} \cdot \frac{u}{2}} e^{-\frac{u}{2}} u^{\frac{m_g}{2}-1} dt du \\
  &=   \int_{1}^{\infty} \frac{t \tau^2}{1+t \tau^2} \cdot \frac{1}{{(1+t \tau^2)}^{\frac{m_g}{2}}} t^{-a-1} L(t) \left(\int_{0}^{2c_1 \log G_n} e^{-\frac{u}{2(1+t \tau^2)}}  u^{\frac{m_g}{2}-1} du \right) dt \\
  &=  \int_{1}^{\infty} \frac{t \tau^2}{1+t \tau^2} \cdot t^{-a-1} L(t) \left(\int_{0}^{\frac{2c_1 \log G_n}{1+t \tau^2}} e^{-\frac{z}{2}} z^{\frac{m_g}{2}-1} dz \right) dt. \tag{5.80}
\end{align*}
Now, we provide an upper bound on \eqref{eq:5.88} separately for $a=1$ and $a>1$. \\
For $a>1$, using the boundedness of $L(t)$, it is easy to show that,
\begin{align*}\label{eq:5.89}
    \int_{1}^{\infty} \frac{t \tau^2}{1+t \tau^2} \cdot t^{-a-1} L(t) \left(\int_{0}^{\frac{2c_1 \log G_n}{1+t \tau^2}} e^{-\frac{z}{2}} z^{\frac{m_g}{2}-1} dz \right) dt & \lesssim \tau^2. \tag{5.81}
\end{align*}
For $a=1$, we have the following
\begin{align*}
  &  \int_{1}^{\infty} \frac{t \tau^2}{1+t \tau^2} \cdot t^{-a-1} L(t) \left(\int_{0}^{\frac{2c_1 \log G_n}{1+t \tau^2}} e^{-\frac{z}{2}} z^{\frac{m_g}{2}-1} dz \right) dt \\
    & \leq  \int_{1}^{\infty} \frac{t \tau^2}{1+t \tau^2} \cdot t^{-a-1} L(t) \left(\int_{0}^{\frac{2c_1 \log G_n}{t \tau^2}} e^{-\frac{z}{2}} z^{\frac{m_g}{2}-1} dz \right) dt  \\
    & \leq C_7 \int_{1}^{\infty} \frac{t \tau^2}{1+t \tau^2} \cdot t^{-a-1} L(t)  dt,
\end{align*}
where $C_7$ is a global constant independent of $n$. Next, note that, $1+t \tau^2 \geq \sqrt{t}$ if and only if $t \geq \frac{1}{\tau^4}$. As a result, we have
\begin{align*} \label{eq:5.90}
  &  \int_{1}^{\frac{1}{\tau^4}}  \frac{t \tau^2}{1+t \tau^2} \cdot t^{-2} L(t) dt 
  \leq M \tau^2 \int_{1}^{\frac{1}{\tau^4}} t^{-1} dt= M \tau^2 \log(\frac{1}{\tau^4}), \tag{5.82}
\end{align*}

    and
\begin{align*}\label{eq:5.91}
    \int_{\frac{1}{\tau^4}}^{\infty}  \frac{t \tau^2}{1+t \tau^2} \cdot t^{-2} L(t) dt 
  \leq M \tau^2  \int_{\frac{1}{\tau^4}}^{\infty} t^{-\frac{3}{2}} dt =2 M \tau^4.  \tag{5.83}
\end{align*}

    Hence, for $a=1$, the upper bound on \eqref{eq:5.88} is obtained as
    \begin{equation*} \label{eq:5.92}
         \int_{1}^{\infty} \frac{t \tau^2}{1+t \tau^2} \cdot t^{-a-1} L(t) \left(\int_{0}^{\frac{2c_1 \log G_n}{1+t \tau^2}} e^{-\frac{z}{2}} z^{\frac{m_g}{2}-1} dz \right) dt  \leq 2 M \tau^2 [\log(\frac{1}{\tau^4}) + \tau^2 ]. \tag{5.84}
    \end{equation*}
    As a consequence, for $a \geq 1$, the upper bound of the third term in the right-hand side of \eqref{eq:5.85} is of the form
    \begin{align*}\label{eq:5.93}
         E \bigg[  K \int_{1}^{\infty} \frac{t \tau^2}{1+t \tau^2} \cdot \frac{1}{{(1+t \tau^2)}^{\frac{m_g}{2}}} t^{-a-1} L(t) e^{\frac{t \tau^2}{1+t \tau^2} \cdot \frac{W_{n,g}}{2}} dt  (1+o(1))  1_{\{W_{n,g} \leq c_1 \log G_n\}} \bigg] & \lesssim \tau^2[\log(\frac{1}{\tau^2}) + \tau^2 ] (1+o(1)). \tag{5.85}
    \end{align*}
    Finally, substituting the upper bounds obtained in \eqref{eq:5.86}, \eqref{eq:5.87} and \eqref{eq:5.93} in \eqref{eq:5.85}
    with $\tau=\alpha_n=\frac{c_2G_{A_n}}{G_n}$, the upper bound on \eqref{eq:5.81} can be obtained as
    \begin{align*}
        &  \sum_{g \notin \mathcal{A}}P(E(1-\kappa_g\mid{\widehat{\tau}}^{\text{EB}},\sigma^2,\mathcal{D})>\frac{1}{2}, {\widehat{\tau}}^{\text{EB}} \leq 2 \alpha_n) \\
        & \lesssim G^{-{c_1}+1}_n (\log G_n)^{\frac{s}{2}-1}+ \frac{(G_{A_n})^2}{G_n}(1+o(1)) +  \frac{(G_{A_n})^2}{G_n} \log G_n(1+o(1)).
    \end{align*}
    Hence, for $c_1 \geq 2$ and $G^{\epsilon_1}_n \lesssim G_{\mathcal{A}_n} \lesssim G^{\epsilon_2}_n$ for some $0< \epsilon_1 <\epsilon_2<\frac{1}{2}$, we have
    \begin{align*}\label{eq:5.94}
\sum_{g \notin \mathcal{A}}P(E(1-\kappa_g\mid{\widehat{\tau}}^{\text{EB}},\sigma^2,\mathcal{D})>\frac{1}{2}, {\widehat{\tau}}^{\text{EB}} \leq 2 \alpha_n)=o(1), \textrm{ as } n\rightarrow\infty.\tag{5.86}
\end{align*}
This completes the proof for $a \geq 1$.\\
\textbf{Case-II} $\frac{1}{2} <a<1$. Again using the monotonicity of the shrinkage coefficient, the  term in the left-hand side of \eqref{eq:5.81} can be bounded as
\begin{align*}
     &  \sum_{g \notin \mathcal{A}}P(E(1-\kappa_g\mid{\widehat{\tau}}^{\text{EB}},\sigma^2,\mathcal{D})>\frac{1}{2}, {\widehat{\tau}}^{\text{EB}} \leq 2 \alpha_n) \\
   & \leq  \sum_{g \notin \mathcal{A}}P(E(1-\kappa_g\mid 2 \alpha_n,\sigma^2,\mathcal{D})>\frac{1}{2}) \\
   & \lesssim G_n \left(\frac{G_{A_n}}{G_n} \right)^{2a} \bigg[\log \left(\frac{G_n}{G_{A_n}}\right)\bigg]^{\frac{s}{2}-1} (1+o(1)).
\end{align*}
Here inequality in the last line follows due to the use of arguments similar to the proof of Theorem 1 where $\tau$ is assumed to be a tuning one.

Note that, for $\frac{1}{2} <a<1$, there exists $\epsilon \in (0,1-\frac{1}{2a})$ such that $2a(1-\epsilon)>1$. Hence, when $\frac{1}{2} <a<1$, for $G^{\epsilon_1}_n \lesssim G_{\mathcal{A}_n} \lesssim G^{\epsilon_2}_n$ for some $0< \epsilon_1 <\epsilon_2<1-\frac{1}{2a}$, we conclude 
 \begin{align*}\label{eq:5.95}
\sum_{g \notin \mathcal{A}}P(E(1-\kappa_g\mid{\widehat{\tau}}^{\text{EB}},\sigma^2,\mathcal{D})>\frac{1}{2}, {\widehat{\tau}}^{\text{EB}} \leq 2 \alpha_n)=o(1), \textrm{ as } n\rightarrow\infty \tag{5.87}
\end{align*}
and the proof is completed for $\frac{1}{2} <a<1$. 
\end{proof}


\textbf{Proof of Theorem \ref{Thm-6}:-}
\begin{proof}
To prove Theorem \ref{Thm-6}, we employ a similar set of arguments used in the proof of Theorem 2 when $\tau$ is used as a tuning parameter. Here, $T_1$ is as defined in Theorem 2 and $T_2= {\boldsymbol{\alpha}}^{\mathrm{T}} {\boldsymbol{\Sigma}^{\frac{1}{2}}_{\mathcal{A}}}(\widehat{\boldsymbol{\beta}}_{\mathcal{A},EB}^{\text{HT}}-\widehat{\boldsymbol{\beta}}_{\mathcal{A}})$. 
\newline
As in Theorem 2, it follows that for proving $T_2 \xrightarrow{P} 0$, it suffices to show that
\begin{equation*} \label{eq:5.96}
     \sum_{g \in \mathcal{A}} U_{ng}  \xrightarrow{P} 0 \hspace*{0.2cm} \text{as} \hspace*{0.2cm}  n \to \infty,
    \tag{5.88} 
\end{equation*}
where $$W_{n,g}=\frac{n\widehat{\boldsymbol{\beta}}_g^{\mathrm{T}}\mathbf{Q}_{n,g}\widehat{\boldsymbol{\beta}}_g}{\sigma^2}, \textrm{ and } U_{n,g}= W_{n,g} E(\kappa^2_g\mid \widehat{\tau}^{\text{EB}},\sigma^2,\mathcal{D}).$$ 
Now to establish \eqref{eq:5.96}, let us first fix any $\epsilon_0>0$, and take  $\gamma_n=\frac{1}{G_n}$, for $n\geq 1$. Note that, for fixed $\mathcal{D}=\{\mathbf{y}\}$ and $\sigma^2$, $E(\kappa_g\mid\tau,\sigma^2,\mathcal{D})$ is non-increasing in $\tau$, and $\widehat{\tau}^{\text{EB}}\geq \gamma_n$ for all $n\geq 1$. Therefore, for $\tau_n=\gamma_n$ and using the monotonicity of $\tau$, we only need to show that

\begin{align*} \label{eq:5.97}
  \sum_{g \in \mathcal{A}} U^{*}_{n,g} =  \sum_{g \in \mathcal{A}}  W_{n,g} E(\kappa^2_g\mid \gamma_n,\sigma^2,\mathcal{D}) \xrightarrow{P} 0 \hspace*{0.2cm} \text{as} \hspace*{0.2cm}  n \to \infty. \tag{5.89}
\end{align*}

Next, we proceed following the steps as mentioned in \eqref{eq:5.41}-\eqref{eq:5.44} with $\tau=\frac{1}{G_n}$. Now we define $C_n$ as $C_n = \{ W_{n,g} >2a\rho \frac{\sqrt{G_n}}{\log G_n} \}$ and $B_n$ as $B_n =\{ U^{*}_{n,g} >\frac{\epsilon_0}{|\mathcal{A}|} \}$. Hence, to obtain the optimal estimation rate, it is enough to show that
\begin{align*} \label{eq:5.98}
    \lim_{n \to \infty} \sum_{g \in \mathcal{A}} P(B_n)=0. \tag{5.90} 
\end{align*}
Here, observe that under the assumption that $|\mathcal{A}| =O( G^{\epsilon}_n), 0<\epsilon<\frac{1}{2}$, we have, for $k=1,2$,
\begin{align*}
    \sup_{W_{n,g} >2a\rho \frac{\sqrt{G_n}}{\log G_n}} h_k(W_{n,g},\frac{1}{G_n})&= o(\frac{1}{|\mathcal{A}|}).
\end{align*}
As a consequence of this, we have for all $\ g \in \mathcal{A}$,
\begin{align*}
    P(B_n \cap C_n)=0.
\end{align*}
This also ensures
\begin{align*} \label{eq:5.99}
  \lim_{n \to \infty}  \sum_{g \in \mathcal{A}} P(B_n \cap C_n)=0. \tag{5.91}
\end{align*}

 Next note that for $\tau=\frac{1}{G_n}$, we have $G_n \sqrt{\tau} \log (\frac{1}{\tau}) \to \infty$ as $n \to \infty$, and hence using similar set of arguments used in \eqref{eq:5.21}-\eqref{eq:5.26}, we can show that
\begin{align*} \label{eq:5.100}
\lim_{n\rightarrow\infty} \sum_{g \in \mathcal{A}}  P( C^{c}_n)&=0.   \tag{5.92}
\end{align*}
Finally combining \eqref{eq:5.99} and \eqref{eq:5.100} implies \eqref{eq:5.98} and \eqref{eq:5.97} holds. As a result, \eqref{eq:5.96} is also established, and that along with the use of Slutsky's Theorem completes the proof of Theorem \ref{Thm-6}.
\end{proof}

\textbf{Proof of Theorem \ref{Thm-8}:}\\
\begin{proof}
    Here also, we will show that
    \begin{align*}\label{eq:5.103}
\sum_{g \in \mathcal{A}} P(E(1-\kappa_g\mid{\widehat{\tau}}^{\text{EB}},\sigma^2,\mathcal{D}) < \frac{1}{2})=o(1), \textrm{ as } n\rightarrow\infty,\tag{5.93}    
\end{align*}
and
\begin{align*}\label{eq:5.104}
\sum_{g \notin \mathcal{A}}P(E(1-\kappa_g\mid{\widehat{\tau}}^{\text{EB}},\sigma^2,\mathcal{D})>\frac{1}{2})=o(1), \textrm{ as } n\rightarrow\infty,\tag{5.94}    
\end{align*}
both when $0.5\leq a<1$, and $a\geq 1$, where ${\widehat{\tau}}^{\text{EB}}$ is defined in \eqref{eq:3.2}.\\
Now, using the same technique as used in the proof of Theorem \ref{Thm-7} and taking $\gamma_n=(\frac{1}{G_n})^{2}$, for $n\geq 1$, and then following steps similar to \eqref{eq:5.78} and \eqref{eq:5.79}, we obtain
\begin{equation}
     \lim_{n\rightarrow \infty}\sum_{g \in \mathcal{A}} P(E(\kappa_g\mid\widehat{\tau}^{\text{EB}},\sigma^2,\mathcal{D})> \frac{1}{2})=0,\notag
\end{equation}
which completes the proof of \eqref{eq:5.103}. Now, we are left to prove \eqref{eq:5.104}. Again following the same steps as given in the proof of Theorem \ref{Thm-7} along with the use of \eqref{eq:5.69} with the use of \eqref{eq:3.6}, we obtain
\begin{align*}
   \sum_{g \notin \mathcal{A}} P(E(1-\kappa_g\mid\widehat{\tau}^{\text{EB}},\sigma^2,\mathcal{D})> \frac{1}{2}, \ ({\widehat{\tau}^{\text{EB}}})^{\frac{1}{2}}> 2\alpha_{n}) 
  &\leq G_n e^{-\frac{ G_{A_n}}{c_2}(1+o(1))} .
\end{align*}
Since, $G^{\epsilon_1}_n \lesssim G_{\mathcal{A}_n} \lesssim G^{\epsilon_2}_n$ for some $0< \epsilon_1 <\epsilon_2<\frac{1}{2}$, we can conclude that
\begin{align*}\label{eq:5.105}
\sum_{g \notin \mathcal{A}}P(E(1-\kappa_g\mid{\widehat{\tau}}^{\text{EB}},\sigma^2,\mathcal{D})>\frac{1}{2}, ({\widehat{\tau}^{\text{EB}}})^{\frac{1}{2}}> 2\alpha_{n})=o(1), \textrm{ as } n\rightarrow\infty.\tag{5.95}    
\end{align*}
Next using the monotonicity of $E(1-\kappa_g\mid\tau,\sigma^2,\mathcal{D})$, it follows 
\begin{align*}
    \sum_{g \notin \mathcal{A}} P(E(1-\kappa_g\mid\widehat{\tau}^{\text{EB}},\sigma^2,\mathcal{D})> \frac{1}{2}, \ ({\widehat{\tau}^{\text{EB}}})^{\frac{1}{2}} \leq 2\alpha_{n}) 
  &\leq  \sum_{g \notin \mathcal{A}} P(E(1-\kappa_g\mid (2 \alpha_n)^{2},\sigma^2,\mathcal{D})> \frac{1}{2}) \\
  & \lesssim G_n \alpha^{2}_n [\log (\frac{1}{\alpha_n})]^{\frac{s}{2}} .
\end{align*}
With the choice of $\alpha_n= \frac{c_2 G_{A_n}}{G_n}$, $c_2 \geq 1$ for all $n \geq 1$, and since, $G^{\epsilon_1}_n \lesssim G_{\mathcal{A}_n} \lesssim G^{\epsilon_2}_n$ for some $0< \epsilon_1 <\epsilon_2<\frac{1}{2}$, we conclude that 
\begin{align*}\label{eq:5.106}
\sum_{g \notin \mathcal{A}}P(E(1-\kappa_g\mid{\widehat{\tau}}^{\text{EB}},\sigma^2,\mathcal{D})>\frac{1}{2}, ({\widehat{\tau}^{\text{EB}}})^{\frac{1}{2}} \leq  2\alpha_{n})=o(1), \textrm{ as } n\rightarrow\infty.\tag{5.96}    
\end{align*}
Combining \eqref{eq:5.105} and \eqref{eq:5.106}, we obtain \eqref{eq:5.104} and the proof is completed.
\end{proof}

\textbf{Proof of Theorem \ref{Thm-9}:}\\
\begin{proof}
    To show that the decision rule \eqref{eq:2.14} is an oracle, at first, we show the selection consistency part. Using similar arguments used before, again we have
    \begin{align*}\label{eq:5.107}
	P(\mathcal{A}_n \neq \widehat{\mathcal{A}}_n) \leq \sum_{g \in \mathcal{A}} P(E(1-\kappa_g\mid
 \sigma^2,\mathcal{D}) < \frac{1}{2})+\sum_{g \notin \mathcal{A}}P(E(1-\kappa_g\mid \sigma^2,\mathcal{D})>\frac{1}{2}).\tag{5.97}
\end{align*}
Here also, our target is to show the following:
\begin{align*}\label{eq:5.108}
\sum_{g \in \mathcal{A}} P(E(1-\kappa_g\mid
\sigma^2,\mathcal{D}) < \frac{1}{2})=o(1), \textrm{ as } n\rightarrow\infty,\tag{5.98}    
\end{align*}
and
\begin{align*}\label{eq:5.109}
\sum_{g \notin \mathcal{A}}P(E(1-\kappa_g\mid \sigma^2,\mathcal{D})>\frac{1}{2})=o(1), \textrm{ as } n\rightarrow\infty,\tag{5.99}    
\end{align*}
both when $0.5\leq a<1$, and $a\geq 1$.\\
In order to show \eqref{eq:5.108}, first note that with the use of \hyperlink{condition1}{D1},
\begin{align*}
    E(\kappa_g\mid \sigma^2,\mathcal{D}) &= \int_{\gamma_{1n}}^{\gamma_{2n}} E(\kappa_g\mid \tau_n, \sigma^2,\mathcal{D}) \pi(\tau_n|\mathcal{D}) d \tau \\
    & \leq E(\kappa_g\mid \gamma_{1n}, \sigma^2,\mathcal{D}).
\end{align*}

   Here, first, we use the fact that given $\tau_n, \sigma^2$ and $\mathcal{D}$, the posterior mean of $\kappa_g$ depends only on $g^{th}$ group. Inequality in the last line follows since, for any fixed $\tau_n$ and $\sigma^2$, $E(\kappa_g \mid \tau_n, \sigma^2, \mathcal{D})$ is non-increasing in $\tau$. As a result of this, we have
   \begin{equation*} \label{eq:5.110}
       \sum_{g \in \mathcal{A}} P(E(1-\kappa_g\mid
\sigma^2,\mathcal{D}) < \frac{1}{2}) \leq \sum_{g \in \mathcal{A}} P(E(\kappa_g\mid \gamma_{1n}, \sigma^2,\mathcal{D})>\frac{1}{2}). \tag{5.100} 
   \end{equation*}
Next using the same steps as used in \eqref{eq:5.20}-\eqref{eq:5.26} with $\tau_n=\gamma_{1n}$ and noting that $\gamma_{1n}$ satisfies $\log (\frac{1}{\tau_n}) \asymp \log(G_n)$, along with the use of \eqref{eq:5.110}, \eqref{eq:5.108} is established. \\
In case of proving \eqref{eq:5.109}, note that
\begin{align*}
    E(1-\kappa_g\mid \sigma^2,\mathcal{D}) &= \int_{\gamma_{1n}}^{\gamma_{2n}} E(1-\kappa_g\mid \tau_n, \sigma^2,\mathcal{D}) \pi(\tau_n|\mathcal{D}) d \tau \\
    & \leq E(1-\kappa_g\mid \gamma_{2n}, \sigma^2,\mathcal{D}).
\end{align*}
    Here, first, we use the fact that given $\tau_n, \sigma^2$ and $\mathcal{D}$, the posterior mean of $1-\kappa_g$ depends only on $g^{th}$ group. Inequality in the last line follows since, for any fixed $\tau_n$ and $\sigma^2$, $E(1-\kappa_g \mid \tau_n, \sigma^2, \mathcal{D})$ is non-decreasing in $\tau$. As a consequence, it follows that
    \begin{align*}
        \sum_{g \notin \mathcal{A}}P(E(1-\kappa_g\mid \sigma^2,\mathcal{D})>\frac{1}{2}) & \leq \sum_{g \notin \mathcal{A}} P(E(1-\kappa_g\mid \gamma_{2n}, \sigma^2,\mathcal{D})>\frac{1}{2}) \\
        & \lesssim G_n \gamma_{2n} [\log (\frac{1}{\gamma_{2n}})]^{\frac{s}{2}-1}(1+o(1)).
    \end{align*}
    Inequality in the last line follows due to the use of arguments similar to those used in \eqref{eq:5.27}-\eqref{eq:5.31}. Finally, under the assumption that $G_n \gamma_{2n} [\log (\frac{1}{\gamma_{2n}})]^{\frac{s}{2}-1} \to 0$ as $n \to \infty$, \eqref{eq:5.109} also holds and completes the proof related to variable selection consistency. \\
    Next, we move forward to show the optimal estimation. Here we aim to show, 
    \begin{align*}
        {\boldsymbol{\alpha}}^{\mathrm{T}} {\boldsymbol{\Sigma}^{\frac{1}{2}}_{\mathcal{A}}}(\widehat{\boldsymbol{\beta}}_{\mathcal{A},FB}^{\text{HT}}-\widehat{\boldsymbol{\beta}}_{\mathcal{A}}) \xrightarrow{P} 0 \textrm{ as } n \to \infty.
    \end{align*}
Now making use of the same arguments as used in \eqref{eq:5.34}-\eqref{eq:5.40} of Theorem \ref{Thm-2}, we only need to show that
\begin{align*} \label{eq:5.111}
  \sum_{g \in \mathcal{A}}  W_{n,g} E(\kappa^2_g\mid \gamma_{1n},\sigma^2,\mathcal{D}) \xrightarrow{P} 0 \hspace*{0.2cm} \text{as} \hspace*{0.2cm}  n \to \infty. \tag{5.101}
\end{align*}
Next we follow the same steps as mentioned in \eqref{eq:5.41}-\eqref{eq:5.48} with $\tau=\gamma_{1n}$.  
Note that since $\gamma_{1n}$ satisfies both \ref{assmpan-3} and $\log (\frac{1}{\gamma_{1n}}) \asymp \log(G_n)$, it is immediate that
\eqref{eq:5.111} holds and as a result, $T_2 \xrightarrow{P} 0$ as $ n \to \infty$. Finally, the use of Slutsky's Theorem completes the proof of Theorem \ref{Thm-9}.

\end{proof}

\section{Concluding Remarks}
\label{sec-6}
In this paper, we studied the problem of finding relevant groups of predictors/regressors in a sparse high-dimensional regression model, assuming that the potential regressors are inherently grouped. For the selection of groups, we considered a half-thresholding rule based on a broad class of global-local shrinkage priors with polynomial tails. We also studied an estimator of group regression coefficients based on this rule.
Our key contributions are as follows.

\begin{enumerate}
    \item First, we extend the idea of one-group global-local shrinkage priors in group regression problems to facilitate group selection by factorizing the similarity within the groups and heterogeneity across groups through the prior distribution of the group coefficients.
\item Secondly, using such priors, we have proposed a half-thresholding rule that can be easily implemented when the sparsity level is known or unknown. It is shown that when the proportion of active groups is known, the global shrinkage parameter can be chosen in such a way using this information that the resulting decision rule achieves both variable selection consistency and the estimate achieves optimal estimation rate simultaneously, a property represented as Oracle property by Fan and Li (2001) \cite{fan2001variable}, Zou (2006) \cite{zou2006adaptive}.
We propose empirical Bayes and full Bayes approaches for dealing with the global shrinkage parameter when the proportion of active groups is unknown. The empirical Bayes estimates used in this context generalize an empirical Bayes estimate proposed by van der Pas et al. (2014) \cite{van2014horseshoe} in the normal means problem. These results also successfully resolve a question left open in Tang et al. (2018) \cite{tang2018bayesian} regarding the Oracle property of the empirical Bayes version of their half-thresholding rule.

\item Third, our rigorous analytical treatment allows us to successfully get around the soft spot in the argument of Tang et al. (2018) \cite{tang2018bayesian} to prove the oracle optimality property of their half-thresholding estimator and complete the proof ( with modifications as above). These results are the first of their kind in the literature and require developing novel analytical techniques for their theoretical proofs.

\item Finally, in our simulation studies, we have compared different versions of our proposed decision rule to some well-known group selection methods in the literature. We have demonstrated that our proposed decision rule outperforms these methods when the number of observations is small or moderate. When the number of observations is large, the methods due to Yang and Narisetty (2020) \cite{yang2020consistent} and Xu and Ghosh (2015) \cite{xu2015bayesian} are comparable.
Therefore, our proposed HT rule can be a viable alternative to the methods available in the literature based on spike and slab priors when dealing with sparse situations.
\end{enumerate}

Throughout this work, our focus is only on selecting relevant groups. However, this might not always be the case. In many cases, the selection of variables at the group and individual levels might be desirable. One such example is presented by Huang et al. (2012) \cite{huang2012selective}. In genome-wide association studies, genetic variations in the same gene form a natural group. However, a genetic variation related to the disease does not necessarily mean that
all other variations in the same gene are also associated with the disease. Hence, in this situation, the selection of important genes both at the group and individual level becomes natural. This problem is known as \textit{ bilevel selection}. Although both Xu and Ghosh (2015) \cite{xu2015bayesian} and Boss et al. (2023) \cite{boss2023group} 
studied this problem, no theoretical guarantee of selection consistency or optimal estimation rate has been achieved in this context for bilevel selection. Hence, it becomes equally important to answer the same question in our hierarchical formulation. This might be an interesting problem to consider elsewhere. \\
\hspace*{0.5cm} Another least visited problem in this literature is a situation where the consideration groups overlap. As mentioned by Huang et al. (2012) \cite{huang2012selective}, in genomic data analysis involving both
genes and pathways, many important genes belong to more than one pathway. Hence, in this context, it is challenging to select important variables without all the groups that contain them. Some of the works in this direction are due to Jacob, Obozinski, and Vert (2009) \cite{jacob2009group}, Liu and Ye (2010) \cite{liu2010fast}, Zhao, Rocha and Yu (2009) \cite{zhao2009composite}, etc. A possible question left unanswered in the literature of one-group priors is to incorporate this situation and provide a corresponding decision rule. We want to study this problem in the future.


\begin{thebibliography}{3}

\bibitem[1]{armagan2011generalized}
Armagan, Artin and Clyde, Merlise and Dunson, David (2011) . Generalized beta mixtures of Gaussians.
{\em Advances in neural information processing systems}, 24.

\bibitem[2]{armagan2013generalized}
Armagan, Artin and Dunson, David B and Lee, Jaeyong (2013) . Generalized double Pareto shrinkage.
{\em Statistica Sinica}, 23, 119.

\bibitem[3]{armagan2013posterior}
Armagan, Artin and Dunson, David B and Lee, Jaeyong and Bajwa, Waheed U and Strawn, Nate (2013) . Posterior consistency in linear models under shrinkage priors.
{\em Biometrika}, 100(4), 1011--1018.


\bibitem[4]{bhadra2017horseshoe+}
Bhadra, Anindya and Datta, Jyotishka and Polson, Nicholas G and Willard, Brandon (2017) . The horseshoe+ estimator of ultra-sparse signals.
{\em Bayesian Analysis}, 12(4), 1105--1131.

\bibitem[5]{bhattacharya2015dirichlet}
Bhattacharya, Anirban and Pati, Debdeep and Pillai, Natesh S and Dunson, David B (2013). Dirichlet--Laplace priors for optimal shrinkage.
{\em Journal of the American Statistical Association}, 110(512), 1479--1490.

 
 

\bibitem[6]{bickel2009simultaneous}
Bickel, Peter J and Ritov, Ya’acov and Tsybakov, Alexandre B (2009) . Simultaneous analysis of Lasso and Dantzig selector.
{\em The Annals of statistics}, 37(4), 1705--1732.

\bibitem[7]{bingham_goldie_teugels_1987}
Bingham, N. H. and Goldie, C. M. and Teugels, J. L. (1987) . Regular Variation.
{\em Cambridge University Press}.


\bibitem[8]{bogdan2011asymptotic}
Bogdan, Ma{\l}gorzata and Chakrabarti, Arijit and Frommlet, Florian and Ghosh, Jayanta K (2011) . Asymptotic Bayes-optimality under sparsity of some multiple testing procedures.
{\em The Annals of Statistics}, 39(3), 1551--1579.

\bibitem[9]{bogdan2008comparison}
Bogdan, Ma{\l}gorzata and Ghosh, Jayanta K and Tokdar, Surya T (2008) . A comparison of the Benjamini-Hochberg procedure with some Bayesian rules for multiple testing.
{\em arXiv preprint arXiv:0805.2479}.


\bibitem[10]{brown2010inference}
Brown, Philip J and Griffin, Jim E (2010) . Inference with normal-gamma prior distributions in regression problems.
{\em Bayesian Analysis}, 5(1), 171--188.

\bibitem[11]{candes2007dantzig}
Candes, Emmanuel and Tao, Terence (2007) . The Dantzig selector: Statistical estimation when p is much larger than n.
{\em The annals of Statistics}, 35(6), 2313--2351.


\bibitem[12]{caron2008sparse}
Caron, Fran{\c{c}}ois and Doucet, Arnaud (2008) . Sparse Bayesian nonparametric regression.
{\em Proceedings of the 25th international conference on Machine learning}, 88--95.

\bibitem[13]{carvalho2009handling}
Carvalho, Carlos M and Polson, Nicholas G and Scott, James G (2009) . Handling sparsity via the horseshoe.
{\em Artificial Intelligence and Statistics}, 73--80.

\bibitem[14]{carvalho2010horseshoe}
Carvalho, Carlos M and Polson, Nicholas G and Scott, James G (2010) . The horseshoe estimator for sparse signals.
{\em Biometrika}, 97(2), 465--480.

\bibitem[15]{casella2010penalized}
Casella, George and Ghosh, Malay and Gill, Jeff and Kyung, Minjung (2010) . Penalized regression, standard errors, and Bayesian lassos.
{\em Bayesian Analysis}, 5(2), 369--411.


\bibitem[16]{chen2012extended}
Chen, Jiahua and Chen, Zehua (2012) . Extended BIC for small-n-large-p sparse GLM.
{\em Statistica Sinica}, 97(2), 555--574.

\bibitem[17]{damlen1999gibbs}
Damlen, Paul and Wakefield, John and Walker, Stephen (1999) . Gibbs sampling for Bayesian non-conjugate and hierarchical models by using auxiliary variables.
{\em Journal of the Royal Statistical Society: Series B (Statistical Methodology)}, 61(2), 331--344.


\bibitem[18]{datta2013asymptotic}
Datta, Jyotishka and Ghosh, Jayanta K (2013) . Asymptotic properties of Bayes risk for the horseshoe prior.
{\em Bayesian Analysis}, 8(1), 111--132.



\bibitem[19]{donoho1992maximum}
Donoho, D. L., Johnstone, I. M. and Hoch, J. C. and Stern, A. S. (1992). Maximum entropy and the nearly black object.
{\em Journal of the Royal Statistical Society: Series B (Methodological)}, 54(1), 41--67.


\bibitem[20]{efron2004large}
Efron, Bradley (2004) . Large-scale simultaneous hypothesis testing: the choice of a null hypothesis.
{\em Journal of the american statistical association}, 99(465), 96--104.



\bibitem[21]{fan2001variable}
Fan, Jianqing and Li, Runze (2001) .Variable selection via nonconcave penalized likelihood and its oracle properties.
{\em Journal of the american statistical association}, 96(456), 1348--1360.

\bibitem[22]{fan2004nonconcave}
Fan, Jianqing and Peng, Heng (2004) .Nonconcave penalized likelihood with a diverging number of parameters.
{\em The annals of statistics}, 32(3), 928--961.

\bibitem[23]{fan2010selective}
Fan, Jianqing and Lv, Jinchi (2010) .A selective overview of variable selection in high dimensional feature space.
{\em Statistica Sinica}, 20(1), 101.


\bibitem[24]{gabcke2015}
Gabcke, Wolfgang (2015) .Derivation of the Riemann-Siegel formula with explicit estimates of its remainders.
{\em http://dx.doi.org/10.53846/goediss-5113}.


\bibitem[25]{george1993variable}
George, Edward I and McCulloch, Robert E (1993) .Variable selection via Gibbs sampling.
{\em Journal of the American Statistical Association}, 88(423), 881--889.


\bibitem[26]{bach2008consistency}
Bach, Francis R (2008) .Consistency of the group lasso and multiple kernel learning.
{\em Journal of Machine Learning Research}, 9(6).

\bibitem[27]{geweke1996variable}
Geweke, John (1996) .Variable selection and model comparison in regression.
{\em In Bayesian Statistics 5}.


\bibitem[28]{ghosh2017asymptotic}
Ghosh, Prasenjit and Chakrabarti, Arijit (2017). Asymptotic optimality of one-group shrinkage priors in sparse high-dimensional problems.
{\em Bayesian Analysis}, 12(4), 1133--1161.

\bibitem[29]{ghosh2016asymptotic}
Ghosh, Prasenjit and Tang, Xueying and Ghosh, Malay and Chakrabarti, Arijit (2016). Asymptotic properties of Bayes risk of a general class of shrinkage priors in multiple hypothesis testing under sparsity.
{\em Bayesian Analysis}, 11(3), 753--796.

\bibitem[30]{ghosal2000convergence}
Ghosal, Subhashis and Ghosh, Jayanta K and Van Der Vaart, Aad W (2000) . Convergence rates of posterior distributions.
{\em The Annals of statistics}, 500--531.

\bibitem[31]{griffin2005alternative}
Griffin, JE and Brown, PJ (2005) . Alternative prior distributions for variable selection with very many more variables than observations.
{\em Technical report, University of Warwick}.

\bibitem[32]{hoerl1970ridge}
Hoerl, Arthur E and Kennard, Robert W (1970) . Ridge regression: Biased estimation for nonorthogonal problems.
{\em Technometrics}, 12(1), 55--67.



\bibitem[33]{jiang2009general}
Jiang, Wenhua and Zhang, Cun-Hui (2009) . General maximum likelihood empirical Bayes estimation of normal means.
{\em The Annals of Statistics}, 37(4), 1647--1684.

\bibitem[34]{johnstone2004needles}
Johnstone, Iain M and Silverman, Bernard W (2004). Needles and straw in haystacks: Empirical Bayes estimates of possibly sparse sequences.
{\em The Annals of Statistics}, 32(4), 1594--1649.

\bibitem[35]{li2010bayesian}
Li, Qing and Lin, Nan (2010) . The Bayesian elastic net.
{\em Bayesian analysis}, 5(1), 151--170.

\bibitem[36]{mitchell1988bayesian}
Mitchell, Toby J and Beauchamp, John J (1988) . Bayesian variable selection in linear regression.
{\em Journal of the american statistical association}, 83(404), 1023--1032.

\bibitem[37]{narisetty2014bayesian}
Narisetty, Naveen Naidu and He, Xuming (2014) . Bayesian variable selection with shrinking and diffusing priors.
{\em The Annals of Statistics}, 42(2), 789--817.


\bibitem[38]{park2008bayesian}
Park, Trevor and Casella, George (2008). The bayesian lasso.
{\em Journal of the american statistical association}, 103(482), 681--686.

\bibitem[39]{polson2010shrink}
Polson, Nicholas G and Scott, James G (2010) . Shrink globally, act locally: Sparse Bayesian regularization and prediction.
{\em Bayesian statistics}, 9, 501--538.

\bibitem[40]{polson2012local}
Polson, Nicholas G and Scott, James G (2012) . Local shrinkage rules, L{\'e}vy processes and regularized regression.
{\em Journal of the Royal Statistical Society: Series B (Statistical Methodology)}, 74(2), 287--311.



\bibitem[41]{rovckova2018bayesian}
Ro{\v{c}}kov{\'a}, Veronika (2018) . Bayesian estimation of sparse signals with a continuous spike-and-slab prior.
{\em The Annals of Statisitcs}, 46(1), 558--575..

\bibitem[42]{rovckova2018spike}
Ro{\v{c}}kov{\'a}, Veronika and George, Edward I (2018) . The spike-and-slab lasso.
{\em Journal of the American Statistical Association}, 113(521), 431--444.


\bibitem[43]{rigollet2012sparse}
Rigollet, Philippe and Tsybakov, Alexandre B (2012) . Sparse estimation by exponential weighting.
{\em Statistical Science}, 27(4),401--437.


\bibitem[44]{tipping2001sparse}
Tipping, Michael E (2001) . Sparse Bayesian learning and the relevance vector machine.
{\em Journal of machine learning research}, 1(Jun), 211--244.

\bibitem[45]{johnson2012bayesian}
Johnson, Valen E and Rossell, David (2012) . Bayesian model selection in high-dimensional settings.
{\em Journal of the American Statistical Association}, 107(498), 649--660.



\bibitem[46]{van2014horseshoe}
Van Der Pas, St{\'e}phanie L and Kleijn, Bas JK and Van Der Vaart, Aad W (2014). The horseshoe estimator: Posterior concentration around nearly black vectors.
{\em Electronic Journal of Statistics}, 8(2), 2585--2618.


\bibitem[47]{van2016conditions}
Van Der Pas, SL and Salomond, J-B and Schmidt-Hieber, Johannes (2016) . Conditions for posterior contraction in the sparse normal means problem.
{\em Electronic journal of statistics}, 10(1), 976--1000.



\bibitem[48]{van2017adaptive}
Van der Pas, St{\'e}phanie and Szab{\'o}, Botond and Van der Vaart, Aad (2017). Adaptive posterior contraction rates for the horseshoe.
{\em Electronic Journal of Statistics}, 11(2), 3196--3225.


\bibitem[49]{hahn2015decoupling}
Hahn, P Richard and Carvalho, Carlos M (2015). Decoupling shrinkage and selection in Bayesian linear models: a posterior summary perspective.
{\em Journal of the American Statistical Association}, 110(509), 435--448.

\bibitem[50]{yang2020consistent}
Yang, Xinming and Narisetty, Naveen N (2020). Consistent group selection with Bayesian high dimensional modeling.
{\em Bayesian Analysis}, 15(3), 909--935.


\bibitem[51]{zou2005regularization}
Zou, Hui and Hastie, Trevor (2005). Regularization and variable selection via the elastic net.
{\em Journal of the royal statistical society: series B (statistical methodology)}, 67(2), 301--320.


\bibitem[52]{zhang2007penalized}
Zhang, Cun Hui (2007). Penalized linear unbiased selection.
{\em Department of Statistics and Bioinformatics, Rutgers University}, 3, 894--942.

\bibitem[53]{zhang2008sparsity}
Zhang, Cun-Hui and Huang, Jian (2008). The sparsity and bias of the lasso selection in high-dimensional linear regression.
{\em The Annals of Statistics}, 36(4), 1567--1594.

\bibitem[54]{zhang2006model}
Zhang, Cun-Hui and Huang, Jian (2006). Model-selection consistency of the lasso in highdimensional linear regression.
{\em The Annals of Statistics}, 36(4), 1567--1594.

\bibitem[55]{tang2018bayesian}
Tang, Xueying and Xu, Xiaofan and Ghosh, Malay and Ghosh, Prasenjit (2018). Bayesian variable selection and estimation based on global-local shrinkage priors.
{\em Sankhya A}, 80(2), 215--246.

\bibitem[56]{xu2016bayesian}
Xu, Zemei and Schmidt, Daniel F and Makalic, Enes and Qian, Guoqi and Hopper, John L (2016). Bayesian grouped horseshoe regression with application to additive models.
{\em Australasian Joint Conference on Artificial Intelligence}, 229--240.

\bibitem[57]{zou2006adaptive}
Zou, Hui (2006). The adaptive lasso and its oracle properties.
{\em Journal of the American Statistical Association}, 101(476), 1418--1429.

\bibitem[58]{wang2008note}
Wang, Hansheng and Leng, Chenlei (2008). A note on adaptive group lasso.
{\em Computational statistics \& data analysis}, 52(12), 5277--5286.

\bibitem[59]{raman2009bayesian}
Raman, Sudhir and Fuchs, Thomas J and Wild, Peter J and Dahl, Edgar and Roth, Volker (2009). The Bayesian group-lasso for analyzing contingency tables.
{\em Proceedings of the 26th Annual International Conference on Machine Learning},  881--888.

\bibitem[60]{xu2015bayesian}
Xu, Xiaofan and Ghosh, Malay (2015). Bayesian variable selection and estimation for group lasso.
{\em Bayesian Analysis}, 10(4), 909--936.

\bibitem[61]{wei2010consistent}
Wei, Fengrong and Huang, Jian (2010). Consistent group selection in high-dimensional linear regression.
{\em Bernoulli: official journal of the Bernoulli Society for Mathematical Statistics and Probability}, 16(4), 1369.


\bibitem[62]{tibshirani1996regression}
Tibshirani, Robert (1996). Regression shrinkage and selection via the lasso.
{\em Journal of the Royal Statistical Society: Series B (Methodological)}, 58(1), 267--288.

\bibitem[63]{yuan2006model}
Yuan, Ming and Lin, Yi (2006). Model selection and estimation in regression with grouped variables.
{\em Journal of the Royal Statistical Society: Series B (Methodological)}, 68(1), 49--67.

\bibitem[64]{castillo2015bayesian}
Castillo, Isma{\"e}l and Schmidt-Hieber, Johannes and Van der Vaart, Aad (2015). Bayesian linear regression with sparse priors.
{\em The Annals of Statistics}, 43(5), 1986--2018.


\bibitem[65]{scott2010bayes}
Scott, James G and Berger, James O (2010). Bayes and empirical-Bayes multiplicity adjustment in the variable-selection problem.
{\em The Annals of Statistics}, 2587--2619.


\bibitem[66]{george1997approaches}
George, Edward I and McCulloch, Robert E(1997). \em{Statistica sinica}, {339--373}.



\bibitem[67]{huang2012selective}
Huang, Jian and Breheny, Patrick and Ma, Shuangge (2012). A selective review of group selection in high-dimensional models. \em{Statistical science: a review journal of the Institute of Mathematical Statistics}, {27},
{4}.


\bibitem[68]{jacob2009group}
Jacob, Laurent and Obozinski, Guillaume and Vert, Jean-Philippe (2009).
Group lasso with overlap and graph lasso. \em{Proceedings of the 26th annual international conference on machine learning}, {433--440},
 

\bibitem[69]{liu2010fast}
Liu, Jun and Ye, Jieping (2010).
\em{arXiv preprint arXiv:1009.0306}.


\bibitem[70]{zhao2009composite}
Zhao, Peng and Rocha, Guilherme and Yu, Bin (2009).
The composite absolute penalties family for grouped and hierarchical variable selection.

\bibitem[71]{zellner1962efficient}
Zellner, Arnold (1962).
An efficient method of estimating seemingly unrelated regressions and tests for aggregation bias.
\em{Journal of the American statistical Association}, {348--368},
  

\bibitem[72]{caruana1997multitask}
Caruana, Rich (1997).
Multitask learning.
\em{Machine learning}, {41--75},
 
\bibitem[73]{argyriou2008convex}
Argyriou, Andreas and Evgeniou, Theodoros and Pontil, Massimiliano (2008).
\em{Machine learning}, {243--272},
  

\bibitem[74]{gelman2006prior}
Gelman, Andrew (2006).
Prior distributions for variance parameters in hierarchical models (comment on article by Browne and Draper).


\bibitem[75]{zellner1986assessing}
Zellner, Arnold (1986).
On assessing prior distributions and Bayesian regression analysis with g-prior distributions.
\em{Bayesian inference and decision techniques}.


\bibitem[76]{zhang2010nearly}
Zhang, Cun-Hui (2010).
Nearly unbiased variable selection under minimax concave penalty.

\bibitem[77]{mukhopadhyay2015consistency}
Mukhopadhyay, Minerva and Samanta, Tapas and Chakrabarti, Arijit (2015).
On consistency and optimality of Bayesian variable selection based on g-prior in normal linear regression models.
\em{Annals of the Institute of Statistical Mathematics}, {963--997}.


 \bibitem[78]{bayarri2012criteria}
Bayarri, Maria J and Berger, James O and Forte, Anabel and Garc{\'\i}a-Donato, Gonzalo (2012).
Criteria for Bayesian model choice with application to variable selection.

\bibitem[79]{paul2022posterior}
Paul, Sayantan and Chakrabarti, Arijit (2023).
Posterior Contraction rate and Asymptotic Bayes Optimality for one group global-local shrinkage priors in sparse normal means problem
\em{arXiv preprint arXiv:2211.02472v2}.

\bibitem[80]{zou2009adaptive}
Zou, Hui and Zhang, Hao Helen (2009).
On the adaptive elastic-net with a diverging number of parameters.
\em{Annals of statistics}, {37}, {4}, {1733}.

\bibitem[81]{zhao2006model}
Zhao, Peng and Yu, Bin (2006).
On model selection consistency of Lasso
\em{The Journal of Machine Learning Research}, {2541--2563}.
  

\bibitem[82]{zhang2016oracle}
Zhang, Caiya and Xiang, Yanbiao (2016).
On the oracle property of adaptive group lasso in high-dimensional linear models.
\em{Statistical Papers}, {249--265}.
  
\bibitem[83]{wang2019adaptive}
Wang, Mingqiu and Tian, Guo-Liang (2019).
Adaptive group Lasso for high-dimensional generalized linear models.
\em{Statistical Papers},
  {1469--1486}.
  
\bibitem[84]{maruyama2011fully}
Maruyama, Yuzo and George, Edward I (2011).
Fully Bayes factors with a generalized g-prior.

\bibitem[85]{liang2008mixtures}
Liang, Feng and Paulo, Rui and Molina, German and Clyde, Merlise A and Berger, Jim O (2008).
Mixtures of g priors for Bayesian variable selection.
\em{Journal of the American Statistical Association}, {410--423},
  
\bibitem[86]{som2014block}
Som, Agniva and Hans, Christopher M and MacEachern, Steven N (2014).
Block hyper-g priors in Bayesian regression.
\em{arXiv preprint arXiv:1406.6419}.
  
\bibitem[87]{boss2023group}
Boss, Jonathan and Datta, Jyotishka and Wang, Xin and Park, Sung Kyun and Kang, Jian and Mukherjee, Bhramar (2023).
Group inverse-gamma gamma shrinkage for sparse linear models with block-correlated regressors.
\em{Bayesian Analysis},
  year={2023}.

 \bibitem[88]{hosmer1989applied}
 Hosmer, D.W. and Lemeshow, Stanley (1989). Applied logistic regression. \em{John Wiley \& Sons}, year={1989}. 

 
 \end{thebibliography}
\end{document}